%% file: Kume-Yamada-Comput_Optim_Appl24.tex
\documentclass[sn-mathphys-num,pdflatex]{sn-jnl}%

\input{preamble}

\newcommand{\cost}{F}
\newcommand{\param}{\bm{\varphi}}
\usepackage{threeparttable}
\usepackage{adjustbox}

\usepackage{eso-pic}

\begin{document}

\title{\Large A Variable Smoothing for Weakly Convex Composite Minimization with Manifold Constraint via Parametrization}

\author*[1]{\fnm{Keita} \sur{Kume}}\email{kume@sp.ict.e.titech.ac.jp}

\author[1]{\fnm{Isao} \sur{Yamada}}\email{isao@sp.ict.e.titech.ac.jp}

\affil*[1]{\centering\orgdiv{Dept. of Information and Communications Engineering},\\ \orgname{Institute of Science Tokyo}, \orgaddress{\street{2-12-1}, \city{Ookayama}, \postcode{152-8550}, \state{Tokyo}, \country{Japan}}}

\abstract{
In this paper, we address a manifold constrained nonsmooth optimization problem involving the composition of a weakly convex function and a smooth mapping under the availability of a parametrization of the manifold.
To find a stationary point of the target problem, we propose a variable smoothing-type algorithm by combining the ideas of (i) translating the constrained problem into a Euclidean optimization problem with a parametrization of the constraint set; (ii) exploiting
a sequence of smooth approximations of the objective constructed using the Moreau envelope.
The proposed algorithm produces a vector sequence by the gradient descent update of a smooth approximation of the objective.
In a case where the proximity operator of the weakly convex function is available, the proposed algorithm does not require any iterative solver for subproblems therein.
By leveraging tools in the variational analysis, we show
the so-called {\em gradient consistency property}, which is a key ingredient for smoothing-type algorithms, of the smooth approximations used in this paper.
Based on the gradient consistency property, we establish a subsequential convergence guarantee of the proposed algorithm to a stationary point of the parameterized problem and clarify conditions, such as a submersion condition on the parametrization, under which the corresponding point is stationary for the original manifold constrained problem.
Numerical experiments demonstrate the efficacy of the proposed algorithm.
}
\keywords{manifold constrained nonsmooth optimization, weakly convex composite, variable smoothing, gradient consistency, Moreau envelope, parametrization}
\pacs[MSC Classification]{49J52,  49K99}

\AddToShipoutPictureFG*{%
  \AtPageUpperLeft{%
    \put(0,-\LenToUnit{15mm}){%
      \makebox[\paperwidth][c]{%
        \shortstack[c]{%
          This paper has been published in
            {\em Computational Optimization and Applications} (2026).
          \\[2pt]
          The published version is available at
          \url{https://doi.org/10.1007/s10589-026-00830-z}.
        }
      }%
    }%
  }%
}

\maketitle

\mathtoolsset{showonlyrefs=true}
\noeqref{eq:gradient_liminf,eq:regular_subdifferential,eq:general_subdifferential,eq:general_normal}

\section{Introduction} \label{sec:introduction}
In this paper, we consider the following nonsmooth optimization problem with a manifold constraint.
\begin{problem}\label{problem:constrained}
Let
$\mathcal{X}$
and
$\mathcal{Z}$
be Euclidean spaces, i.e.,
finite-dimensional real Hilbert spaces, and
$(\emptyset \neq)C \subset \mathcal{X}$
a connected closed embedded submanifold of
$\mathcal{X}$ (see Example~\ref{ex:normal}).
Then,
\begin{equation}
  \mathrm{find} \ \bm{x}^{\star} \in \argmin_{\bm{x}\in C} (\underbrace{h+g\circ \mathfrak{S}}_{\eqqcolon \cost})(\bm{x}) =
  \argmin_{\bm{x}\in \mathcal{X}} (h + g \circ \mathfrak{S}+\iota_{C})(\bm{x}),
\end{equation}
where
$\iota_{C}$
is the indicator function of
$C$
defined by
\begin{equation}
  \iota_{C}:\mathcal{X}\to \exR:\bm{x}\mapsto
  \left\{\begin{array}{cl} 0,  & \mathrm{if}\ \bm{x}\in C;    \\
             +\infty, & \mathrm{if}\ \bm{x}\notin C,
  \end{array}\right.
\end{equation}
and
$h$,
$\mathfrak{S}$,
and
$g$
satisfy the following:
\begin{enumerate}[label=(\roman*)]
  \item \label{enum:L_smooth}
        $h:\mathcal{X} \to \mathbb{R}$
        is continuously differentiable and its gradient
        $\nabla h:\mathcal{X}\to\mathcal{X}$
        is
        $L_{\nabla h} (> 0)$-Lipschitz continuous
        over
        $C$,
        i.e.,
        \begin{equation}
          (\bm{x}_{1},\bm{x}_{2}\in C ) \quad \norm{\nabla h(\bm{x}_{1}) - \nabla h(\bm{x}_{2})} \leq L_{\nabla h}\norm{\bm{x}_{1}-\bm{x}_{2}};
        \end{equation}
  \item
        $\mathfrak{S}: \mathcal{X} \to \mathcal{Z}$
        is a continuously differentiable (possibly nonlinear) mapping;
  \item \label{enum:problem:constrained:g}
        $g:\mathcal{Z} \to \mathbb{R}$
        is a possibly nonsmooth,
        $L_{g}(>0)$-Lipschitz continuous,
        and
        $\eta(>0)$-weakly convex function,
        i.e.,
        $g+\frac{\eta}{2}\norm{\cdot}^{2}$
        is convex over
        $\mathcal{Z}$.
        Moreover,
        $g$
        is {\em prox-friendly}, i.e., the closed-form expression of the {\em proximity operator}
        $\prox{\mu g}$
        in~\eqref{eq:prox} of
        $g$
        is available as a computable tool for every
        $\mu \in (0,\eta^{-1})$;
  \item \label{enum:cost_bounded}
        $\cost$
        is bounded below over
        $C$,
        i.e.,
        $\inf\{\cost(\bm{x})\mid \bm{x} \in C \} > -\infty$.
\end{enumerate}
\end{problem}

Problems of the form of Problem~\ref{problem:constrained} serve as key models, e.g., for finding sparse solutions, in a wide range of signal processing and machine learning.
Such applications include, e.g., sparse Principal Component Analysis (PCA)~\cite{Benidis-Sun-Babu-Palomar16}, sparse variable PCA~\cite{Ulfarsson-Solo08}, robust sparse PCA~\cite{Breloy-Kumar-Sun-Palomar21}, orthogonal dictionary learning/robust subspace recovering~\cite{Chen-Deng-Ma-So21}, and sparse spectral clustering~\cite{Lu-Yan-Lin16}.

One of simple cases of Problem~\ref{problem:constrained} is a convex case where
$h$
and
$g$
are convex,
$\mathfrak{S}$
is linear,
and
$C \coloneqq \mathcal{X}$.
For the convex case, the so-called {\em proximal splitting algorithms} have been extensively applied as computationally reliable algorithms (see, e.g.,~\cite{Bauschke-Combettes17,Condat-Kitahara-Contreras-Hirabayashi23}).
The algorithms consist of computable tools including the gradient of
$h$
and
the  proximity operator 
of
$g$
(see~\eqref{eq:prox}).
Although evaluating the proximity operator generally requires solving a convex optimization problem, this operator has closed-form expressions for many commonly-used functions, e.g.,
$\ell_{1}$-norm and nuclear norm (see~\cite{Bauschke-Combettes17,Chierchia-Chouzenoux-Combettes-Pesquet}).

In Problem~\ref{problem:constrained}, restricting
$g$
to be convex is not always sufficient~\cite{Chen-Gu14,Selesnick17,Yang-Chen-Ma-Chen-Gu-So19,Abe-Yamagishi-Yamada20,Yata-Yamagishi-Yamada22,Kuroda24} in order to enhance estimation performance especially for sparsity-aware applications.
For example,
$\ell_{1}$-norm has been widely used as a best convex approximation of the naive sparsity promoting function, i.e.,
$\ell_{0}$-pseudonorm, which counts the number of nonzero entries of a vector.
However, the use of
$\ell_{1}$-norm
as
$g$
leads to some underestimation of the desired sparse target because
$\ell_{1}$-norm is a coercive function (see, e.g.,~\cite{Zhang10}), i.e.,
$\ell_{1}$-norm excessively penalizes large magnitude of target.
To overcome the underestimation effect caused mainly by the sparse promoting convex function, the utilization of  weakly convex functions has been attracting great attention~\cite{Chen-Gu14,Yang-Chen-Ma-Chen-Gu-So19}.
Recent advances in proximal splitting techniques~\cite{Selesnick17,Abe-Yamagishi-Yamada20,Yata-Yamagishi-Yamada22,Kuroda24} encourage us to find a global minimizer of Problem~\ref{problem:constrained} under the convexity of
$h+g\circ \mathfrak{S}$
with a linear operator
$\mathfrak{S}$
and even with a weakly convex
$g$.

\begin{table*}
\begin{adjustbox}{center}
  \begin{threeparttable}
  {
\fontsize{6.5}{6}\selectfont
  \caption{Descriptive summary of existing and proposed algorithms for Problem~\ref{problem:constrained}, where
    ``cvx'', ``w-cvx'', ''lin'',
    ``$\bf C^{q}$'',
    ``$\Id$'',
  ``man''
  and
  ``$\St(p,N)$''
    denote respectively convex function, weakly convex function, linear operator, $q$-times continuously differentiable mapping, the identity operator, compact manifold, and the Stiefel manifold.
    In ``Convergence result'',
    ``Subsequential convergence'' means the cited work explicitly proves that every cluster point of the generated sequence, or of a specified subsequence, is a stationary point for the corresponding instance of Problem~\ref{problem:constrained}.
    The complexity results in Table~\ref{table:nonsmooth} are descriptive and are tied to the stationarity measure used in each cited work (see also footnotes in this table),
    where
    $\mathcal{O}$ stands for Landau's big O-notation.
  }\label{table:nonsmooth}
  \centering
    {
      \begin{tabular}[c]{cc|ccc|l}
        \toprule
        \multicolumn{1}{c}{Method}                                       &
        \multicolumn{1}{c|}{Ref.}                                        &
        \multicolumn{1}{c}{$g$}                                          &
        \multicolumn{1}{c}{$\mathfrak{S}$}                               &
        \multicolumn{1}{c|}{$C (\subset \mathcal{X})$}                   &
        {Convergence result}
        \\
        \midrule
                                                                          \multirow{1}{*}{ManPG}               &
                                                                          \multirow{1}{*}{\cite[Thm. 5.5]{Chen-Ma-Man-Anthony-Zhang20}}    & \multirow{1}{*}{cvx}                 & \multirow{1}{*}{$\Id$}                                        & \multirow{1}{*}{\textbf{man}} & Subsequential convergence \&                                                                                                                                                                                                                                                                                                                                                                                                                                     \\
        (double-loop)                                      &                                                               &                                         &                                                                                                          &                                                        &
                                                                                                         $\mathcal{O}(\epsilon^{-2})$ iterations for
        $\|{\bm{v}_{n}^{\rm (a)}}\| \leq \epsilon$\tnote{(a)}.                                                                                                                                                                                                                                                                                                                                                       \\
        \cmidrule{1-6}
                                                                         \multirow{3}{*}{PLM}                 &
                                                                         \multirow{1}{*}{\cite[Thm. 6.4, p.535]{Drusvyatskiy-Paquette19}} & \multirow{1}{*}{cvx\tnote{(b)}}      & \multirow{1}{*}{$\bf{C^{1}}$}                                 & \multirow{1}{*}{$\mathcal{X}$}          & $\mathcal{O}(\epsilon^{-3}\log(\epsilon^{-1}))$ oracles for $\|{\bm{v}_{n}^{\rm (a)}}\|\leq \epsilon$\tnote{(a)} \quad \&                                                                                                                                                      \\
        \multirow{3}{*}{(double-loop)}                                      &                                                               &                                         &                                                                                                          &                                                        & $\min_{k=1,2,\ldots,n}\|{\bm{v}_{k}^{\rm (a)}}\| \leq \chi n^{-1/2}$ with some $\chi>0$.      \\ \cmidrule{2-6}
                                                                            & \cite[Thm. 5.4]{Lewis-Wright16}                               & \textbf{w-cvx}\tnote{(c)}               & $\bf{C^{2}}$                                                                                             & $\mathcal{X}$                                          & Subsequential convergence.                                                                 \\\cmidrule{2-6}
                                                                            & \multirow{1}{*}{\cite[Thm. 2.1]{Wang-Liu-Chen-Ma-Xue-Zhao22}} & \multirow{1}{*}{cvx}                    & \multirow{1}{*}{$\bf{C^{1}}$}                                                                            & \multirow{1}{*}{\textbf{man}}                & Subsequential convergence \&  \\
                                                                                                               &                                                               &                                         &                                                                                                          &                                                        & $\mathcal{O}(\epsilon^{-2})$ iterations for $\|{\bm{v}_{n}^{\rm (a)}}\| \leq \epsilon$.\tnote{(a)}
        \\
        \cmidrule{1-6}
                                                                         \multirow{1}{*}{ALM}
          & \cite[Thm. 4.1]{Zhou-Bao-Ding-Zhu23} & cvx                                                           & $\bf{C^{1}}$                            & {\bf man}\tnote{(d)}                                                                                    & Subsequential convergence.                                                                                                                          \\\cmidrule{2-6}
        (double-loop)                                      & \cite[Thm. 3.1]{Xu-Jian-Liu-So25}                             & cvx                                     & $\bf{C^{1}}$                                                                                             & \bf man                                                & $\mathcal{O}(\epsilon^{-3})$ oracles for $\epsilon$-stationarity.\tnote{(e)}              \\
        \midrule
                                                                         \multirow{1}{*}{RSub}                &
                                                                         \cite[Thm. 2]{Li-Chen-Deng-Qu-Zhu-Man21}                       & \textbf{w-cvx}                       & $\Id$                                                         & \textbf{man}                  & $\mathcal{O}(\epsilon^{-4})$ iterations (oracles) for $\|{\bm{v}_{n}^{\rm (g)}}\|\leq \epsilon$\tnote{(g)}                                                                      \quad \&                                                                                       \\ 
        (single-loop)                                      &                                                               &                                         &                                                                                                          &                                                        & $\min_{k=1,2,\ldots,n}\|{\bm{v}_{k}^{\rm (g)}}\| \leq \chi n^{-1/4}$ with some $\chi>0$.       \\\cmidrule{1-6}
                                                                         \multirow{1}{*}{OADMM}                &
        \cite[Thm. 4.13]{Yuan24}                       & \textbf{w-cvx}                       & lin                                                         & $\St(p,N)$                  & $\mathcal{O}(\epsilon^{-3})$ iterations (oracles) for $\epsilon$-stationarity.\tnote{(e)}                                                                                                                                                             \\
        (single-loop)                                      &                                                               &                                         &                                                                                                          &                                                        &        \\ \cmidrule{1-6}
                                                                         \multirow{1}{*}{RADA}                &
        \cite[Thm. 4]{Xu-Jiang-Liu-So26}                       & cvx                       & $\bf{C^{1}}$                                                         & \textbf{man}                  & $\mathcal{O}(\epsilon^{-3})$ iterations for $\epsilon$-stationarity.\tnote{(f)}                                                                                                                                                             \\
        (single-loop)                                      &                                                               &                                         &                                                                                                          &                                                        &        \\ \cmidrule{1-6}
                                                                         \multirow{6}{*}{VSmooth}
                                                         & \cite[Thm. 4.2]{Bohm-Wright21}                 & \textbf{w-cvx}                                                & lin                                     & $\mathcal{X}$                                                                                            & 
                                                         $\mathcal{O}(\epsilon^{-3})$ iterations (oracles) for $\exists\bm{z}\in\mathcal{Z}$ s.t. \\
                                                         \multirow{6}{*}{(single-loop)}& & & & & 
                                                         $\quad\|{\bm{z} - \mathfrak{S}(\bm{x}_{n})}\| \leq \epsilon,\ \dist(-\nabla h(\bm{x}_{n}),\mathfrak{S}^{*}(\Lsubdiff g(\bm{z})))\leq \epsilon.$ \\\cmidrule{2-6}
                                                                                       & \cite[Thm. 4.1, 4.2]{Beck-Rosset23}                                          & cvx                                     & lin                                                                                                      & \textbf{man}                                 &  Subsequential convergence \&
                                                                                                                           \\
                                                                                       & & & & & $\mathcal{O}(\epsilon^{-3})$ iterations (oracles) for $\epsilon$-stationarity.\tnote{(e)} \\\cmidrule{2-6}
                                                                                       & \cite[Prop. 3]{Peng-Wu-Hu-Deng23}                                      & \textbf{w-cvx}                          & lin                                                                                                      & \textbf{man}                                 & 
                                                                                       $\mathcal{O}(\epsilon^{-3})$ iterations (oracles) for $\epsilon$-stationarity.\tnote{(e)}
                                                                                       \\\cmidrule{2-6}
                                                                                                               & Alg.~\ref{alg:vsmooth} (this paper)                                       & \textbf{w-cvx}         & $\bf{C^{1}}$                                                                            & \textbf{man}                          & 

                                                                                                               Subsequential convergence \&
                                                                                        \\
                                                                                                                                                                                           & Thm.~\ref{theorem:convergence_extension}, Cor.~\ref{corollary:complexity}                                                   &                    \multicolumn{3}{c|}{($\param$ in Asm.~\ref{assumption:parametrization} is assumed }                                                                                                                                                                                       & $\mathcal{O}(\epsilon^{-3})$\tnote{(h)} \ \ oracles for $\exists \bm{z} \in \mathcal{Z}$ s.t.$\|{\bm{z} - \mathfrak{S}\circ \param(\bm{y}_{n})}\| \leq \epsilon$ and

                                                                                        \\
                                                                                                                                                                                                                                                                                                                                                                                                                                                                                                                                                                                       & (see~\eqref{eq:approximate_stationary} in Rem.~\ref{remark:approximate_stationarity}) & \multicolumn{3}{c|}{ to be, e.g., submersion)} & 
                                                                                                               $\quad\dist(-\nabla (h\circ\param)(\bm{y}_{n}),(\mathrm{D}(\mathfrak{S}\circ \param)(\bm{y}_{n}))^{*}\left[\Lsubdiff g(\bm{z})\right])\leq \epsilon$. \\
        \bottomrule
      \end{tabular}
      \begin{tablenotes}
        \item
        In this table,
        $\bm{x}_{n} \in C$
        denotes the current iterate of algorithms, and
        $t > 0$
        is a certain constant.
        \item[(a)]
          In~\cite{Chen-Ma-Man-Anthony-Zhang20,Drusvyatskiy-Paquette19,Wang-Liu-Chen-Ma-Xue-Zhao22},
          $\bm{v}_{n}^{\rm (a)}\coloneqq t^{-1}\bm{d}_{n}^{\rm (a)}$
        with
        $\bm{d}_{n}^{\rm (a)}\coloneqq \mathrm{argmin}_{\bm{d}\in T_{C}(\bm{x}_{n})} (\inprod{\nabla h(\bm{x}_{n})}{\bm{d}} + g(\mathfrak{S}(\bm{x}_{n}) + \mathrm{D}\mathfrak{S}(\bm{x}_{n})[\bm{d}]) +\frac{1}{2t}\norm{\bm{d}}^{2} )$,
        where the tangent cone
        (see~\eqref{eq:manifold_tangent})
        is expressed as
        $T_{C}(\bm{x}_{n})=\mathcal{X}$
        if
        $C = \mathcal{X}$
        for PLM~\cite{Drusvyatskiy-Paquette19},
        and
        $\mathfrak{S}$
        is
        $\Id$
        for ManPG~\cite{Chen-Ma-Man-Anthony-Zhang20}.
        \item[(b)]
        PLM in~\cite{Drusvyatskiy-Paquette19} also handles
        the sum of
        $h+g\circ\mathfrak{S}$
        and
        an additional proper lower semicontinuous convex function as a cost function.
        \item[(c)]
          PLM in~\cite{Lewis-Wright16} is not restricted to weak convexity of
          $g$;
          \cite[Thm. 5.4]{Lewis-Wright16} assumes subdifferential  regularity of
          $g$
          at the cluster point.
      \item[(d)]
        ALM in~\cite{Zhou-Bao-Ding-Zhu23} also handles additional inequality constraints.
        \item[(e)]
        In~\cite{Beck-Rosset23,Peng-Wu-Hu-Deng23,Xu-Jian-Liu-So25},
        $\bm{x}_{n} \in C$
        is said to be
        $\epsilon$-stationary point if
        there exist
        $\bm{z} \in \mathcal{Z}$
        and
        $\bm{v} \in \Lsubdiff g(\bm{z})$
        such that
        $\max\{\|P_{T_{C}(\bm{x}_{n})}(\nabla h(\bm{x}_{n}) + (\mathrm{D}\mathfrak{S}(\bm{x}_{n}))^{*}[\bm{v}])\|, \norm{\mathfrak{S}(\bm{x}_{n})-\bm{z}}\} \leq \epsilon$,
        where
        $\mathfrak{S}$
        is a linear operator for~\cite{Beck-Rosset23,Peng-Wu-Hu-Deng23}.
        A related
        $\epsilon$-stationarity notion
        is used in \cite[Def. 2.7]{Yuan24}.
      \item[(f)]
        RADA in~\cite{Xu-Jiang-Liu-So26} considers nonconvex-linear minimax problems on manifolds, which covers Problem~\ref{problem:constrained} in a case where
        $g$
        is convex, and two notions of $\epsilon$-stationarity are discussed; 
        $\epsilon$-ROS in~\cite[Def. 1]{Xu-Jiang-Liu-So26} coincides with
        the $\epsilon$-stationarity in footnote\tnote{(e)}.
        \item[(g)]
          In~\cite{Li-Chen-Deng-Qu-Zhu-Man21},
          $\bm{v}_{n}^{\rm (g)} \coloneqq t^{-1}( \bm{d}_{n}^{\rm (g)} - \bm{x}_{n} )$
        with
        $\bm{d}_{n}^{\rm (g)} \in \mathrm{argmin}_{\bm{d}\in C}(g(\bm{d}) + \frac{1}{2t}\norm{\bm{d}-\bm{x}_{n}}^{2})$,
        where
        $h\equiv 0$
        and
        $\mathfrak{S}\coloneqq \Id$.
      \item[(h)]
        Cor.~\ref{corollary:complexity} gives the oracle complexities
        $\mathcal{O}(\epsilon^{-3})$
        and
        $\mathcal{O}(\epsilon^{-3}\log(\epsilon^{-1}))$
        of Alg.~\ref{alg:vsmooth} with, respectively, diminishing stepsizes and stepsizes obtained by a backtracking algorithm (see Alg.~\ref{alg:backtracking}), where
        $(\bm{y}_{n})_{n=1}^{\infty}\subset \mathcal{Y}$
        is a generated sequence of Alg.~\ref{alg:vsmooth} for Problem~\ref{problem:origin} with a parametrization
        $\param:\mathcal{Y}\to C$.
      \end{tablenotes}
    }
}
  \end{threeparttable}
\end{adjustbox}
\end{table*}

Beyond the convex cases, nonconvex optimization problems with nonconvex constraint have a great potential for flexible formulation to align with real-world tasks.
Indeed, allowing a nonlinear smooth mapping
$\mathfrak{S}$
in the composite term
$g\circ \mathfrak{S}$
broadens the modeling scope of Problem~\ref{problem:constrained}.
Such composite formulations arise in, e.g., 
robust phase retrieval~\cite{Duchi-Ruan18}, nonnegative matrix factorization~\cite{Gillis17}, robust matrix recovery~\cite{Charisopoulos-Chen-Davis-Diaz-Ding-Drusvyatskiy21}, and robust blind deconvolution~\cite{Charisopoulos-Davis-Dias-Drusvyatskiy20}.
Moreover, manifold constrained optimization has continued to attract considerable attention mainly by its remarkable expressive ability of the target conditions expected to be achieved by a solution (see, e.g.,~\cite{Absil-Mahony-Sepulchre08,Sato21,Boumal23,Absil-Hosseini19}).

Recently, unconstrained nonsmooth optimization algorithms, e.g.,~\cite{Bot-Hendrich15,Bohm-Wright21,Bauschke-Combettes17,Lewis-Wright16,Drusvyatskiy-Paquette19}, have been extended to manifold constrained nonsmooth optimization algorithms~\cite{Chen-Ma-Man-Anthony-Zhang20,Wang-Liu-Chen-Ma-Xue-Zhao22,Zhou-Bao-Ding-Zhu23,Deng-Peng23,Xu-Jian-Liu-So25,Li-Chen-Deng-Qu-Zhu-Man21,Beck-Rosset23,Peng-Wu-Hu-Deng23,Yuan24,Deng-Hu-Wu-Wen25,Xu-Jiang-Liu-So26} (see Table~\ref{table:nonsmooth}).
For Problem~\ref{problem:constrained}, these extended algorithms can be categorized as (i) double-loop algorithms that require an iterative solver to solve a certain subproblem at each iteration
[e.g.,
a manifold proximal gradient method (ManPG)~\cite{Chen-Ma-Man-Anthony-Zhang20}, a proximal linear method (PLM)~\cite{Wang-Liu-Chen-Ma-Xue-Zhao22}, and augmented Lagrangian methods (ALM)~\cite{Deng-Peng23,Zhou-Bao-Ding-Zhu23,Xu-Jian-Liu-So25,Deng-Hu-Wu-Wen25}];
(ii) single-loop algorithms without requiring such an iterative solver at each iteration
[e.g., a Riemannian subgradient method (RSub)~\cite{Li-Chen-Deng-Qu-Zhu-Man21}, an alternating direction method of multipliers with orthogonality constraints (OADMM)~\cite{Yuan24}, Riemannian alternating descent-ascent algorithm (RADA)~\cite{Xu-Jiang-Liu-So26} and variable smoothing algorithms (VSmooth)~\cite{Beck-Rosset23,Peng-Wu-Hu-Deng23}].

Double-loop algorithms~\cite{Wang-Liu-Chen-Ma-Xue-Zhao22,Zhou-Bao-Ding-Zhu23,Xu-Jian-Liu-So25} can deal with a nonlinear continuously differentiable mapping
$\mathfrak{S}$
but with a convex function
$g$
in Problem~\ref{problem:constrained},
whereas~\cite{Deng-Peng23,Deng-Hu-Wu-Wen25} and~\cite{Chen-Ma-Man-Anthony-Zhang20} assume
the linearity of
$\mathfrak{S}$
and
$\mathfrak{S}=\Id$
respectively.
A subsequential convergence guarantee to a stationary (or KKT) point has been shown for~\cite{Chen-Ma-Man-Anthony-Zhang20,Wang-Liu-Chen-Ma-Xue-Zhao22,Zhou-Bao-Ding-Zhu23,Deng-Peng23}, where
a subsequential convergence guarantee in this paper means every cluster point of the generated sequence, or of a specified subsequence, is a stationary point.
Moreover, the oracle (resp. iteration) complexity
i.e., the number of oracle calls (resp. iterations),
for finding an $\epsilon$-approximate stationary point has been shown for~\cite{Xu-Jian-Liu-So25} (resp. \cite{Chen-Ma-Man-Anthony-Zhang20,Wang-Liu-Chen-Ma-Xue-Zhao22}).
However, each update of double-loop algorithms heavily relies on iterative solvers for solving subproblems.
Indeed,
convergence analyses of double-loop algorithms
have been established typically with an exact (or approximate) solution to the subproblem, although such an exact solution is not available in general within a finite number of iterations of an iterative solver.
Moreover, an iterative solver for subproblems sometimes takes a lot of time to obtain even an approximate solution to a subproblem (as will be shown in numerical experiments in Section~\ref{sec:SPCA}).

Single-loop algorithms have also been attracting great attention, e.g.,~\cite{Li-Chen-Deng-Qu-Zhu-Man21,Beck-Rosset23,Peng-Wu-Hu-Deng23,Yuan24,Xu-Jiang-Liu-So26}, for Problem~\ref{problem:constrained}
because such algorithms do not require solving any subproblem at every iteration.
Under the linearity of
$\mathfrak{S}$,
the oracle complexity for finding an $\epsilon$-approximate stationary point has been shown for~\cite{Li-Chen-Deng-Qu-Zhu-Man21,Beck-Rosset23,Peng-Wu-Hu-Deng23,Yuan24},
and a subsequential convergence guarantee to a stationary point has been shown explicitly for~\cite{Beck-Rosset23}.
A recent work in~\cite{Xu-Jiang-Liu-So26} proposes a single-loop algorithm, called RADA,  for nonconvex-linear minimax problems on manifolds.
Problem~\ref{problem:constrained} with a nonlinear mapping
$\mathfrak{S}$
and a convex function
$g$
can be reformulated into the target problem in~\cite{Xu-Jiang-Liu-So26} through the Fenchel conjugate of the convex function
$g$
(see~\cite[Eq. (7)]{Xu-Jiang-Liu-So26}), and 
an iteration complexity
$\mathcal{O}(\epsilon^{-3})$
for finding an $\epsilon$-stationary point is established in~\cite[Thm. 4]{Xu-Jiang-Liu-So26}.
These results motivate the study of a single-loop algorithm for Problem~\ref{problem:constrained}
that can handle both a nonlinear continuously differentiable mapping
$\mathfrak{S}$
and a weakly convex function
$g$.

The goal of this paper is to design a single-loop algorithm with a subsequential convergence guarantee to a stationary point of Problem~\ref{problem:constrained} with a nonlinear continuously differentiable mapping
$\mathfrak{S}$
and a weakly convex
$g$.
For this goal, under the availability of 
a {\em parametrization}
$\param:\mathcal{Y}\to C$
of
$C$ in Assumption~\ref{assumption:parametrization} below, we reformulate Problem~\ref{problem:constrained} with a manifold constraint into an unconstrained Problem~\ref{problem:origin}, where Assumption~\ref{assumption:parametrization} is satisfied by a broad class of manifolds encountered in applications (see Example~\ref{example:parametrization}).
\begin{problem}\label{problem:origin}
Consider
$\cost=h+g\circ\mathfrak{S}$
and
$C \subset \mathcal{X}$
in Problem~\ref{problem:constrained} under Assumption~\ref{assumption:parametrization} below.
Then, with a parametrization
$\param$
in Assumption~\ref{assumption:parametrization},
\begin{equation}
  \mathrm{find} \ \bm{y}^{\star} \in \argmin_{\bm{y}\in \mathcal{Y}} (h+g\circ \mathfrak{S})\circ \param (\bm{y}) (= \cost \circ \param (\bm{y})).
  \label{eq:weakly-composite}
\end{equation}
\end{problem}
\begin{assumption}[Parametrization] \label{assumption:parametrization}
  The manifold
  $C$
  in Problem~\ref{problem:constrained}
  can be parameterized by a continuously differentiable mapping
  $\param :\mathcal{Y}\to C(\subset \mathcal{X})$
  defined over a Euclidean space
  $\mathcal{Y}$,
  i.e.,
  $\param (\mathcal{Y})\coloneqq  \{\param (\bm{y}) \in \mathcal{X} \mid \bm{y} \in \mathcal{Y}\} = C$.
  We call
  $\param$
  a parametrization of
  $C$.
\end{assumption}
\begin{example}[On parametrization; see Section~\ref{sec:experiment} for other examples] \label{example:parametrization}
  \mbox{}
  \begin{enumerate}[label=(\alph*)]
    \item
          Let
          $C$
          be a linear subspace of
          $\mathcal{X}$.
          Consider a linear operator
          $L:\mathcal{Y} \to \mathcal{X}$
          satisfying
          $L(\mathcal{Y})\coloneqq\{L\bm{y}\in \mathcal{X}\mid \bm{y}\in\mathcal{Y}\} = C$.
          Then,
          $\param\coloneqq L$
          is a parametrization of
          $C$.
          In particular,
          $\param\coloneqq \Id$
          is a parametrization of
          $C$
          if
          $C\coloneqq \mathcal{X}$.
    \item
          For
          $\mathcal{X} \coloneqq \mathbb{R}^{2}$,
          let
          $C\coloneqq \{\bm{x}\in \mathbb{R}^{2}\mid \norm{\bm{x}}=1\}$
          be the unit circle.
          Then,
          $\param:\mathbb{R}\to \mathbb{R}^{2}:\theta\mapsto [\cos(\theta)\ \sin(\theta)]^{\TT}$
          is a parametrization of
          $C$.
    \item \label{enum:ex:parameterization:genearal}
          Consider
          $C\subset \mathcal{X}$
          in Problem~\ref{problem:constrained}, i.e.,
          $C$
          is an embedded submanifold of
          $\mathcal{X}$,
          endowed with the {\em Riemannian metric} induced from
          $\mathcal{X}$,
          such that
          $C$
          is connected and closed in
          $\mathcal{X}$.
          Let
          $T_{C}(\bm{x})\subset \mathcal{X}$
          be a {\em tangent (vector) space} to
          $C$
          at
          $\bm{x} \in C$ (see~\eqref{eq:manifold_tangent}).
          In this case, for any
          $\bm{x} \in C$,
          $\mathrm{Exp}_{\bm{x}}:T_{C}(\bm{x})\to C$
          with the {\em exponential map}\footnote{
          By~\cite[Defs. 10.6 and 10.7, Thm. 10.8, Exe. 10.12]{Boumal23}
          and~\cite[Thm. 4.27]{Lee18},
          {\em the exponential map} is uniquely defined over the {\em tangent bundle}
          $TC\coloneqq \{(\bm{x},\bm{v}) \mid \bm{x} \in C, \bm{v}\in T_{C}(\bm{x})\}$
          as
          $\mathrm{Exp}:TC \to C:(\bm{x},\bm{v})\mapsto \mathrm{Exp}_{\bm{x}}(\bm{v})\coloneqq \gamma_{\bm{v}}(1)$
          with the {\em geodesics}
          $\gamma_{\bm{v}}:I (\subset \mathbb{R})\to C$ (see~\cite[Def. 5.38]{Boumal23},~\cite[(4.16)]{Lee18})
          such that
          $\gamma_{\bm{v}}(0) = \bm{x}$,
          $\gamma_{\bm{v}}'(0) = \bm{v}$
          and
          the {\em geodesic equation} $\gamma_{\bm{v}}''(t) = 0$
          is satisfied for all
          $t \in I$,
          where
          $I\subset \mathbb{R}$
          is an interval containing
          $[0,1]$.
          }
          $\mathrm{Exp}$
          is continuously differentiable by~\cite[Prop. 5.19]{Lee18}, and surjective onto
          $C$ by~\cite[Cor. 6.21]{Lee18} together with the completeness (closedness in
          $\mathcal{X}$) of
          $C$.
          Hence,
          $\mathrm{Exp}_{\bm{x}}$
          is a parametrization of
$C$.
          See~\cite[p.136-p.142]{Lee18} for examples of the exponential map of some manifolds.
  \end{enumerate}
\end{example}

For Problem~\ref{problem:origin}, we propose a variable smoothing algorithm as an extension of~\cite{Bohm-Wright21} proposed for a special instance of Problem~\ref{problem:constrained} with
$C=\mathcal{X}$
and linear
$\mathfrak{S}$.
We note that the variable smoothing algorithm~\cite{Bohm-Wright21} cannot be applied directly to Problem~\ref{problem:origin} due to the nonlinearity of
$\mathfrak{S}\circ\param$.
The proposed algorithm performs a gradient descent update of a smoothed surrogate function (smooth approximation)
$\cost^{\langle \mu_{n}\rangle}\circ \param \coloneqq (h+\moreau{g}{\mu_{n}}\circ \mathfrak{S}) \circ \param\ (\mu_{n}\searrow 0)$,
of
$\cost \circ \param$,
designed partially with the {\em Moreau envelope}
$\moreau{g}{\mu}$
of
$g$
(see~\eqref{eq:Moreau}).
In a case where the proximity operator
$\prox{\mu g}$
is computable, the proposed algorithm is a single-loop algorithm.

Our subsequential convergence analysis in Theorem~\ref{theorem:convergence_extension} is first carried out for Problem~\ref{problem:origin} by exploiting a {\em gradient consistency property} of
$\cost^{\langle \mu\rangle} \circ \param$
for
$\cost \circ \param$
shown in Theorem~\ref{theorem:stationary}.
This property makes it possible to express the stationarity condition of Problem~\ref{problem:origin} in terms of a gradient sequence
$\nabla (\cost^{\langle \mu\rangle} \circ \param)\ (\mu \searrow 0)$.
The corresponding stationarity statement of the proposed algorithm for the original Problem~\ref{problem:constrained} follows under an additional condition on
$\param$,
such as the submersion
condition, discussed in Theorem~\ref{theorem:necessary} and Corollary~\ref{corollary:optimality_C_F}.
Finally, numerical experiments, in Section~\ref{sec:experiment},
demonstrate the efficacy of the proposed algorithm in two scenarios of the sparse spectral clustering and the sparse principal component analysis.
A preliminary short version of this paper was partially presented in~\cite{Kume-Yamada24}.

The main contributions and novelties are summarized below.

\subsection*{Contributions and Novelties of This Paper}
\begin{enumerate}[label=(\Roman*)]
  \item
        A framework of a parametrization strategy for Problem~\ref{problem:constrained} via Problem~\ref{problem:origin} is presented in Section~\ref{sec:necessary} where relations between stationarities of Problems~\ref{problem:constrained} and Problem~\ref{problem:origin} are presented in Theorem~\ref{theorem:necessary} and Corollary~\ref{corollary:optimality_C_F}.
        For example, under the {\em submersion condition} on
        $\param$,
        stationarities for Problems~\ref{problem:constrained} and~\ref{problem:origin} are equivalent.
        Parametrization strategies have been originally developed for smooth optimization problems (see, e.g.,~\cite{Yamada-Ezaki03,Fraikin-Huper-Dooren07,Hori-Tanaka10,Helfrich-Willmott-Ye18,Lezcano19,Kume-Yamada21,Kume-Yamada22,Levin-Kileel-Boumal24}).
        In contrast, our framework enables us to apply a parametrization strategy even to nonsmooth optimization problems.
  \item
        The gradient consistency property of our smoothed surrogate function
        $(h+\moreau{g}{\mu}\circ \mathfrak{S}) \circ \param$
        for
        $\cost\circ\param$
        in Problem~\ref{problem:origin}
        is presented in Section~\ref{sec:smoothing}
        (see Theorem~\ref{theorem:stationary} and Remark~\ref{remark:gradient_consistency}~\ref{enum:converse_inclusion}).
        The gradient consistency property plays a key role in establishing a subsequential convergence guarantee of smoothing-type algorithms~\cite{Chen12} (see Appendix~\ref{appendix:related_work:smoothing} for a review of smoothing-type algorithms).

  \item
        In Section~\ref{sec:proposed}, we propose a single-loop variable smoothing algorithm for Problem~\ref{problem:constrained} via Problem~\ref{problem:origin}.
        For the proposed algorithm, we present a subsequential convergence guarantee to a stationary point
        $\bm{y}^{\star} \in \mathcal{Y}$
        of Problem~\ref{problem:origin}
        (Theorem~\ref{theorem:convergence_extension})
        by leveraging the gradient consistency property in Theorem~\ref{theorem:stationary}.
        Moreover, under an additional condition in Theorem~\ref{theorem:necessary} (in particular the submersion condition in Corollary~\ref{corollary:optimality_C_F}) of
        $\param$,
        $\param(\bm{y}^{\star})\in C$
        is a stationary point of Problem~\ref{problem:constrained}.
        Even without this condition on
        $\param$,
        $\param(\bm{y}^{\star})$
        satisfies a necessary condition for the local optimality of
        Problem~\ref{problem:constrained}
        (see Remark~\ref{remark:submersion}).
        We also note that, in a special case where
        $C=\mathcal{X}$
        and
        $\param=\Id$,
        the proposed algorithm reduces to the variable smoothing algorithm in~\cite{Bohm-Wright21}
        if
        $\mathfrak{S}$
        is linear.
        In the same unconstrained setting, the present algorithm further allows a
        nonlinear continuously differentiable mapping
        $\mathfrak{S}$,
        which is outside the setting of~\cite{Bohm-Wright21}.
        The key technical ingredient enabling this extension is the gradient consistency
        property established in Theorem~\ref{theorem:stationary}.
  \item
        As a supplementary result, Corollary~\ref{corollary:complexity} establishes a first-order oracle complexity
        $\mathcal{O}(\epsilon^{-3})$
        (resp.
        $\mathcal{O}(\epsilon^{-3}\log(\epsilon^{-1}))$)
        of the proposed algorithm with diminishing stepsizes (resp. stepsizes obtained by a backtracking algorithm in Algorithm~\ref{alg:backtracking})
        for finding an $\epsilon$-approximate stationary point of Problem~\ref{problem:origin}.
        We note that our approximate stationarity notion in Definition~\ref{definition:complexity}~\ref{enum:definition:complexity:approximate_stationarity} is not directly comparable with the notions used in, e.g.,~\cite{Peng-Wu-Hu-Deng23,Beck-Rosset23,Xu-Jian-Liu-So25,Yuan24,Xu-Jiang-Liu-So26},~\cite{Chen-Ma-Man-Anthony-Zhang20,Wang-Liu-Chen-Ma-Xue-Zhao22}, and~\cite{Li-Chen-Deng-Qu-Zhu-Man21} for Problem~\ref{problem:constrained},
        because the underlying stationarity notions differ.
        In contrast, in the unconstrained case
        $C=\mathcal{X}$
        in Problem~\ref{problem:constrained}
        (i.e.,
        $\param=\Id$
        in Problem~\ref{problem:origin}),
        our stationarity notion reduces to that used in~\cite{Bohm-Wright21} when
        $\mathfrak{S}$
        is linear; hence 
        Corollary~\ref{corollary:complexity} extends the oracle complexity
        $\mathcal{O}(\epsilon^{-3})$
        of the variable smoothing algorithm~\cite{Bohm-Wright21} to the case where
        $\mathfrak{S}$
        is a continuously differentiable nonlinear mapping.

\end{enumerate}

\noindent
{\bf Notation}
$\mathbb{N}$,
$\mathbb{R}$,
$\mathbb{R}_{+}$,
and
$\mathbb{R}_{++}$
denote respectively the sets of all positive integers, all real numbers, all nonnegative real numbers, and all positive real numbers.
$\inprod{\cdot}{\cdot}_{\mathcal{X}}$
and
$\norm{\cdot}_{\mathcal{X}}$
stand respectively for the standard inner product and its induced norm defined in the Euclidean space
$\mathcal{X}$,
i.e.,
$\norm{\bm{x}}_{\mathcal{X}} = \sqrt{\inprod{\bm{x}}{\bm{x}}_{\mathcal{X}}}\ (\bm{x}\in \mathcal{X})$.
We simply use
$\inprod{\cdot}{\cdot}$
and
$\norm{\cdot}$
without any subscript if their domains are clear.
The symbol
$\mathrm{Id}$
stands for the identity operator.
For a linear operator
$A:\mathcal{X}\to \mathcal{Y}$,
its adjoint operator
$A^{*}:\mathcal{Y}\to \mathcal{X}$
is defined as the linear operator satisfying
$\inprod{\bm{x}}{A^{*}\bm{y}}_{\mathcal{X}} = \inprod{A\bm{x}}{\bm{y}}_{\mathcal{Y}}\ (\forall \bm{x}\in \mathcal{X}, \bm{y}\in \mathcal{Y})$,
$\mathrm{Ker}(A)\coloneqq  \{\bm{x} \in \mathcal{X} \mid A\bm{x} =\bm{0}\}$
and
$\mathrm{Ran}(A)\coloneqq  \{A\bm{x} \in \mathcal{Y} \mid \bm{x} \in \mathcal{X}\}$
are respectively the kernel (null) space and the range space of
$A$,
and $\norm{A}_{\rm op}\coloneqq \sup_{\norm{\bm{x}}_{\mathcal{X}}\leq 1, \bm{x}\in \mathcal{X}} \norm{A\bm{x}}_{\mathcal{Y}}$
denotes the operator norm of
$A$.
The symbol
$(\cdot)^{\perp}$
denotes the orthogonal complement subspace of a given set.
For a matrix
$\bm{X} \in \mathbb{R}^{m\times n}$,
$[\bm{X}]_{i,j}$
denotes the
$(i,j)$
entry of
$\bm{X}$,
and
$\bm{X}^{\TT}$
denotes the transpose of
$\bm{X}$.
For a given point
$\widebar{\bm{x}} \in \mathcal{X}$
and a given set
$E \subset \mathcal{X}$,
$d:\mathcal{X}\times 2^{\mathcal{X}} \to \mathbb{R}_{+} \cup \{+\infty\}$
stands for the distance function, i.e.,
\begin{equation}
  d(\widebar{\bm{x}},E)\coloneqq
  \left\{
  \begin{array}{cl}
    \inf\{\norm{\bm{v}-\widebar{\bm{x}}}\mid \bm{v} \in E\}, & \mathrm{if}\  E\neq \emptyset; \\
    +\infty,                                                 & \mathrm{if}\  E=\emptyset.
  \end{array}
  \right. \label{eq:distance}
\end{equation}
For a given
$(a_{n})_{n=1}^{\infty} \subset \mathbb{R}$
and a given
$\widebar{a} \in \mathbb{R}$,
$a_{n} \searrow \widebar{a}$
and
$(a_{n})_{n=1}^{\infty} \searrow \widebar{a}$
mean in this paper that
$(a_{n})_{n=1}^{\infty}$
is
monotonically nonincreasing together with
$\lim_{n\to\infty} a_{n} = \widebar{a}$.

For a differentiable mapping
$\mathcal{F}:\mathcal{Y}\to \mathcal{X}$,
its G\^{a}teaux derivative at
$\bm{y}\in \mathcal{Y}$
is the linear operator
$\mathrm{D}\mathcal{F}(\bm{y}):\mathcal{Y} \to \mathcal{X}$
defined by
\begin{equation}
  (\bm{v}\in \mathcal{Y}) \quad \mathrm{D}\mathcal{F}(\bm{y})[\bm{v}] = \lim_{\mathbb{R}\setminus \{0\} \ni t\to 0}\frac{\mathcal{F}(\bm{y}+t\bm{v}) - \mathcal{F}(\bm{y})}{t}.
\end{equation}
For a
differentiable function
$J:\mathcal{X} \to \mathbb{R}$,
$\nabla J(\bm{x}) \in \mathcal{X}$
is the gradient of
$J$
at
$\bm{x} \in \mathcal{X}$
if
$\mathrm{D}J(\bm{x})[\bm{v}] = \inprod{\nabla J(\bm{x})}{\bm{v}}$
for all
$\bm{v} \in \mathcal{X}$.
Note that for a differentiable mapping
$\mathcal{F}:\mathcal{Y} \to \mathcal{X}$,
we have the expression\footnote{
  It can be checked by
  $(\forall\bm{v}\in \mathcal{Y})\ \inprod{\nabla (J\circ \mathcal{F})(\bm{x})}{\bm{v}}_{\mathcal{Y}} = \mathrm{D}(J\circ \mathcal{F})(\bm{x})[\bm{v}] = \mathrm{D}J(\mathcal{F}(\bm{x}))[\mathrm{D}\mathcal{F}(\bm{x})[\bm{v}]] = \inprod{\nabla J(\mathcal{F}(\bm{x}))}{\mathrm{D}\mathcal{F}(\bm{x})[\bm{v}]}_{\mathcal{X}} = \inprod{(\mathrm{D}\mathcal{F}(\bm{x}))^{*}[\nabla J(\mathcal{F}(\bm{x}))]}{\bm{v}}_{\mathcal{Y}}$.}
$\nabla (J\circ \mathcal{F})(\bm{y}) = (\mathrm{D}\mathcal{F}(\bm{y}))^{*}[\nabla J(\mathcal{F}(\bm{y}))]\ (\bm{y} \in \mathcal{Y})$.
For a function
$J:\mathcal{X} \to \mathbb{R}\cup\{+\infty\}$,
$\dom{J}\coloneqq \{\bm{x}\in \mathcal{X} \mid J(\bm{x}) < +\infty\} \subset \mathcal{X}$
and
$\epi{J}:=\{(\bm{x},a) \in \mathcal{X}\times \mathbb{R} \mid J(\bm{x})\leq a\} \subset \mathcal{X}\times \mathbb{R}$
denote respectively the domain and the epigraph of
$J$.
Moreover,
$J$
is said to be
(i)
{\em proper} if
$\dom{J} \neq \emptyset$;
(ii)
{\em lower semicontinuous} if
$\epi{J}$
is closed in
$\mathcal{X} \times \mathbb{R}$;
(iii)
{\em locally Lipschitz continuous} at
$\widebar{\bm{x}} \in \mathcal{X}$
if there exist an open neighborhood
$\mathcal{N}_{\widebar{\bm{x}}} \subset \mathcal{X}$
of
$\widebar{\bm{x}}$
and
$L > 0$
such that
$\abs{J(\bm{x}_{1})-J(\bm{x}_{2})} \leq L \norm{\bm{x}_{1} - \bm{x}_{2}}\ (\bm{x}_{1},\bm{x}_{2}\in \mathcal{N}_{\widebar{\bm{x}}})$.
For a function
$J:\mathcal{X} \to \mathbb{R}\cup\{-\infty,+\infty\}$,
{\em the limit inferior}~\cite[Definition 1.5]{Rockafellar-Wets98}
and
  {\em the limit superior}~\cite[p.13]{Rockafellar-Wets98}
of
$J$
at
$\widebar{\bm{x}} \in \mathcal{X}$
are defined respectively by
  {
    \begin{align}
      \liminf_{\mathcal{X}\ni \bm{x}\to \widebar{\bm{x}}} J(\bm{x})
       & \coloneqq
      \sup_{\epsilon > 0} \left(\inf_{\norm{\bm{x}-\widebar{\bm{x}}}< \epsilon} J(\bm{x})\right)
      \coloneqq
      \sup\limits_{\epsilon > 0} \left(\inf\left\{
      J(\bm{x})\in \mathbb{R}\cup\{-\infty,+\infty\}  \mid \norm{\bm{x}-\widebar{\bm{x}}} < \epsilon \right\}\right); \\
      \limsup_{\mathcal{X}\ni \bm{x}\to \widebar{\bm{x}}} J(\bm{x})
       & \coloneqq
      \inf_{\epsilon > 0} \left(\sup_{\norm{\bm{x}-\widebar{\bm{x}}}< \epsilon} J(\bm{x})\right)
      \coloneqq
      \inf\limits_{\epsilon > 0} \left(\sup\left\{
      J(\bm{x})\in \mathbb{R} \cup\{-\infty,+\infty\}  \mid \norm{\bm{x}-\widebar{\bm{x}}} < \epsilon \right\}\right).
      \label{eq:limsup}
    \end{align}}%

For a set sequence
$(E_{n})_{n=1}^{\infty}$
with subsets
$E_{n} \subset \mathcal{X}$,
{\em the outer limit}~\cite[Definition 4.1]{Rockafellar-Wets98} of
$(E_{n})_{n=1}^{\infty}$
is defined by
\begin{equation}
  \Limsup_{n\to\infty} E_{n}
  \coloneqq  \left\{\bm{v} \in \mathcal{X} \mid
  \exists\mathcal{N}\subset \mathbb{N}, \exists \bm{v}_{n}\in E_{n}\ (n\in \mathcal{N})
  \ {\rm with}\  \abs{\mathcal{N}}=\aleph_{0}\ {\rm and}\ \bm{v}= \lim_{\mathcal{N}\ni n\to\infty} \bm{v}_{n}
  \right\},
\end{equation}
where
$\aleph_{0}$
stands for the cardinality of
$\mathbb{N}$.
With the outer limit of a set sequence,
the outer limit~\cite[p.152]{Rockafellar-Wets98} of
a set-valued mapping
$S:\mathcal{X}\rightrightarrows\mathcal{X}$
at
$\widebar{\bm{x}} \in \mathcal{X}$
is defined as
\begin{align}
   & \Limsup_{\mathcal{X}\ni\bm{x}\to\widebar{\bm{x}}} S(\bm{x})
  \coloneqq
  \bigcup_{\mathcal{X} \ni \bm{x}_{n}\to \widebar{\bm{x}}}
  \Limsup_{n\to\infty} S(\bm{x}_{n})\label{eq:outer_limit_another}                                                                                                                                                                                                                           \\
   & =  \left\{\bm{v} \in \mathcal{X} \mid \exists (\bm{x}_{n})_{n=1}^{\infty} \subset \mathcal{X},\ \exists \bm{v}_{n} \in S(\bm{x}_{n})\ \mathrm{with}\  \widebar{\bm{x}}=\lim_{n\to\infty}\bm{x}_{n}\ \mathrm{and}\ \bm{v} = \lim_{n\to\infty}\bm{v}_{n} \right\}, \label{eq:outer_limit}
\end{align}
where
we follow a clear notation
``$\Limsup$''
(used, e.g., in~\cite{Burke-Hoheisel13}).
In this paper, even for the outer limit of a single-valued mapping
$S:\mathcal{X}\to\mathcal{X}$,
we use
$\Limsup_{\mathcal{X}\ni\bm{x}\to\widebar{\bm{x}}}S(\bm{x}) \coloneqq  \Limsup_{\mathcal{X}\ni \bm{x}\to\widebar{\bm{x}}}\{S(\bm{x})\}$
in order to distinguish the standard notation for the limit superior
``$\limsup$''
in~\eqref{eq:limsup}.

\section{Preliminaries}
\subsection{Subdifferential and Normal Cone}
We review some necessary notions, namely subdifferentials and normal cones, in nonsmooth analysis.
We refer to the seminal book~\cite{Rockafellar-Wets98} for a thorough review.
\begin{definition}[Subdifferential~{\cite[Def. 8.3]{Rockafellar-Wets98}}]\label{definition:subdifferential}
  For a possibly nonconvex function
  $J:\mathcal{X} \to \exR$,
  {\em the limiting (or general) subdifferential} of
  $J$
  at
  $\widebar{\bm{x}}\in \mathcal{X}$
  is defined by
  \begin{equation}
    (\widebar{\bm{x}} \in \dom{J}) \quad
    \Lsubdiff  J(\widebar{\bm{x}}) \coloneqq
    \Limsup_{\mathcal{X}\ni \bm{x}\to \widebar{\bm{x}},\ J(\bm{x})\to J(\widebar{\bm{x}})} \Fsubdiff J(\bm{x}),
    \label{eq:general_subdifferential}
  \end{equation}
  where
  $\Fsubdiff  J(\widebar{\bm{x}}) \coloneqq \left\{\bm{v} \in \mathcal{X} \mid \liminf_{\mathcal{X}\setminus\{\widebar{\bm{x}}\}\ni\bm{x}\to \widebar{\bm{x}}} \frac{J(\bm{x}) - J(\widebar{\bm{x}}) - \inprod{\bm{v}}{\bm{x}-\widebar{\bm{x}}}}{\norm{\bm{x}-\widebar{\bm{x}}}}  \geq 0\right\}$
  is called {\em the Fr\'echet (or regular)  subdifferential} of
  $J$
  at
  $\widebar{\bm{x}} \in \dom{J}$,
  and
  $\Lsubdiff J(\widebar{\bm{x}})$
  and
  $\Fsubdiff J(\widebar{\bm{x}})$
  at
  $\widebar{\bm{x}} \notin \dom{J}$
  are understood as
  $\emptyset$
  (Note: for a convex function
  $J$,
  $\Lsubdiff  J$
  and
  $\Fsubdiff  J$
  are identical to the standard convex subdifferential of
  $J$~\cite[Prop. 8.12]{Rockafellar-Wets98}).
\end{definition}

By using the subdifferentials,
{\em the general normal cone}
$N_{C}(\widebar{\bm{x}})$
and
  {\em the regular normal cone}
$\widehat{N}_{C}(\widebar{\bm{x}})$
to a nonempty subset
$C\subset \mathcal{X}$
at
$\widebar{\bm{x}} \in C$
are defined (see~\cite[Def.~6.3,~Prop.~6.5, and~Ex.~8.14]{Rockafellar-Wets98}) respectively by
\begin{align}
  N_{C}(\widebar{\bm{x}})
   & \coloneqq
  \Lsubdiff \iota_{C}(\widebar{\bm{x}})
  =
  \Limsup_{C\ni \bm{x}\to \widebar{\bm{x}}}  \Fsubdiff \iota_{C}(\bm{x})
  = \Limsup_{C\ni \bm{x}\to \widebar{\bm{x}}} \widehat{N}_{C}(\bm{x}) \label{eq:general_normal} \\
  \widehat{N}_{C}(\widebar{\bm{x}})
   & \coloneqq \Fsubdiff \iota_{C}(\widebar{\bm{x}})
  = \left\{\bm{v} \in \mathcal{X} \mid \limsup_{C\setminus \{\widebar{\bm{x}}\}\ni \bm{x}\to \widebar{\bm{x}}} \frac{\inprod{\bm{v}}{\bm{x}-\widebar{\bm{x}}}}{\norm{\bm{x}-\widebar{\bm{x}}}} \leq 0\right\},
  \label{eq:regular_normal}
\end{align}
where
$N_{C}(\widebar{\bm{x}})$
and
$\widehat{N}_{C}(\widebar{\bm{x}})$
at
$\widebar{\bm{x}} \notin C$
are understood as
$\emptyset$.

With notions of normal cone operators, we define regularities of a subset and functions as follows.
Such regularities ensure the chain and sum rules for the limiting subdifferential as shown in Lemma~\ref{lemma:subdifferentially} below (see also Fact~\ref{fact:chain_rule} in Appendix~\ref{appendix:lemma:subdifferentially}).

\begin{definition}[Regularity]\label{definition:regular}
  \mbox{}
  \begin{enumerate}[label=(\alph*)]
    \item (Clarke regularity~\cite[Def. 6.4]{Rockafellar-Wets98})
          A nonempty set
          $C \subset \mathcal{X}$
          is said to be {\em Clarke regular} at
          $\widebar{\bm{x}} \in C$
          if
          (i)
          $N_{C}(\widebar{\bm{x}}) = \widehat{N}_{C}(\widebar{\bm{x}})$ (Note: the inclusion
          $N_{C}(\widebar{\bm{x}}) \supset \widehat{N}_{C}(\widebar{\bm{x}})$
          always holds),
          and
          (ii)
          $C \cap \overline{\mathcal{N}}_{\widebar{\bm{x}}}$
          is closed for some closed neighborhood\footnote{
          In other words,
          $\overline{\mathcal{N}}_{\widebar{\bm{x}}}$
          is a closed set that contains
          $\widebar{\bm{x}}$
          as an interior point of
          $\overline{\mathcal{N}}_{\widebar{\bm{x}}}$.
          }
          $\overline{\mathcal{N}}_{\widebar{\bm{x}}} \subset \mathcal{X}$
          of
          $\widebar{\bm{x}}$.
          Moreover,
          $C$
          is said to be Clarke regular if it is Clarke regular at every
          $\widebar{\bm{x}} \in C$.
    \item (Subdifferential regularity~\cite[Def. 7.25]{Rockafellar-Wets98}) \label{enum:definition:regular:function}
          A proper lower semicontinuous function
          $J:\mathcal{X}\to \exR$
          is said to be {\em subdifferentially regular} at
          $\widebar{\bm{x}} \in \dom{J}$
          if
          $\epi{J} \subset \mathcal{X}\times \mathbb{R}$
          is Clarke regular
          at
          $(\widebar{\bm{x}}, J(\widebar{\bm{x}}))$.
          Moreover,
          $J$
          is said to be subdifferentially regular if it is subdifferentially regular at every
          $\widebar{\bm{x}} \in \dom{J}$.
          We note that all functions in this paper are shown to be subdifferentially regular (see, e.g., Lemma~\ref{lemma:subdifferentially} below and also Example~\ref{ex:regular}).
  \end{enumerate}
\end{definition}

A typical Clarke regular set is a closed convex set~\cite[Thm. 6.9]{Rockafellar-Wets98}.
Another important example of Clarke regular sets is an embedded submanifold of the Euclidean space
$\mathcal{X}$~\cite[Exm. 6.8]{Rockafellar-Wets98}
as follows.

\begin{example}[Embedded submanifold of $\mathcal{X}${~\cite[Exm. 6.8]{Rockafellar-Wets98}}]\label{ex:normal}
  Let
  $(\emptyset \neq)C \subset \mathcal{X}$
  be a
  $d (\leq \dim(\mathcal{X}))$-dimensional embedded submanifold of
  $\mathcal{X}$,
  i.e.,
  for each
  $\widebar{\bm{x}} \in C$,
  there exist
  (i) an open neighborhood
  $\mathcal{N}_{\widebar{\bm{x}}} \subset \mathcal{X}$
  of
  $\widebar{\bm{x}}$
  in
  $\mathcal{X}$
  and
  (ii)
  a continuously differentiable mapping
  $H_{\widebar{\bm{x}}}:\mathcal{N}_{\widebar{\bm{x}}} \to \mathbb{R}^{\dim(\mathcal{X}) - d}$
  such that
  $C\cap \mathcal{N}_{\widebar{\bm{x}}} = H_{\widebar{\bm{x}}}^{-1}(\bm{0})$
  and
  $\mathrm{Rank}(\mathrm{D}H_{\widebar{\bm{x}}}(\widebar{\bm{x}})) = \dim(\mathcal{X})-d$,
  where
  $H_{\widebar{\bm{x}}}$
  is called a {\em local defining map}~\cite[Def. 3.10]{Boumal23}
  (Note: by Whitney's embedding theorem (see, e.g.,~\cite[Thm. 6.15]{Lee12}), every
  $d$-dimensional smooth manifold can be seen as an embedded submanifold in
  $\mathbb{R}^{2d+1}$).
  Then, by~\cite[Def. 6.4 and Exm. 6.8]{Rockafellar-Wets98}, the embedded submanifold
  $C$
  is Clarke regular, and
  we have the following expression:
  \begin{equation}
    (\widebar{\bm{x}} \in C) \quad
    N_{C}(\widebar{\bm{x}}) = \widehat{N}_{C}(\widebar{\bm{x}})
    = \mathrm{Ran}((\mathrm{D}H_{\widebar{\bm{x}}}(\widebar{\bm{x}}))^{*})\subset \mathcal{X},
    \quad
    T_{C}(\widebar{\bm{x}}) = \left(N_{C}(\widebar{\bm{x}})\right)^{\perp} \subset \mathcal{X},
    \label{eq:manifold_tangent}
  \end{equation}
  where the orthogonal complement subspace
  $T_{C}(\widebar{\bm{x}})$
  of
  $N_{C}(\widebar{\bm{x}})$
  is called the {\em tangent cone to $C$ at $\widebar{\bm{x}}$}
  (Note:
  $N_{C}(\widebar{\bm{x}})$
  and
  $T_{C}(\widebar{\bm{x}})$
  are also called respectively the normal space and the tangent space to the manifold
  $C$
  at
  $\widebar{\bm{x}} \in C$~\cite[Thm. 3.15, Def. 3.16, and p.107]{Boumal23}).
\end{example}

\begin{lemma}[Subdifferential calculus for Problems~\ref{problem:constrained} and~\ref{problem:origin}] \label{lemma:subdifferentially}
  Consider
  $\cost= h+g\circ\mathfrak{S}$,
  $C \subset \mathcal{X}$
  and
  $\param$
  in Problems~\ref{problem:constrained} and~\ref{problem:origin}.
  Then,
  $\cost$,
  $\cost+\iota_{C}$
  and
  $\cost\circ\param$
  are subdifferentially regular, and
  we have the following expressions:
  \begin{align}
    (\widebar{\bm{x}} \in C) \quad           &
    \Lsubdiff  (\cost + \iota_{C})(\widebar{\bm{x}}) =
    \Lsubdiff  \cost (\widebar{\bm{x}}) + N_{C}(\widebar{\bm{x}}); \label{eq:subdifferentially_f_C} \\
    (\widebar{\bm{y}} \in \mathcal{Y}) \quad &
    \Lsubdiff  (\cost \circ \param )(\widebar{\bm{y}}) = (\mathrm{D}\param (\widebar{\bm{y}}))^{*}[\Lsubdiff \cost(\param (\widebar{\bm{y}}))]. \label{eq:chain2}
  \end{align}
\end{lemma}
\begin{proof}
  See Appendix~\ref{appendix:lemma:subdifferentially}.
\end{proof}

The following states a relation between the normal cone to the parameterized set, and the parametrization
$\param $.
\begin{fact}[Normal cone to image sets~{\cite[Thm. 6.43]{Rockafellar-Wets98}}] \label{fact:normal_image}
  Let
  $\mathcal{X},\mathcal{Y}$
  be Euclidean spaces, and
  $\param :\mathcal{Y} \to \mathcal{X}$
  a continuously differentiable mapping.
  Then, for every
  $\widebar{\bm{x}} \in \param (\mathcal{Y})$,
  \begin{equation}
    \thickmuskip=0.3\thickmuskip
    \medmuskip=0.3\medmuskip
    \thinmuskip=0.3\thinmuskip
    \arraycolsep=0.3\arraycolsep
    \widehat{N}_{ \param (\mathcal{Y})}(\widebar{\bm{x}})
    \subset
    \bigcap_{\widebar{\bm{y}} \in \param ^{-1}(\widebar{\bm{x}})} \left\{\bm{v} \in \mathcal{X} \mid (\mathrm{D}\param (\widebar{\bm{y}}))^{*}[\bm{v}] \in \widehat{N}_{\mathcal{Y}}(\widebar{\bm{y}}) = \{\bm{0}\}\right\}
    = \bigcap_{\widebar{\bm{y}} \in \param ^{-1}(\widebar{\bm{x}})} \mathrm{Ker} \left((\mathrm{D}\param (\widebar{\bm{y}}))^{*}\right),
  \end{equation}
  where
  $\param ^{-1}(\widebar{\bm{x}})\coloneqq \{\widebar{\bm{y}} \in \mathcal{Y} \mid \param (\widebar{\bm{y}}) = \widebar{\bm{x}}\} \subset \mathcal{Y}$
  is the preimage of
  $\widebar{\bm{x}} \in \mathcal{X}$
  under
  $\param $.
\end{fact}
\begin{proof}
  The first inclusion is verified by~\cite[Thm. 6.43]{Rockafellar-Wets98}.
  The equality follows
  by the relation
  $(\mathrm{D}\param (\widebar{\bm{y}}))^{*}[\bm{v}] \in  \{\bm{0}\} \Leftrightarrow \bm{v} \in \mathrm{Ker} \left((\mathrm{D}\param (\widebar{\bm{y}}))^{*}\right)\ (\widebar{\bm{y}} \in \param ^{-1}(\widebar{\bm{x}}))$.
\end{proof}

\subsection{Proximity Operator and Moreau Envelope}
{\em The proximity operator} and {\em the Moreau envelope} have been serving as key ingredients for minimization of nonsmooth convex functions (see, e.g.,~\cite{Bauschke-Combettes17,Condat-Kitahara-Contreras-Hirabayashi23}).
Even for weakly convex functions, they are useful in design of nonsmooth optimization algorithms.
We note that closed-form expressions of the proximity operators for commonly-used weakly convex functions have been found, e.g.,
$\ell_{1}$-norm~\cite{Bauschke-Combettes17}, Minimax Concave Penalty (MCP)~\cite{Zhang10}, Smoothly Clipped Absolute Deviation (SCAD)~\cite{Fan-Li01},
log-sum penalty~\cite{Prater-Bennette-Shen-Tripp22},
piece-wise exponential function~\cite{Liu-Zhou-Lin24},
and arctangent penalty~\cite{He-Shu-Wen-So24}.

\begin{definition}[Proximity operator and Moreau envelope~{\cite[Def. 1.22]{Rockafellar-Wets98}}]\label{definition:Moreau}
  Let
  $g:\mathcal{Z}\to \exR$
  be proper lower semicontinuous, and
  $\eta$-weakly convex, i.e.,
  $g+\eta/2\norm{\cdot}^{2}$
  is convex with
  $\eta > 0$.
  For
  $\mu \in (0,\eta^{-1})$,
  {\em the proximity operator} and {\em the Moreau envelope} of
  $g$
  of index
  $\mu$
  are defined respectively as
    {
      \thickmuskip=0.1\thickmuskip
      \medmuskip=0.1\medmuskip
      \thinmuskip=0.1\thinmuskip
      \arraycolsep=0.1\arraycolsep
      \begin{align}
        (\widebar{\bm{z}} \in \mathcal{Z}) \  & \prox{\mu g}(\widebar{\bm{z}})  \coloneqq  \argmin_{\bm{z} \in \mathcal{Z}} \left(g(\bm{z}) + \frac{1}{2\mu}\norm{\bm{z}-\widebar{\bm{z}}}^{2}\right); \label{eq:prox} \\
        (\widebar{\bm{z}} \in \mathcal{Z}) \  & \moreau{g}{\mu}(\widebar{\bm{z}}) \coloneqq  \inf_{\bm{z} \in \mathcal{Z}} \left(g(\bm{z}) + \frac{1}{2\mu}\norm{\bm{z}-\widebar{\bm{z}}}^{2}\right)
        = g(\prox{\mu g}(\widebar{\bm{z}})) + \frac{1}{2\mu}\norm{\prox{\mu g}(\widebar{\bm{z}}) - \widebar{\bm{z}}}^{2}
        , \label{eq:Moreau}
      \end{align}}%
  where
  $\prox{\mu g}(\widebar{\bm{z}})$
  is single-valued
  because
  the unique existence of the minimizer is guaranteed by the strong convexity of
  $g(\cdot) + \frac{1}{2\mu}\norm{\cdot-\widebar{\bm{z}}}^{2}$.
\end{definition}

\begin{fact}[Properties of Moreau envelope of weakly convex function] \label{fact:moreau}
  Let
  $g:\mathcal{Z}\to \exR$
  be proper lower semicontinuous, and
  $\eta$-weakly convex
  with some
  $\eta > 0$.
  Then,
  the following hold:
  \begin{enumerate}[label=(\alph*)]
    \item \label{enum:prox_convergence}
          For any
          $\widebar{\bm{z}} \in \mathcal{Z}$
          and
          $(\bm{z}_{n})_{n=1}^{\infty} \subset \mathcal{Z}$
          satisfying
          $\lim_{n\to\infty}\bm{z}_{n} = \widebar{\bm{z}}$,
          we have
          $\lim_{n\to\infty} \prox{\mu_{n} g}(\bm{z}_{n}) = P_{\overline{\dom{\Lsubdiff  g}}}(\widebar{\bm{z}})$
          if
          $(\mu_{n})_{n=1}^{\infty}\searrow 0$
          with
          $\mu_{n} \in (0,\eta^{-1})$~\cite[Lemma 2.3]{Atenas25},
          where
          $\dom{\Lsubdiff  g}\coloneqq \{\bm{z} \in \mathcal{Z} \mid \Lsubdiff  g(\bm{z}) \neq \emptyset\}\subset \mathcal{Z}$,
          $\overline{(\cdot)}$
          denotes the closure of a given set, and
          $P_{\overline{\dom{\Lsubdiff  g}}} = \prox{\iota_{\overline{\dom{\Lsubdiff  g}}}}$
          denotes the projection mapping onto the closed convex set
          $\overline{\dom{\Lsubdiff  g}} \subset \mathcal{Z}$.
    \item \label{enum:gradient_moreau}
          For any
          $\mu \in (0,\eta^{-1})$,
          (i)
          $\moreau{g}{\mu}:\mathcal{Z}\to\mathbb{R}$
          is continuously differentiable with
          $\nabla \moreau{g}{\mu} = \frac{\mathrm{Id} - \prox{\mu g}}{\mu}$,
          and
          (ii)
          $\nabla \moreau{g}{\mu}$
          is Lipschitz continuous with a Lipschitz constant
          $\max\left\{\mu^{-1},\frac{\eta}{1-\eta\mu}\right\}$~\cite[Cor.~3.4]{Hoheisel-Laborde-Oberman20}.
    \item \label{enum:prox_subdifferential}
          $(\widebar{\bm{z}}\in\mathcal{Z}, \mu \in (0,\eta^{-1})) \quad \nabla \moreau{g}{\mu}(\widebar{\bm{z}}) \in \Lsubdiff  g(\prox{\mu g}(\widebar{\bm{z}}))$~\cite[Lemma 3.2]{Bohm-Wright21}.
    \item \label{enum:moreau_bound}
          Assume that
          $g$
          is Lipschitz continuous with a Lipschitz constant
          $L_{g} > 0$.
          By~\cite[Lemma 3.3]{Bohm-Wright21} and~\cite[Prop. 3.4]{Planiden-Wang19}\footnote{
            In~\cite[Prop. 3.4]{Planiden-Wang19}, the inequality~\eqref{eq:relation_weakly_Moreau} is proved under the convexity of
            $g$.
            However, the proof in~\cite[Prop. 3.4]{Planiden-Wang19} can be applied to the weakly convex case as well.
          }, the following hold:
          \begin{align}
            (\widebar{\bm{z}}\in\mathcal{Z}, \mu \in (0,\eta^{-1})) \quad & \norm{\nabla \moreau{g}{\mu}(\widebar{\bm{z}})} \leq L_{g}; \label{eq:moreau_gradient_bound}                                                                    \\
            (\widebar{\bm{z}}\in\mathcal{Z}, \mu \in (0,\eta^{-1})) \quad & \moreau{g}{\mu}(\widebar{\bm{z}}) \leq g(\widebar{\bm{z}}) \leq \moreau{g}{\mu}(\widebar{\bm{z}}) + \frac{1}{2}\mu L_{g}^{2}. \label{eq:relation_weakly_Moreau}
          \end{align}
          Moreover, for
          $0 < \mu_{2} < \mu_{1}< \eta^{-1}$,
          we have~\cite[Thm. 1.25]{Rockafellar-Wets98},
          \cite[Lemma 4.1]{Bohm-Wright21}
          \begin{equation}
            (\widebar{\bm{z}}\in\mathcal{Z}) \quad \moreau{g}{\mu_{1}}(\widebar{\bm{z}}) \leq \moreau{g}{\mu_{2}}(\widebar{\bm{z}}) \leq \moreau{g}{\mu_{1}}(\widebar{\bm{z}}) + \frac{1}{2}\frac{\mu_{1}-\mu_{2}}{\mu_{2}}\mu_{1}L_{g}^{2}.\label{eq:moreau_inequality}
          \end{equation}
  \end{enumerate}
\end{fact}

\section{Relation between First-order Optimality Conditions of Problems~\ref{problem:constrained} and~\ref{problem:origin}}
\label{sec:parametrization}
\label{sec:necessary}
This section presents a relation between the original Problem~\ref{problem:constrained} and the parameterized Problem~\ref{problem:origin}.
We begin with the notions of optimality.
\begin{definition}[Optimality]
  Let
  $J:\mathcal{X}\to \exR$
  be proper.
  \begin{enumerate}
    \item (Local minimizer)
          $\bm{x}^{\star\star} \in \dom{J}$
          is said to be a {\em local minimizer} of
          $J$
          if there exists an open neighborhood
          $\mathcal{N}_{\bm{x}^{\star\star}}\subset \mathcal{X}$
          of
          $\bm{x}^{\star\star}$
          such that
          $J(\bm{x}^{\star\star}) \leq J(\bm{x})\ (\bm{x} \in \mathcal{N}_{\bm{x}^{\star\star}})$.
    \item (Stationary point)
          $\bm{x}^{\star} \in \mathcal{X}$
          is said to be a {\em stationary point}\footnote{
            A stationary point
            $\bm{x}^{\star} \in \mathcal{X}$
            is also called a {\em limiting stationary point} of
            $J$
            in order to distinguish from a {\em Fr\'echet stationary point} where
            $\Fsubdiff J(\bm{x}^{\star}) \ni \bm{0}$
            holds.
            However, these stationarities are identical in a case where
            $J$
            is subdifferentially regular in Definition~\ref{definition:regular}~\ref{enum:definition:regular:function} (see~\cite[Cor. 8.11]{Rockafellar-Wets98}).
            Since all functions in this paper are subdifferentially regular, we do not distinguish these stationarities in this paper.
          }
          of
          $J$
          if
          $\Lsubdiff  J(\bm{x}^{\star}) \ni \bm{0}$.
  \end{enumerate}
\end{definition}

By Fermat's rule~\cite[Thm. 10.1]{Rockafellar-Wets98}, for a proper function
$J:\mathcal{X}\to\exR$,
$\bm{x}^{\star\star} \in \mathcal{X}$
is a stationary point of
$J$
if
$\bm{x}^{\star\star}$
is a local minimizer of
$J$.
Thus, the condition
$\Lsubdiff  J(\bm{x}^{\star}) \ni \bm{0}$
at
$\bm{x}^{\star} \in \mathcal{X}$
seems to be the most versatile expressions of {\em the first-order optimality condition} for minimization of
$J$
at
$\bm{x}^{\star}$
(see, e.g.,~\cite[p.207]{Rockafellar-Wets98}).
For this reason, we focus on finding such stationary points of Problem~\ref{problem:constrained} and Problem~\ref{problem:origin} throughout this paper.
By Lemma~\ref{lemma:subdifferentially}, the first-order optimality conditions can be expressed respectively with
$\cost = h+g\circ\mathfrak{S}$
as:
\begin{equation}
  \Lsubdiff  (\cost + \iota_{C})(\bm{x}^{\star}) \left(\overset{\eqref{eq:subdifferentially_f_C}}{=} \Lsubdiff  \cost(\bm{x}^{\star}) + N_{C}(\bm{x}^{\star})\right) \ni \bm{0} \label{eq:necessarily_nonsmooth}
\end{equation}
for
$\bm{x}^{\star}\in C$
to be a stationary point of Problem~\ref{problem:constrained}, and
\begin{align}
   & \Lsubdiff  (\cost \circ \param )(\bm{y}^{\star}) \left(\overset{\eqref{eq:chain2}}{=} (\mathrm{D}\param (\bm{y}^{\star}))^{*}[\Lsubdiff \cost (\param (\bm{y}^{\star}))]\right)  \ni \bm{0}, \label{eq:necessarily_origin} \\
   & {\rm or\ equivalently},\ d\left(\bm{0}, \Lsubdiff  (\cost \circ  \param )(\bm{y}^{\star})\right)  = 0 \label{eq:distance_stationary}
\end{align}
for
$\bm{y}^{\star} \in \mathcal{Y}$
to be a stationary point of Problem~\ref{problem:origin}.

Theorem~\ref{theorem:necessary} below presents a relation between~\eqref{eq:necessarily_nonsmooth} and~\eqref{eq:necessarily_origin}.
Based on Theorem~\ref{theorem:necessary}, the conditions~\eqref{eq:necessarily_nonsmooth} and~\eqref{eq:necessarily_origin} are equivalent with
$\bm{x}^{\star} = \param(\bm{y}^{\star})$
under a certain condition (see Corollary~\ref{corollary:optimality_C_F}).
In a special case where the cost function
$\cost$
is continuously differentiable,
a similar relation between the first-order optimality conditions is found in~\cite[Thm. 2.4]{Levin-Kileel-Boumal24}.
We note that, even if
$C$
is not an embedded submanifold of
$\mathcal{X}$
but is just a Clarke regular subset of
$\mathcal{X}$,
the same statements in Theorem~\ref{theorem:necessary} hold.

\begin{theorem}[Pointwise relation between the first-order optimality conditions~\eqref{eq:necessarily_nonsmooth} and~\eqref{eq:necessarily_origin}]\label{theorem:necessary}
  Consider Problems~\ref{problem:constrained} and~\ref{problem:origin}.
  Let
  $\bm{x}^{\star} \coloneqq  \param (\bm{y}^{\star}) \in C$
  with some
  $\bm{y}^{\star} \in \mathcal{Y}$.
  Then, the following hold:
  \begin{enumerate}[label=(\alph*)]
    \item \label{enum:necessary_x_to_y}
          If
          $\bm{x}^{\star}$
          satisfies~\eqref{eq:necessarily_nonsmooth}, then
          $\bm{y}^{\star}$
          satisfies~\eqref{eq:necessarily_origin}.
    \item \label{enum:necessary_C_F}
          Assume
          $N_{ C}(\bm{x}^{\star}) \supset  \mathrm{Ker}((\mathrm{D}\param (\bm{y}^{\star}))^{*})$.
          If
          $\bm{y}^{\star}$
          satisfies~\eqref{eq:necessarily_origin}, then
          $\bm{x}^{\star}$
          satisfies~\eqref{eq:necessarily_nonsmooth}.
  \end{enumerate}
\end{theorem}
To prove Theorem~\ref{theorem:necessary}, we use the following lemma.
\begin{lemma}[] \label{lemma:normal_kernel}
  Consider Problem~\ref{problem:constrained}.
  Let
  $\param :\mathcal{Y} \to C$
  be continuously differentiable, and
  let
  $\widebar{\bm{x}} \coloneqq \param (\widebar{\bm{y}}) \in C$
  with some
  $\widebar{\bm{y}} \in \mathcal{Y}$.
  Then,
  $N_{C}(\widebar{\bm{x}}) \subset \mathrm{Ker}((\mathrm{D}\param (\widebar{\bm{y}}))^{*})$.
\end{lemma}
\begin{proof}
  Recall
  $\widehat{N}_{\param (\mathcal{Y})}(\widebar{\bm{x}}) \subset \mathrm{Ker}((\mathrm{D}\param (\widebar{\bm{y}}))^{*})$
  from Fact~\ref{fact:normal_image}.
  Moreover, we have
  $\widehat{N}_{\param (\mathcal{Y})}(\widebar{\bm{x}}) \supset \widehat{N}_{C}(\widebar{\bm{x}}) = N_{C}(\widebar{\bm{x}})$,
  where
  ``$\supset$''
  follows\footnote{
    We can check the inclusion by showing
    $\limsup_{\param (\mathcal{Y})\setminus\{\widebar{\bm{x}}\}\ni \bm{x}\to\widebar{\bm{x}}} \frac{\inprod{\bm{v}}{\bm{x} - \widebar{\bm{x}}}}{\norm{\bm{x}-\widebar{\bm{x}}}} \leq \limsup_{C\setminus\{\widebar{\bm{x}}\}\ni \bm{x}\to\widebar{\bm{x}}} \frac{\inprod{\bm{v}}{\bm{x} - \widebar{\bm{x}}}}{\norm{\bm{x}-\widebar{\bm{x}}}}$
    from~\eqref{eq:regular_normal}.
    The inequality can be verified with~\eqref{eq:limsup} from
    $\inf\limits_{\epsilon >0} \left(\sup\limits_{\substack{0 < \norm{\bm{x}-\widebar{\bm{x}}}< \epsilon \\ \bm{x} \in \param (\mathcal{Y})}}\frac{\inprod{\bm{v}}{\bm{x} - \widebar{\bm{x}}}}{\norm{\bm{x}-\widebar{\bm{x}}}}\right) \overset{\param (\mathcal{Y})\subset C}{\leq} \inf\limits_{\epsilon >0} \left(\sup\limits_{\substack{0 < \norm{\bm{x}-\widebar{\bm{x}}}< \epsilon \\ \bm{x} \in C}}\frac{\inprod{\bm{v}}{\bm{x} - \widebar{\bm{x}}}}{\norm{\bm{x}-\widebar{\bm{x}}}}\right)$.
  }
  by
  $C\supset \param (\mathcal{Y})$
  (Note:
  ``$\supset$''
  can be replaced by
  ``$=$''
  if
  $\param (\mathcal{Y}) = C$),
  and the  equality follows by the Clarke regularity of
  $C$ (see~\eqref{eq:manifold_tangent} in Example~\ref{ex:normal}).
  By combining these inclusions, we complete the proof.
\end{proof}

\begin{proof}[Proof of Theorem~\ref{theorem:necessary}]
  For simplicity, let
  $\cost\coloneqq h+g\circ \mathfrak{S}$.

  (a)
  Let
  $\bm{x}^{\star}$
  satisfy~\eqref{eq:necessarily_nonsmooth},
  i.e.,
  there exists
  $\bm{v} \in \Lsubdiff \cost(\bm{x}^{\star})$
  such that
  $-\bm{v} \in N_{ C}(\bm{x}^{\star})$.
  Then,
  Lemma~\ref{lemma:normal_kernel} implies
  $-\bm{v} \in N_{ C}(\bm{x}^{\star}) \subset \mathrm{Ker}((\mathrm{D}\param (\bm{y}^{\star}))^{*})$.
  By combining this inclusion with the linearity of
  $\mathrm{Ker}((\mathrm{D}\param (\bm{y}^{\star}))^{*})$,
  we have
  $\bm{v} \in  \mathrm{Ker}((\mathrm{D}\param (\bm{y}^{\star}))^{*})$.
  Thus,
  $(\mathrm{D}\param (\bm{y}^{\star}))^{*}\left[\Lsubdiff   \cost \left(\param (\bm{y}^{\star})\right)\right] \ni \bm{0}$,
  i.e.,
  $\bm{y}^{\star}$
  satisfies~\eqref{eq:necessarily_origin}.

  (b)
  Let
  $\bm{y}^{\star} \in \mathcal{Y}$
  satisfy~\eqref{eq:necessarily_origin},
  i.e., there exists
  $\bm{v} \in \Lsubdiff \cost (\param (\bm{y}^{\star}))$
  such that
  $(\mathrm{D}\param (\bm{y}^{\star}))^{*}[\bm{v}] = \bm{0}$.
  The linearity of
  $\mathrm{Ker}((\mathrm{D}\param (\bm{y}^{\star}))^{*})$
  implies
  $-\bm{v} \in \mathrm{Ker}((\mathrm{D}\param (\bm{y}^{\star}))^{*})$,
  and then
  the assumption
  $N_{ C}(\bm{x}^{\star}) \supset  \mathrm{Ker}((\mathrm{D}\param (\bm{y}^{\star}))^{*})$
  yields
  $-\bm{v} \in N_{C}(\param (\bm{y}^{\star}))$.
  Thus, we have
  $\bm{0} \in \Lsubdiff \cost (\param (\bm{y}^{\star})) + N_{C}(\param (\bm{y}^{\star})) = \Lsubdiff \cost (\bm{x}^{\star}) + N_{C}(\bm{x}^{\star})$,
  i.e.,
  $\bm{x}^{\star}$
  satisfies~\eqref{eq:necessarily_nonsmooth}.
\end{proof}

\begin{corollary}[Special case for~\eqref{eq:necessarily_nonsmooth} and~\eqref{eq:necessarily_origin} to be equivalent] \label{corollary:optimality_C_F}
  Consider Problems~\ref{problem:constrained} and~\ref{problem:origin}.
  Let
  $\bm{x}^{\star} \coloneqq \param (\bm{y}^{\star}) \in C$
  with some
  $\bm{y}^{\star} \in \mathcal{Y}$.
  Suppose that
  $\param $
  is submersion\footnote{ \label{foot:submersion}
    A sufficient condition for
  $\param :\mathcal{Y}\to C$
  to be a submersion at
  $\bm{y}^{\star} \in \mathcal{Y}$
  is that
  $\param$
  is a local diffeomorphism
  at
  $\bm{y}^{\star}$
  (see~\cite[Prop. 4.8]{Lee12}), i.e., there exists an open neighborhood
  $\mathcal{N}_{\bm{y}^{\star}} \subset \mathcal{Y}$
  of
  $\bm{y}^{\star}$
  such that the restriction
  $\param |_{\mathcal{N}_{\bm{y}^{\star}}}:\mathcal{N}_{\bm{y}^{\star}} \to \param (\mathcal{N}_{\bm{y}^{\star}}) (\subset C)$
  of
  $\param $
  is bijective and differentiable, and its inversion mapping is also differentiable over
  $\param (\mathcal{N}_{\bm{y}^{\star}})$.
  }at
  $\bm{y}^{\star}$,
  i.e.,
  $\mathrm{Ran}(\mathrm{D}\param (\bm{y}^{\star}))=T_{C}(\bm{x}^{\star})$
  holds (see~\eqref{eq:manifold_tangent} for $T_{C}$).
  Then,
  $\bm{x}^{\star}$
  satisfies~\eqref{eq:necessarily_nonsmooth} if and only if
  $\bm{y}^{\star}$
  satisfies~\eqref{eq:necessarily_origin}.
\end{corollary}
\begin{proof}
  From Theorem~\ref{theorem:necessary}, it suffices to show
  $N_{C}(\bm{x}^{\star}) \supset  \mathrm{Ker}((\mathrm{D}\param (\bm{y}^{\star}))^{*})$.
  This inclusion can be verified by
  $N_{C}(\bm{x}^{\star}) \overset{\eqref{eq:manifold_tangent}}{=} (T_{C}(\bm{x}^{\star}))^{\perp} = (\mathrm{Ran}(\mathrm{D}\param (\bm{y}^{\star})))^{\perp} = \mathrm{Ker}((\mathrm{D}\param (\bm{y}^{\star}))^{*})$,
  where the second equality follows by the assumption
  $\mathrm{Ran}(\mathrm{D}\param (\bm{y}^{\star}))=T_{C}(\bm{x}^{\star})$,
  and the last one follows by\footnote{
  $\bm{u} \in (\mathrm{Ran}(\mathrm{D}\param (\bm{y}^{\star})))^{\perp}
    \Leftrightarrow (\forall \bm{v} \in \mathcal{Y}) \ \inprod{\bm{u}}{\mathrm{D}\param (\bm{y}^{\star})[\bm{v}]} = 0
    \Leftrightarrow (\forall \bm{v} \in \mathcal{Y}) \ \inprod{(\mathrm{D}\param (\bm{y}^{\star}))^{*}[\bm{u}]}{\bm{v}} = 0
    \Leftrightarrow
    (\mathrm{D}\param (\bm{y}^{\star}))^{*}[\bm{u}] = \bm{0}
    \Leftrightarrow
    \bm{u} \in \mathrm{Ker}((\mathrm{D}\param (\bm{y}^{\star}))^{*})$.
    }
  $(\mathrm{Ran}(\mathrm{D}\param (\bm{y}^{\star})))^{\perp} = \mathrm{Ker}((\mathrm{D}\param (\bm{y}^{\star}))^{*})$.
\end{proof}

\begin{remark}[Case for $\param (\mathcal{Y})\subsetneq C$] \label{remark:not_surjective}
  Even in a case where
  $\param :\mathcal{Y}\to C$
  is not surjective onto
  $C$,
  the same conclusions as in Theorem~\ref{theorem:necessary}~\ref{enum:necessary_x_to_y} and~\ref{enum:necessary_C_F}, and in Corollary~\ref{corollary:optimality_C_F} can be achieved because the surjectivity of
  $\param $
  is not used in their proofs at all.
\end{remark}
\begin{remark}[On submersion condition] \label{remark:submersion}
  A parametrization
  $\param$
  of
  $C$
  in Problem~\ref{problem:constrained}
  does not automatically satisfy
  the submersion condition
  $\mathrm{Ran}(\mathrm{D}\param (\bm{y}))=T_{C}(\param(\bm{y}))$
  at each point
  $\bm{y} \in \mathcal{Y}$
  (see, e.g.,~\cite{Kume-Yamada21} and footnote\footref{foot:Klein} in Appendix~\ref{appendix:related_work:Lagrangian}, for a parametrization with the submersion condition).
  However, even in the absence of the submersion condition\footnote{
  For a manifold
  $C$
  in Problem~\ref{problem:constrained},
  there exists a parametrization satisfying the submersion condition locally.
  Indeed, for
  $\bm{x} \in C$,
  the exponential map
  $\mathrm{Exp}_{\bm{x}}:T_{C}(\bm{x})\to C$
  (see Example~\ref{example:parametrization}~\ref{enum:ex:parameterization:genearal}) is a local diffeomorphism at
  $\bm{0}\in T_{C}(\bm{x})$~\cite[Lemma 6.16]{Lee18},
  hence submersion over an open neighborhood of
  $\bm{0}$.
  By using a family of parametrizations satisfying the submersion condition locally, an adaptive parametrization strategy has also been proposed~\cite{Lezcano19,Criscitiello-Boumal22,Kume-Yamada24A} for the smooth case of the cost function.
  } of
  $\param$,
  Theorem~\ref{theorem:necessary}~\ref{enum:necessary_x_to_y} ensures that
  the first-order optimality condition~\eqref{eq:necessarily_origin} of Problem~\ref{problem:origin} at
  $\bm{y}^{\star} \in \mathcal{Y}$
  is a necessary condition for the first-order optimality condition~\eqref{eq:necessarily_nonsmooth} of Problem~\ref{problem:constrained} at
  $\bm{x}^{\star} \coloneqq \param(\bm{y}^{\star}) \in C$,
  hence is a necessary condition for the local optimality of Problem~\ref{problem:constrained} at
  $\bm{x}^{\star}$.
\end{remark}

Theorem~\ref{theorem:necessary} and Corollary~\ref{corollary:optimality_C_F} encourage us to focus on finding a stationary point of Problem~\ref{problem:origin} for our goal of finding a stationary point of Problem~\ref{problem:constrained}.
Since Problem~\ref{problem:origin} is an unconstrained optimization problem thanks to the parametrization
$\param $,
Problem~\ref{problem:origin} seems to be more tractable than Problem~\ref{problem:constrained} in terms of designing iterative algorithms.
Indeed, iterative algorithms with parametrization have been proposed in a case where the cost function
$\cost=h+g\circ\mathfrak{S}$
is smooth~\cite{Yamada-Ezaki03,Fraikin-Huper-Dooren07,Hori-Tanaka10,Helfrich-Willmott-Ye18,Lezcano19,Ha-Liu-Barber20,Kume-Yamada21,Kume-Yamada22,Criscitiello-Boumal22,Kume-Yamada24A}.
In Section~\ref{sec:algorithm}, along such parametrization strategies, we propose an iterative algorithm applicable to Problem~\ref{problem:origin} even with a nonsmooth cost function
$\cost$.

We remark that, as an alternative strategy to the parametrization strategy for translating Problem~\ref{problem:constrained} into an unconstrained problem, penalty-based strategies
have been proposed, e.g.,~\cite[Ch. 12]{Nocedal-Wright06},~\cite{Xiao-Liu-Toh24},
in a case where
$C$
is defined as equality constraints.
See Appendix~\ref{appendix:related_work:Lagrangian} for a further discussion of the penalty-based and parametrization strategies.

\section{Variable Smoothing Algorithm for Problem~\ref{problem:origin}}\label{sec:algorithm}
\subsection{Properties of smoothed surrogate function for Problem~\ref{problem:origin}}  \label{sec:smoothing}
In this section, we investigate properties of our smoothed surrogate function
$(h+\moreau{g}{\mu}\circ \mathfrak{S}) \circ \param\ (\mu\in(0,\eta^{-1}))$
of the cost function
$(h+g\circ\mathfrak{S})\circ\param$
in Problem~\ref{problem:origin}, where
$\moreau{g}{\mu}$
is the Moreau envelope of
$g$
in~\eqref{eq:Moreau}.
The goal of this section is to present an asymptotic property, called {\em gradient (sub-)consistency}, of the gradient of
$(h+\moreau{g}{\mu}\circ \mathfrak{S}) \circ \param\ (\mu \searrow 0)$
to approximate subgradients of
$(h+g\circ\mathfrak{S})\circ\param$
(see Theorem~\ref{theorem:stationary} and Remark~\ref{remark:gradient_consistency}).
By exploiting the gradient (sub-)consistency of
$(h+\moreau{g}{\mu}\circ \mathfrak{S}) \circ \param$,
we will establish a subsequential convergence analysis for the proposed variable smoothing algorithm in Section~\ref{sec:proposed}.
We summarize basic properties of
$(h+\moreau{g}{\mu}\circ \mathfrak{S}) \circ \param$,
with
$\mu \in (0,\eta^{-1})$
fixed,
which follow from direct applications of Fact~\ref{fact:moreau}~\ref{enum:moreau_bound}.

\begin{lemma} \label{lemma:basic}
  Let
  $\cost^{\langle \mu \rangle} \circ \param \coloneqq (h+\moreau{g}{\mu}\circ\mathfrak{S})\circ \param\ (\mu\in(0,\eta^{-1}))$
  be a smoothed surrogate function of
  $\cost \circ \param = (h+g\circ\mathfrak{S})\circ\param$
  in Problem~\ref{problem:origin}.
  Then, for
  $\widebar{\bm{y}} \in \mathcal{Y}$
  and
  $\mu \in (0,\eta^{-1})$,
  we have the following expressions by~\eqref{eq:Moreau} and
  Fact~\ref{fact:moreau}~\ref{enum:gradient_moreau}:
  {
  \begin{equation}
    \thickmuskip=0.5\thickmuskip
    \medmuskip=0.5\medmuskip
    \thinmuskip=0.5\thinmuskip
    \arraycolsep=0.5\arraycolsep
    \begin{cases}
      \cost^{\langle \mu \rangle} \circ \param(\widebar{\bm{y}})
      = h(\param(\widebar{\bm{y}}))+g\left(\prox{\mu g}\left(\mathfrak{S}(\param(\widebar{\bm{y}}))\right)\right) + \frac{1}{2\mu}\norm{\prox{\mu g}\left(\mathfrak{S}\left(\param(\widebar{\bm{y}})\right)\right) - \mathfrak{S}\left(\param(\widebar{\bm{y}})\right)}^{2}, \\
      \nabla (\cost^{\langle \mu \rangle}\circ \param)(\widebar{\bm{y}})
      = (\mathrm{D}\param(\widebar{\bm{y}}))^{*}\left[
        \nabla h(\param(\widebar{\bm{y}})) + \left(\mathrm{D}\mathfrak{S}\left(\param(\widebar{\bm{y}})\right)\right)^{*}\left[\frac{\mathfrak{S}(\param(\widebar{\bm{y}})) - \prox{\mu g}\left(\mathfrak{S}(\param(\widebar{\bm{y}}))\right)}{\mu}\right]
        \right].
    \end{cases}
    \label{eq:expression_cost_gradient}
  \end{equation}}%
  With a Lipschitz constant
  $L_{g}> 0$
  of
  $g$,
  we have
  \begin{align}
    (\widebar{\bm{y}}\in\mathcal{Y}, \mu\in (0,\eta^{-1})) \quad
     &
    \cost^{\langle \mu \rangle}\circ\param(\widebar{\bm{y}}) \leq \cost\circ\param(\widebar{\bm{y}}) \leq \cost^{\langle \mu \rangle} \circ\param (\widebar{\bm{y}}) + \frac{1}{2}\mu L_{g}^{2}, \label{eq:smoothed_uniformly} \\
    (\mu\in (0,\eta^{-1})) \quad
     &
    -\infty <
    \inf_{\bm{y}\in \mathcal{Y}} \cost^{\langle \mu \rangle}\circ \param(\bm{y}). \label{eq:smoothed_bounded_below}
  \end{align}
  Moreover, with
  $0 < \mu_{2} \leq \mu_{1} < \eta^{-1}$,
  we have
  \begin{equation}
    (\widebar{\bm{y}}\in\mathcal{Y}) \quad
    \cost^{\langle \mu_{1} \rangle}\circ \param(\widebar{\bm{y}})
    \leq \cost^{\langle \mu_{2} \rangle}\circ \param(\widebar{\bm{y}})
    \leq \cost^{\langle \mu_{1} \rangle}\circ \param(\widebar{\bm{y}}) + L_{g}^{2}(\mu_{1}-\mu_{2}). \label{eq:smoothed_uniformly_different_parameter}
  \end{equation}
\end{lemma}
\begin{proof}
  The expressions~\eqref{eq:expression_cost_gradient} follow from~\eqref{eq:Moreau} and Fact~\ref{fact:moreau}~\ref{enum:gradient_moreau}.
  The inequality~\eqref{eq:smoothed_uniformly} follows from~\eqref{eq:relation_weakly_Moreau}.
  The inequality~\eqref{eq:smoothed_bounded_below} is verified by~\eqref{eq:smoothed_uniformly} and
  $\inf_{\bm{y}\in\mathcal{Y}}\cost \circ \param(\bm{y}) > -\infty$
  from the condition~\ref{enum:cost_bounded} in Problem~\ref{problem:constrained}.
  To derive~\eqref{eq:smoothed_uniformly_different_parameter}, it suffices to show
  \begin{equation}
    ( 0 < \mu_{2} \leq \mu_{1} < \eta^{-1},\
    \widebar{\bm{z}}\in\mathcal{Z}) \quad \moreau{g}{\mu_{1}}(\widebar{\bm{z}}) \leq \moreau{g}{\mu_{2}}(\widebar{\bm{z}}) \leq \moreau{g}{\mu_{1}}(\widebar{\bm{z}}) + L_{g}^{2}(\mu_{1}-\mu_{2}).\label{eq:moreau_inequality_Lipschitz}
  \end{equation}
  For the case
  $\frac{\mu_{1}}{\mu_2} \leq 2$,
  the inequality~\eqref{eq:moreau_inequality_Lipschitz} follows from~\eqref{eq:moreau_inequality}.
  Now suppose
  $\frac{\mu_{1}}{\mu_{2}} > 2$,
  and
  define
  $\widetilde{\mu}_{t} \coloneqq 2^{-t}\mu_{1}\ (t=0,1,2,\ldots,l-1)$
  and
  $\widetilde{\mu}_{l} \coloneqq \mu_{2} (\geq 2^{-l}\mu_{1})$
  with
  $l\coloneqq \lceil \log_{2}(\frac{\mu_{1}}{\mu_{2}}) \rceil \geq 2$,
  where
  $\lceil \cdot \rceil$
  is the ceiling function.
  Then, each pair
  $(\widetilde{\mu}_{t},\widetilde{\mu}_{t+1})$
  satisfies
  $\frac{\widetilde{\mu}_{t}}{\widetilde{\mu}_{t+1}} \leq 2$,
  and thus the inequality~\eqref{eq:moreau_inequality_Lipschitz} with
  $(\mu_{1},\mu_{2}) \coloneqq (\widetilde{\mu}_{t},\widetilde{\mu}_{t+1})$
  (already proved for the case
  $\frac{\mu_{1}}{\mu_2} \leq 2$)
  yields
  $0\leq \moreau{g}{\widetilde{\mu}_{t+1}}(\widebar{\bm{z}}) - \moreau{g}{\widetilde{\mu}_{t}}(\widebar{\bm{z}}) \leq L_{g}^{2}(\widetilde{\mu}_{t}-\widetilde{\mu}_{t+1})\ (\widebar{\bm{z}} \in \mathcal{Z})$.
  Hence, we get
  $0\leq \moreau{g}{\mu_{2}}(\widebar{\bm{z}}) - \moreau{g}{\mu_{1}}(\widebar{\bm{z}}) = \sum_{t=0}^{l-1}(\moreau{g}{\widetilde{\mu}_{t+1}}(\widebar{\bm{z}}) - \moreau{g}{\widetilde{\mu}_{t}}(\widebar{\bm{z}})) \leq L_{g}^{2}\sum_{t=0}^{l-1}(\widetilde{\mu}_{t}-\widetilde{\mu}_{t+1}) = L_{g}^{2}(\mu_{1}-\mu_{2})$,
  implying the inequality~\eqref{eq:moreau_inequality_Lipschitz} also follows for the case
  $\frac{\mu_{1}}{\mu_{2}} > 2$.
\end{proof}

The Lipschitz continuity of the gradient of a cost function is one of important ingredients for optimization to ensure the descent lemma (see, e.g.,~\cite[Lemma 5.7]{Beck17}).
The following presents a sufficient condition for the Lipschitz constant of the gradient of
$(h+\moreau{g}{\mu}\circ\mathfrak{S})\circ\param$ (see also Remark~\ref{remark:assumption} for its further sufficient conditions).
\begin{proposition}[Lipschitzian of $\nabla ((h+\moreau{g}{\mu}\circ\mathfrak{S})\circ\param)$]\label{proposition:extension:Lipschitz_f}
  Consider a smoothed surrogate function
  $\cost^{\langle \mu \rangle} \circ \param \coloneqq (h+\moreau{g}{\mu}\circ\mathfrak{S})\circ\param\ (\mu\in(0,\eta^{-1}))$
  of
  $\cost \circ \param = (h+g\circ\mathfrak{S})\circ\param$
  in Problem~\ref{problem:origin}, and the following conditions:
  \begin{enumerate}[label={\rm (A\arabic*)}]
    \item \label{enum:bounded_DG}
          The operator norm of
          $(\mathrm{D}\mathfrak{S}(\cdot))^{*}$
          is bounded above by
          $\kappa_{\mathfrak{S}} > 0$
          over the convex hull
          $\mathrm{Conv}(C) \subset \mathcal{X}$
          of
          $C$,
          i.e.,
          \begin{equation}
            (\exists \kappa_{\mathfrak{S}} > 0, \forall \bm{x} \in\mathrm{Conv}(C)) \quad
            \norm{(\mathrm{D}\mathfrak{S}(\bm{x}))^{*}}_{\mathrm{op}} \leq \kappa_{\mathfrak{S}};
          \end{equation}
    \item \label{enum:Lipschitz_DG}
          $(\mathrm{D}\mathfrak{S}(\cdot))^{*}$
          is Lipschitz continuous with a Lipschitz constant
          $L_{\mathrm{D}\mathfrak{S}} > 0$
          over
          $C$,
          i.e.,
          \begin{equation}
            (\exists L_{\mathrm{D}\mathfrak{S}}>0, \forall \bm{x}_{1},\bm{x}_{2} \in C) \ \norm{(\mathrm{D}\mathfrak{S}(\bm{x}_{1}))^{*}-(\mathrm{D}\mathfrak{S}(\bm{x}_{2}))^{*}}_{\mathrm{op}} \leq L_{\mathrm{D}\mathfrak{S}}\norm{\bm{x}_{1} - \bm{x}_{2}}; \label{eq:DG_Lipschitz}
          \end{equation}
    \item \label{enum:bounded_gradh}
          The norm of
          $\nabla h$
          is bounded above by
          $\kappa_{h}> 0$
          over
          $C$,
          i.e.,
          \begin{equation}
            (\exists \kappa_{h}>0, \forall\bm{x} \in C)\ \norm{\nabla h(\bm{x})} \leq \kappa_{h};
          \end{equation}
    \item \label{enum:bounded_DF}
          The operator norm of
          $(\mathrm{D}\param (\cdot))^{*}$
          is bounded above by
          $\kappa_{\param } > 0$
          over
          $\mathcal{Y}$;
    \item \label{enum:Lipschitz_DF}
          $(\mathrm{D}\param (\cdot))^{*}$
          is Lipschitz continuous with a Lipschitz constant
          $L_{\mathrm{D}\param } > 0$
          over
          $\mathcal{Y}$.
  \end{enumerate}
  {
  Then, for
  \thickmuskip=0.2\thickmuskip
  \medmuskip=0.2\medmuskip
  \thinmuskip=0.2\thinmuskip
  \arraycolsep=0.2\arraycolsep
  $\mu \in  (0,2^{-1}\eta^{-1}]$,
  $\cost^{\langle \mu \rangle} \circ \param$
  is continuously differentiable, and
  the following hold:
  }
  \begin{enumerate}[label=(\alph*)]
    \item \label{enum:Lipschitz_g_G}
          If the conditions~\ref{enum:bounded_DG} and~\ref{enum:Lipschitz_DG} hold, then
          $\nabla \cost^{\langle \mu \rangle}$
          is Lipschitz continuous with a Lipschitz constant
          $L_{\nabla \cost^{\langle \mu \rangle}}\coloneqq  L_{\nabla h} + L_{g}L_{\mathrm{D}\mathfrak{S}}  +  \kappa_{\mathfrak{S}}^{2}\mu^{-1} > 0$
          over
          $C$.
    \item \label{enum:Lipschitz}
          If the conditions~\ref{enum:bounded_DG}-\ref{enum:Lipschitz_DF} hold,
          then
          $\nabla (\cost^{\langle \mu \rangle}\circ \param )$
          is Lipschitz continuous with a Lipschitz constant
          $L_{\nabla (\cost^{\langle \mu \rangle}\circ \param )}\coloneqq  \varpi_{1} + \varpi_{2}\mu^{-1} > 0$
          over
          $\mathcal{Y}$,
          where constants
          $\varpi_{1} \coloneqq  \kappa_{\param }^{2}(L_{\nabla h} + L_{g}L_{\mathrm{D}\mathfrak{S}}) +  L_{\mathrm{D}\param }(\kappa_{h}+ \kappa_{\mathfrak{S}}L_{g})$
          and
          $\varpi_{2} \coloneqq  \kappa_{\param }^{2}\kappa_{\mathfrak{S}}^{2}$
          do not depend on
          $\mu$.
  \end{enumerate}
\end{proposition}
\begin{proof}
  See Appendix~\ref{appendix:extension:Lipschitz_f}.
\end{proof}

\begin{remark}[Sufficient conditions for assumptions in Proposition~\ref{proposition:extension:Lipschitz_f}] \label{remark:assumption}
  \
  \begin{enumerate}[label=(\alph*)]
    \item (Sufficient conditions regarding $\mathfrak{S}$) \label{enum:S}
          In a case where
          $\mathfrak{S}$
          is a linear operator,
          the conditions~\ref{enum:bounded_DG} and~\ref{enum:Lipschitz_DG} in Proposition~\ref{proposition:extension:Lipschitz_f} are automatically satisfied.
          In another case where
          $C$
          is compact and
          $\mathfrak{S}$
          is twice continuously differentiable, the conditions~\ref{enum:bounded_DG},~\ref{enum:Lipschitz_DG} and~\ref{enum:bounded_gradh} are automatically satisfied.
    \item (Sufficient conditions regarding $\param $)
          The conditions~\ref{enum:bounded_DF} and~\ref{enum:Lipschitz_DF}
          are automatically satisfied in a case where
          $C$ is a linear subspace of
          $\mathcal{X}$
          and
          $\param $
          is a linear operator.
          Even for the other cases,
          we will present an example of
          $(C,\param )$,
          in Corollary~\ref{corollary:Cayley}
          (see also Lemma~\ref{lemma:Cayley_Lipschitz}),
          achieving the conditions~\ref{enum:bounded_DF} and~\ref{enum:Lipschitz_DF}.
  \end{enumerate}
\end{remark}

From~\eqref{eq:distance_stationary} and Theorem~\ref{theorem:necessary}, we aim to find a stationary point
$\bm{y}^{\star} \in \mathcal{Y}$
of Problem~\ref{problem:origin}
such that
$d\left(\bm{0}, \Lsubdiff  (\cost\circ  \param )(\bm{y}^{\star})\right)=0$
with
$\cost = h+g\circ\mathfrak{S}$.
Theorem~\ref{theorem:stationary} below
presents an upper bound of this distance with the gradient of
$(h+\moreau{g}{\mu}\circ \mathfrak{S})\circ \param$
in~\eqref{eq:necessary_bound} by showing the gradient sub-consistency of $(h+\moreau{g}{\mu}\circ \mathfrak{S})\circ \param$
for
$\cost\circ\param$
(see also Remark~\ref{remark:gradient_consistency}).
\begin{theorem}[Gradient sub-consistency of smoothed surrogate functions]\label{theorem:stationary}
  \
  \begin{enumerate}[label=(\alph*)]
    \item \label{enum:gradient_consistency}
          Let
          $\mathcal{Y}$
          and
          $\mathcal{Z}$
          be Euclidean spaces.
          Let
          $J:\mathcal{Z}\to \mathbb{R}$
          be
          $\eta$-weakly convex,
          and
          $\mathcal{F}:\mathcal{Y}\to \mathcal{Z}$
          be continuously differentiable.
          Then,
          $\moreau{J}{\mu}\circ \mathcal{F}$
          enjoys the {\em gradient sub-consistency property} (see Remark~\ref{remark:gradient_consistency}) for
          $J\circ \mathcal{F}$,
          i.e.,
          \begin{equation}
            (\widebar{\bm{y}} \in \mathcal{Y}) \quad
            \Limsup_{\mathbb{R}_{++}\ni \mu\searrow0,\ \mathcal{Y}\ni \bm{y}\to\widebar{\bm{y}}} \nabla (\moreau{J}{\mu}\circ \mathcal{F})(\bm{y})
            \subset
            \Lsubdiff  (J\circ \mathcal{F})(\widebar{\bm{y}}). \label{eq:gradient_consistency}
          \end{equation}
    \item \label{enum:stationarity}
          Consider a smoothed surrogate function
          $\cost^{\langle \mu \rangle} \circ \param \coloneqq (h+\moreau{g}{\mu}\circ\mathfrak{S})\circ\param\ (\mu\in(0,\eta^{-1}))$
          of
          $\cost \circ \param = (h+g\circ\mathfrak{S})\circ\param$
          in Problem~\ref{problem:origin} (see also Problem~\ref{problem:constrained}).
          Then, we have
          \begin{equation}
            \Limsup_{\mathbb{R}_{++} \ni \mu\searrow 0,\ \mathcal{Y}\ni \bm{y}\to \widebar{\bm{y}}} \nabla (\cost^{\langle \mu \rangle}\circ \param )(\bm{y}) \subset \Lsubdiff  (\cost \circ \param )(\widebar{\bm{y}}). \label{eq:gradient_consistency_problem}
          \end{equation}
          Moreover, if a given sequence
          $(\bm{y}_{n})_{n=1}^{\infty} \subset \mathcal{Y}$
          converges to some
          $\widebar{\bm{y}} \in \mathcal{Y}$,
          and
          $(\mu_{n})_{n=1}^{\infty} \searrow 0$
          with
          $\mu_{n} \in (0,\eta^{-1})$,
          then
          \begin{equation}
            d(\bm{0}, \Lsubdiff  (\cost \circ  \param )(\widebar{\bm{y}}))
            \leq \liminf_{n\to\infty}\norm{\nabla (\cost^{\langle \mu_{n} \rangle}\circ \param )(\bm{y}_{n})}.\label{eq:necessary_bound}
          \end{equation}
  \end{enumerate}
\end{theorem}
\begin{proof}
  See Appendix~\ref{appendix:theorem_stationary}.
\end{proof}

\begin{remark}[Gradient (sub-)consistency] \label{remark:gradient_consistency}
  \
  \begin{enumerate}[label=(\alph*)]
    \item \label{enum:converse_inclusion}
          Although the converse inclusion of~\eqref{eq:gradient_consistency} is not required in this paper, the equality
          \begin{equation}
            (\widebar{\bm{y}}\in \mathcal{Y}) \quad \Limsup_{\mathbb{R}_{++}\ni \mu\searrow0,\ \mathcal{Y}\ni \bm{y}\to\widebar{\bm{y}}} \nabla (\moreau{J}{\mu}\circ \mathcal{F})(\bm{y}) = \Lsubdiff  (J\circ \mathcal{F})(\widebar{\bm{y}}), \label{eq:gradient_consistency_equall}
          \end{equation}
          called
          the {\em gradient consistency property}~\cite{Chen12} of
          $\moreau{J}{\mu}\circ\mathcal{F}$
          for
          $J\circ \mathcal{F}$,
          can also be shown (i) by
          checking that
          $\moreau{J}{\mu}\circ \mathcal{F}$
          {\em epi-converges} to
          $J\circ \mathcal{F}$ (see~\cite[Def. 7.1]{Rockafellar-Wets98} for the definition of epi-convergence);
          and
          (ii) by applying~\cite[Lemma 3.4]{Burke-Hoheisel13} to
          $\moreau{J}{\mu}\circ \mathcal{F}$.
          Indeed, (i) can be verified by~\cite[Thm. 7.4(d)]{Rockafellar-Wets98} and
          $\moreau{J}{\mu}\circ \mathcal{F} \nearrow J\circ \mathcal{F}$
          as
          $\mu \searrow 0$.
    \item
          In a special case where
          $J$
          is convex, we can check that
          $\moreau{J}{\mu}\circ\mathcal{F}$
          in Theorem~\ref{theorem:stationary}~\ref{enum:gradient_consistency}
          has the gradient consistency property for
          $J\circ \mathcal{F}$
          by using~\cite[Cor. 5.8]{Burke-Hoheisel13}.
          However, the discussion in the proof of~\cite[Cor. 5.8]{Burke-Hoheisel13} cannot apply to our case where
          $J$
          is just weakly convex because the proof of~\cite[Cor. 5.8]{Burke-Hoheisel13} relies on the structure of the convexity of
          $J$
          (see Attouch's theorem~\cite[Thm. 12.35]{Rockafellar-Wets98}).
          In order to show the gradient consistency of
          $\moreau{J}{\mu} \circ \mathcal{F}$
          even with
          $J$
          weakly convex, we take a different approach in the proof of Theorem~\ref{theorem:stationary}.
          We note that our smoothed surrogate function
          $(h+\moreau{g}{\mu}\circ \mathfrak{S}) \circ \param$
          with a linear
          $\mathfrak{S}$
          and
          $\param\coloneqq \Id$
          has been employed in the existing variable smoothing algorithms~\cite{Bohm-Wright21,Beck-Rosset23,Peng-Wu-Hu-Deng23}, although the gradient consistency property has not been discussed therein.
    \item
          We can construct another smoothed surrogate function, with gradient consistency property, of
          $\cost\circ\param$
          via the integral convolution of
          $\cost \circ \param$
          and a kernel function, called {\em mollifier} (see~\cite[Exm. 7.19, Thm. 9.67]{Rockafellar-Wets98} or~\cite[Thm. 4.7]{Burke-Hoheisel-Kanzow13}), because
          $\cost\circ\param$
          is subdifferentially regular (see Lemma~\ref{lemma:subdifferentially}).
          However, it is questionable to obtain closed-form expressions
          of this type of smoothed surrogate functions for
          $\cost \circ \param$
          and its gradient, except for simple functions, e.g., the plus function
          $\mathbb{R}\ni x \mapsto \max\{x, 0\}$~\cite{Chen-Mangasarian96}.
          In contrast, the closed-form expressions of our smoothed surrogate function
          $(h+\moreau{g}{\mu}\circ \mathfrak{S}) \circ \param$
          and its gradient are available as computable tools in a case where
          $g$
          is prox-friendly
          (see~\eqref{eq:expression_cost_gradient}).
  \end{enumerate}
\end{remark}

\subsection{Variable Smoothing Algorithm with Convergence Analysis} \label{sec:proposed} \label{sec:convergence}

The inequality~\eqref{eq:necessary_bound} in Theorem~\ref{theorem:stationary}~\ref{enum:stationarity} implies that finding a stationary point
$\bm{y}^{\star} \in \mathcal{Y}$
of Problem~\ref{problem:origin}
satisfying~\eqref{eq:distance_stationary}
can be reduced to the following problem:
\begin{equation} \label{eq:Moreau_optimality}
  \begin{split}
     & {\rm find\ a\ convergent\ sequence}\ (\bm{y}_{n})_{n=1}^{\infty} \subset \mathcal{Y}\ {\rm such\ that}            \\
     & \quad \liminf_{n\to\infty}\norm{\nabla ((h+\moreau{g}{\mu_{n}}\circ \mathfrak{S})\circ \param )(\bm{y}_{n})}  = 0
    \ {\rm with\ } (\mu_{n})_{n=1}^{\infty} (\subset (0,\eta^{-1})) \searrow 0.
  \end{split}
\end{equation}
In this section,
for finding a stationary point of  Problem~\ref{problem:origin} via the problem in~\eqref{eq:Moreau_optimality},
we present a variable smoothing algorithm in Algorithm~\ref{alg:vsmooth} as an extension of the variable smoothing algorithm~\cite{Bohm-Wright21} that has been proposed for Problem~\ref{problem:origin} with
$\param = \Id$
and
a linear operator
$\mathfrak{S}$.
We note that the algorithm~\cite{Bohm-Wright21} is not applicable to Problem~\ref{problem:origin} due to the nonlinearity of
$\mathfrak{S}$
even if
$\param = \Id$.
The proposed algorithm is designed with a time-varying gradient descent update for minimization of
\begin{equation}
  \cost_{n}\circ \param \coloneqq (h+\moreau{g}{\mu_{n}}\circ \mathfrak{S}) \circ \param
  \ \mathrm{with}\ \cost_{n}:\mathcal{X}\to\mathbb{R}:\widebar{\bm{x}}\mapsto (h+\moreau{g}{\mu_{n}}\circ \mathfrak{S})(\widebar{\bm{x}})
\end{equation}
at $n$th update from
$\bm{y}_{n} \in \mathcal{Y}$
to
$\bm{y}_{n+1} \coloneqq  \bm{y}_{n} - \gamma_{n}\nabla ( \cost_{n}\circ \param )(\bm{y}_{n}) \in \mathcal{Y}$
with
$\mu_{n} \searrow 0$
and a stepsize
$\gamma_{n}>0$.
In a case where
$\prox{\mu_{n}g}$
has a closed-form expression,
the values of
$ \cost_{n} \circ \param(\widebar{\bm{y}})$
and
$\nabla (\cost_{n}\circ \param)(\widebar{\bm{y}})$
are available for every
$\widebar{\bm{y}} \in \mathcal{Y}$
(see~\eqref{eq:expression_cost_gradient}).

The index
$\mu_{n} \in (0,(2\eta)^{-1}]$
in Algorithm~\ref{alg:vsmooth} is designed to establish a convergence analysis of Algorithm~\ref{alg:vsmooth} as
\begin{equation}
  (\mu_{n})_{n=1}^{\infty}\subset (0,(2\eta)^{-1}]\
  \mathrm{satisfies}\
  \begin{cases}
     & {\rm(i)}\ \lim_{n\to\infty} \mu_{n} = 0,
    {\rm(ii)}\ \sum_{n=1}^{\infty} \mu_{n} = +\infty, \\
     &
    {\rm(iii)}\ \mu_{n+1}\leq \mu_{n}\ (n\in \mathbb{N}),
  \end{cases}
  \label{eq:nonsummable}
\end{equation}
where
$\eta>0$
is the level of the weak convexity of
$g$
(see Problem~\ref{problem:constrained}).
We have many choices of
$(\mu_{n})_{n=1}^{\infty}$
enjoying~\eqref{eq:nonsummable} as shown in
Example~\ref{ex:nonsummable}.

\begin{algorithm}[t]
  \caption{Variable smoothing for Problem~\ref{problem:origin} (VSmooth)}
  \label{alg:vsmooth}
  \begin{algorithmic}[1]
    \Require
    $\bm{y}_{1}\in \mathcal{Y},\ c\in(0,1),\ (\mu_{n})_{n=1}^{\infty}\subset (0,(2\eta)^{-1}]\ \mathrm{satisfying~\eqref{eq:nonsummable}\ (see\ Example~\ref{ex:nonsummable})}$
    \State
    $n \leftarrow 1$
    \While{the stopping criterion is not satisfied}
    \State
    Set
    $ \cost_{n}\coloneqq  (h+\moreau{g}{\mu_{n}}\circ \mathfrak{S})$
    \State
    Set
    $\gamma_{n} > 0$
    satisfying~\eqref{eq:stepsize} with
    $c$
    (see Example~\ref{ex:stepsize})
    \State
    $\bm{y}_{n+1} \leftarrow \bm{y}_{n}-\gamma_{n}\nabla ( \cost_{n}\circ \param )(\bm{y}_{n})$
    (see~\eqref{eq:expression_cost_gradient} for the expression of $\nabla(\cost_{n}\circ\param)(\bm{y}_{n})$)
    \State
    $n \leftarrow n+1$
    \EndWhile
    \Ensure
    $\bm{y}_{n} \in \mathcal{Y}$
  \end{algorithmic}
\end{algorithm}

\begin{example}[$(\mu_{n})_{n=1}^{\infty}$ satisfying~\eqref{eq:nonsummable}] \label{ex:nonsummable}
  \
  \begin{enumerate}[label=(\alph*)]
    \item \label{enum:ex:nonsummable:exp}
          For any
          $\alpha \geq 1$,
          $(\mu_{n})_{n=1}^{\infty}\coloneqq ((2\eta)^{-1}n^{-1/\alpha})_{n=1}^{\infty}$
          satisfies~\eqref{eq:nonsummable}.
    \item
          $(\mu_{n})_{n=1}^{\infty}\coloneqq\left(\frac{1}{2\eta(n+1)\log(n+1)}\right)_{n=1}^{\infty}$
          satisfies~\eqref{eq:nonsummable}.
  \end{enumerate}
\end{example}
\begin{proof}
  (a)
  Let
  $\alpha \geq 1$
  be chosen arbitrarily.
  Clearly,
  $(\mu_{n})_{n=1}^{\infty} \subset (0,(2\eta)^{-1}]$
  satisfies
  the conditions (i) and (iii) in~\eqref{eq:nonsummable}.
  The condition (ii) can be checked by
  \begin{equation}
    \thickmuskip=0.5\thickmuskip
    \medmuskip=0.5\medmuskip
    \thinmuskip=0.5\thinmuskip
    \arraycolsep=0.5\arraycolsep
    \sum_{n=1}^{K} n^{-1/\alpha}
    = \sum_{n=1}^{K} \int_{n}^{n+1}n^{-1/\alpha}dx
    \geq
    \int_{1}^{K+1} x^{-1/\alpha} dx
    =
    \begin{cases}
      \frac{(K+1)^{1-\alpha^{-1}}- 1}{1-\alpha^{-1}} & (\alpha > 1) \\
      \log(K+1)                                      & (\alpha = 1)
    \end{cases}
    \overset{K\to\infty}{\to} +\infty. \label{eq:delta_example}
  \end{equation}

  (b)
  Clearly,
  $(\mu_{n})_{n=1}^{\infty} \subset (0,(2\eta)^{-1}]$
  satisfies the conditions (i) and (iii).
  The condition (ii) can be checked by
  \begin{align}
     & \sum_{n=1}^{K} \frac{1}{(n+1)\log(n+1)}
    = \sum_{n=1}^{K} \int_{n}^{n+1}\frac{dx}{(n+1)\log(n+1)}
    \geq
    \int_{1}^{K+1} \frac{dx}{(x+1)\log(x+1)}                            \\
     & \overset{u=\log(x+1)}{=} \int_{\log(2)}^{\log(K+2)} \frac{du}{u}
    = \log(\log(K+2)) - \log(\log(2))
    \overset{K\to\infty}{\to} +\infty.
  \end{align}
\end{proof}

Every stepsize
$\gamma_{n} > 0$
in Algorithm~\ref{alg:vsmooth} is chosen to enjoy the so-called {\em Armijo condition} (see, e.g.,~\cite[Thm. 1.1]{Andrei20}) with a predetermined constant
$c \in (0,1)$:
\begin{equation} \label{eq:Armijo}
  \cost_{n}\circ \param (\bm{y}_{n} - \gamma_{n} \nabla ( \cost_{n}\circ \param )(\bm{y}_{n})) \leq \cost_{n}\circ \param  (\bm{y}_{n}) - c\gamma_{n} \norm{\nabla ( \cost_{n}\circ \param )(\bm{y}_{n})}^{2}.
\end{equation}
In a case where
$\gamma_{n}$
satisfies the Armijo condition~\eqref{eq:Armijo}, the descent condition
$ \cost_{n}\circ \param (\bm{y}_{n} - \gamma_{n} \nabla ( \cost_{n}\circ \param )(\bm{y}_{n})) < \cost_{n}\circ \param (\bm{y}_{n})$
is achieved if
$\nabla ( \cost_{n}\circ \param )(\bm{y}_{n})\neq \bm{0}$,
and thus we update
$\bm{y}_{n+1} \coloneqq  \bm{y}_{n} - \gamma_{n} \nabla ( \cost_{n}\circ \param )(\bm{y}_{n})$.

To use
$\gamma_{n}$
satisfying~\eqref{eq:Armijo} in Algorithm~\ref{alg:vsmooth}, we assume Assumption~\ref{assumption:Lipschitz} below.
Sufficient conditions for Assumption~\ref{assumption:Lipschitz}~\ref{enum:gradient_Lipschitz} are found in Proposition~\ref{proposition:extension:Lipschitz_f} and Remark~\ref{remark:assumption}.
Moreover, Example~\ref{ex:stepsize} presents typical choices of
$\gamma_{n}$
achieving Assumption~\ref{assumption:Lipschitz}~\ref{enum:stepsize} under
Assumption~\ref{assumption:Lipschitz}~\ref{enum:gradient_Lipschitz}.

\begin{assumption}\label{assumption:Lipschitz}
  Consider Problem~\ref{problem:origin} and Alg.~\ref{alg:vsmooth}.
  For
  $ \cost_{n} =  h+\moreau{g}{\mu_{n}}\circ\mathfrak{S}$,
  we assume:
  \begin{enumerate}[label=(\alph*)]
    \item (Gradient Lipschitz continuity condition) \label{enum:gradient_Lipschitz}
          {\small
          \begin{align}
             & \hspace{-2em}(\exists \varpi_{1} \in \mathbb{R}_{+}, \exists \varpi_{2} \in \mathbb{R}_{++},  \forall n\in \mathbb{N})                                                                                                                   \\
             & \hspace{-2em} \nabla ( \cost_{n}\circ \param )~ \mathrm{is~Lipschitz~continuous}\mathrm{~with~a~Lipschitz~constant~} L_{\nabla ( \cost_{n}\circ \param )}\coloneqq \varpi_{1}+\frac{\varpi_{2}}{\mu_{n}}. \label{eq:surrogate_Lipschitz}
          \end{align}}
          (Note:
          under the condition~\eqref{eq:surrogate_Lipschitz},
          we have
          $L_{\nabla ( \cost_{n+1}\circ \param )} \geq L_{\nabla ( \cost_{n}\circ \param )} \geq \cdots \geq L_{\nabla ( \cost_{1}\circ \param )}$
          by
          $\mu_{n}\searrow 0$
          [see the condition (iii) in~\eqref{eq:nonsummable}] for $(\mu_{n})_{n=1}^{\infty}$)
    \item (Lower bound condition for stepsizes) \label{enum:stepsize}
          \begin{equation}
            (\exists \beta \in \mathbb{R}_{++}, \forall n\in \mathbb{N})\quad
            \gamma_{n}\in \mathbb{R}_{++}\ \mathrm{satisfies}~\eqref{eq:Armijo}\ \mathrm{and}\ \gamma_{n} \geq \beta L_{\nabla ( \cost_{n}\circ \param )}^{-1}. \label{eq:stepsize}
          \end{equation}
  \end{enumerate}
\end{assumption}

\begin{example}[Stepsize achieving Assumption~\ref{assumption:Lipschitz}~\ref{enum:stepsize}] \label{ex:stepsize}
  Under Assumption~\ref{assumption:Lipschitz}~\ref{enum:gradient_Lipschitz}, we have mainly two choices  of
  $\gamma_{n}$
  achieving~\eqref{eq:stepsize} (see, e.g.,~\cite[Thm. 1.1]{Andrei20}):
  \begin{enumerate}[label=(\alph*)]
    \item \label{enum:ex:stepsize:Lipschitz}
          With the Lipschitz constant
          $L_{\nabla ( \cost_{n}\circ \param )}\in \mathbb{R}_{++}$,
          \begin{equation}
            (n\in \mathbb{N}) \quad \gamma_{n}\coloneqq  \beta L_{\nabla ( \cost_{n}\circ \param )}^{-1}\coloneqq 2(1-c)L_{\nabla ( \cost_{n}\circ \param )}^{-1} \in \mathbb{R}_{++} \label{eq:step_Lipschitz}
          \end{equation}
          satisfies the Armijo condition~\eqref{eq:Armijo} with
          $c \in (0,1)$, and thus~\eqref{eq:stepsize}.
    \item \label{enum:ex:stepsize:backtracking}
          The so-called {\em backtracking algorithm} in Algorithm~\ref{alg:backtracking} returns
          $\gamma_{n}$
          enjoying the Armijo condition~\eqref{eq:Armijo} with
          $c \in (0,1)$, where
          $\gamma_{\rm init} \in \mathbb{R}_{++}$
          is an initial guess for
          $\gamma_{n}$,
          and
          $\rho \in (0,1)$.
          In this case,
          $\gamma_{n}$
          is known to have
          the following lower bound:
          \begin{equation}
            (n\in\mathbb{N}) \quad \gamma_{n} \geq \min\left\{\gamma_{\rm init}L_{\nabla ( \cost_{n}\circ \param )}, 2\rho(1-c)\right\}L_{\nabla ( \cost_{n}\circ \param )}^{-1}
            \geq \beta L_{\nabla ( \cost_{n}\circ \param )}^{-1}
            \label{eq:stepsize_lower_bound}
          \end{equation}
          with
          $\beta \coloneqq\min\left\{\gamma_{\rm init}L_{\nabla ( \cost_{1}\circ \param )}, 2\rho(1-c)\right\} \in \mathbb{R}_{++}$,
          where the first inequality follows by~\cite[Thm. 1.1]{Andrei20}, and the second inequality follows\footnote{
            The derivation of the second inequality is inspired by~\cite[the equality just after (4.12)]{Beck-Rosset23}.
          } by
          $L_{\nabla ( \cost_{n}\circ \param )}   \geq L_{\nabla ( \cost_{1}\circ \param )}$
          (see the sentence just after~\eqref{eq:surrogate_Lipschitz}).
          From~\eqref{eq:stepsize_lower_bound},
          Algorithm~\ref{alg:backtracking} is guaranteed to terminate in at most
          $\max\left\{1, \left\lceil \log_{\rho}\left(\frac{2(1-c)}{\gamma_{\rm init}L_{\nabla (\cost_{n}\circ\param)}}\right)\right\rceil\right\}
            = \max\left\{1, \left\lceil \log_{\rho}\left(\frac{2\mu_{n}(1-c)}{\gamma_{\rm init}(\mu_{n}\varpi_{1}+\varpi_{2})}\right)\right\rceil\right\}$
          updates, and
          $\gamma_{n}$
          given by Algorithm~\ref{alg:backtracking} satisfies~\eqref{eq:stepsize},
          where
          $\lceil \cdot \rceil$
          stands for the ceiling function.
          We note that any knowledge on
          $L_{\nabla ( \cost_{n}\circ \param )}=\varpi_{1}+\varpi_{2}\mu_{n}^{-1}$
          with
          $\varpi_{1},\varpi_{2} \in \mathbb{R}_{+}$
          is not required in Algorithm~\ref{alg:backtracking}.
  \end{enumerate}
\end{example}

\begin{algorithm}[t]
  \caption{Backtracking algorithm to find $\gamma_{n}$ satisfying~\eqref{eq:Armijo}}
  \label{alg:backtracking}
  \begin{algorithmic}
    \Require
    $c \in (0,1), \rho \in (0, 1),\ \gamma_{\rm init} \in \mathbb{R}_{++}$
    \State
    $\gamma_{n} \leftarrow \gamma_{\rm init}$
    \State
    $f_{n} \leftarrow \cost_{n}\circ \param (\bm{y}_{n}) \in \mathbb{R}$,
    $\bm{d}_{n} \leftarrow \nabla ( \cost_{n}\circ \param )(\bm{y}_{n}) \in \mathcal{Y}$
    \While{$ \cost_{n}\circ \param (\bm{y}_{n} - \gamma_{n} \bm{d}_{n}) > f_{n} - c\gamma_{n} \norm{\bm{d}_{n}}^{2} $}
    \State
    $\gamma_{n} \leftarrow \rho \gamma_{n}$
    \EndWhile
    \Ensure
    $\gamma_{n}$
  \end{algorithmic}
\end{algorithm}

We present below a convergence analysis of Algorithm~\ref{alg:vsmooth} under Assumption~\ref{assumption:Lipschitz}.

\begin{theorem}[Convergence analysis of Algorithm~\ref{alg:vsmooth}]\label{theorem:convergence_extension}
  Consider Problem~\ref{problem:origin}.
  Choose arbitrarily
  $\bm{y}_{1} \in \mathcal{Y}$,
  $c \in (0,1)$,
  and
  $(\mu_{n})_{n=1}^{\infty} \subset (0,(2\eta)^{-1}]$
  satisfying~\eqref{eq:nonsummable}.
  Suppose that
  $(\bm{y}_{n})_{n=1}^{\infty} \subset \mathcal{Y}$
  is generated by Algorithm~\ref{alg:vsmooth} under Assumption~\ref{assumption:Lipschitz} with
  $ \cost_{n}\coloneqq \cost^{\langle \mu_{n}\rangle} \coloneqq  h+\moreau{g}{\mu_{n}}\circ\mathfrak{S}\ (n\in \mathbb{N})$.
  Then,
  $ \cost_{n}\circ \param \ (n\in \mathbb{N})$
  and
  $(\bm{y}_{n})_{n=1}^{\infty}$
  enjoy the following:
  \begin{enumerate}[label=(\alph*)]
    \item \label{enum:convergence_2}
          For any pair
          $(\underline{k}, \overline{k}) \in \mathbb{N}^{2}$
          such that
          $\underline{k} \leq \overline{k}$,
          we have
          \begin{equation}
            \min_{\underline{k}\leq n \leq \overline{k}} \norm{ \nabla ( \cost_{n}\circ \param )(\bm{y}_{n})}
            \leq
            \sqrt{\frac{\chi}{\sum_{n=\underline{k}}^{\overline{k}}\mu_{n}}}.
            \label{eq:convergence_rate}
          \end{equation}
          with
          $\chi\coloneqq \frac{(\cost_{1}\circ \param (\bm{y}_{1}) - \inf_{\bm{y}\in\mathcal{Y}} \cost_{1}\circ \param (\bm{y}) + L_{g}^{2}\mu_{1})(\varpi_{1} + \eta\varpi_{2})}{c\beta\eta} \in \mathbb{R}_{++}$,
          where
          (I)
          $L_{g},\eta \in \mathbb{R}_{++}$
          are respectively the Lipschitz constant and the level of the weak convexity of
          $g$ (see the condition~\ref{enum:problem:constrained:g} in Problem~\ref{problem:constrained}),
          and
          (II)
          $\varpi_{1} \in \mathbb{R}_{+}$
          and
          $\varpi_{2},\beta \in \mathbb{R}_{++}$
          are given in Assumption~\ref{assumption:Lipschitz}.
    \item \label{enum:convergence_3}
          \begin{equation}
            \liminf_{n\to\infty}\norm{\nabla ( \cost_{n}\circ \param )(\bm{y}_{n})} = 0. \label{eq:gradient_liminf}
          \end{equation}
    \item \label{enum:convergence_4}
          Choose a subsequence
          $(\bm{y}_{m(l)})_{l=1}^{\infty} \subset \mathcal{Y}$
          of
          $(\bm{y}_{n})_{n=1}^{\infty}$
          such that
          $\lim_{l\to\infty} \norm{\nabla ( \cost_{m(l)}\circ \param )(\bm{y}_{m(l)})} = 0$,
          where
          $m:\mathbb{N}\to\mathbb{N}$
          is monotonically increasing\footnote{
            From Theorem~\ref{theorem:convergence_extension}~\ref{enum:convergence_3}, we can construct such a subsequence
            $(\bm{y}_{m(l)})_{l=1}^{\infty}$
            by using, e.g.,
            \begin{equation}
              m(l) \coloneqq
              \begin{cases}
                1                                                                                                          & (l=1)    \\
                \min \{n \in \mathbb{N} \mid n > m(l-1), \norm{\nabla ( \cost_{n}\circ \param )(\bm{y}_{n})} \leq 2^{-l}\} & (l > 1).
              \end{cases}
            \end{equation}
          }.
          Then, every cluster point
          $\bm{y}^{\star} \in \mathcal{Y}$
          of
          $(\bm{y}_{m(l)})_{l=1}^{\infty}$
          is a stationary point of Problem~\ref{problem:origin}.
          Moreover,
          if
          $N_{ C}(\param(\bm{y}^{\star})) \supset  \mathrm{Ker}((\mathrm{D}\param (\bm{y}^{\star}))^{*})$
          holds, then
          $\param(\bm{y}^{\star})\in C$
          is a stationary point of Problem~\ref{problem:constrained}
          by Theorem~\ref{theorem:necessary}~\ref{enum:necessary_C_F} (see also Corollary~\ref{corollary:optimality_C_F} and Remark~\ref{remark:submersion}).
    \item \label{enum:convergence:rate}
          Consider
          $\mu_{n} \coloneqq \tau n^{-1/\alpha}\ (n\in\mathbb{N})$
          with
          $\tau \in (0,(2\eta)^{-1}]$
          and
          $\alpha > 1$.
          Then,
          $(\mu_{n})_{n=1}^{\infty}$
          satisfies~\eqref{eq:nonsummable} by Example~\ref{ex:nonsummable}~\ref{enum:ex:nonsummable:exp},
          and, for a given tolerance
          $\epsilon \in (0, 1)$,
          the conditions
          $\norm{\nabla (\cost_{n}\circ\param)(\bm{y}_{n})} < \epsilon$
          and
          $\mu_{n} < \epsilon$
          are jointly achieved within at most
          \begin{equation}
            \left\lceil\left(\chi\tau^{-1}(1-\alpha^{-1}) + \tau^{\alpha-1} + 2\right)^{\frac{\alpha}{\alpha-1}} \epsilon^{\min\left\{\frac{-2\alpha}{\alpha-1},-\alpha\right\}}\right\rceil
            \ \mathrm{iterations},
          \end{equation}
          where
          $\chi \in \mathbb{R}_{++}$
          is given in Theorem~\ref{theorem:convergence_extension}~\ref{enum:convergence_2},
          $\lceil \cdot \rceil$
          stands for the ceiling function, and such
          $\bm{y}_{n}$
          is an $\epsilon$-approximate stationary point defined in Definition~\ref{definition:complexity}~\ref{enum:definition:complexity:approximate_stationarity}.
  \end{enumerate}
\end{theorem}
\begin{proof}
  The proof for (a) is inspired partially by that for~\cite[Thm. 4.1]{Bohm-Wright21}.

  (a)
  Since
  $\gamma_{n}>0$
  satisfies the Armijo condition~\eqref{eq:Armijo} by Assumption~\ref{assumption:Lipschitz}~\ref{enum:stepsize}, we have
  \begin{align}
    (n \in \mathbb{N}) \quad
    \cost_{n}\circ \param (\bm{y}_{n+1})
     & = \cost_{n}\circ \param (\bm{y}_{n}-\gamma_{n}\nabla ( \cost_{n}\circ \param )(\bm{y}_{n}))                                              \\
     & \leq \cost_{n}\circ \param (\bm{y}_{n}) - c\gamma_{n}\norm{\nabla (\cost_{n}\circ \param )(\bm{y}_{n})}^{2}. \label{eq:extension:Armijo}
  \end{align}
  By~\eqref{eq:smoothed_uniformly_different_parameter} in Lemma~\ref{lemma:basic} with
  $0 < \mu_{n+1} \leq \mu_{n} < \eta^{-1}$,
  we have
  \begin{equation}
    (n \in \mathbb{N}, \bm{y} \in \mathcal{Y})\quad  \cost_{n+1}\circ \param (\bm{y})
    \overset{\eqref{eq:smoothed_uniformly_different_parameter}}{\leq} \cost_{n}\circ\param (\bm{y}) + L_{g}^{2}(\mu_{n}-\mu_{n+1}).
    \label{eq:extension:f_k_moreau}
  \end{equation}
  By combining~\eqref{eq:extension:Armijo} and~\eqref{eq:extension:f_k_moreau}, we obtain
  \begin{align}
    (n\in \mathbb{N}) \quad &
    \cost_{n+1}\circ \param (\bm{y}_{n+1})
    \overset{\eqref{eq:extension:f_k_moreau}}{\leq} \cost_{n}\circ \param (\bm{y}_{n+1}) + L_{g}^{2}(\mu_{n}-\mu_{n+1})                                                                                                                                     \\
                            & \overset{\eqref{eq:extension:Armijo}}{\leq} \cost_{n}\circ \param (\bm{y}_{n}) - c\gamma_{n}\norm{\nabla (\cost_{n}\circ \param )(\bm{y}_{n})}^{2} + L_{g}^{2}(\mu_{n}-\mu_{n+1}), \label{eq:extension:gradient_inequality_1}
  \end{align}
  and thus
  \begin{equation}
    (n\in \mathbb{N}) \quad
    \cost_{n+1}\circ \param (\bm{y}_{n+1})
    \leq \cost_{n}\circ \param (\bm{y}_{n})  + L_{g}^{2}(\mu_{n}-\mu_{n+1}). \label{eq:extension:gradient_inequality}
  \end{equation}
  By summing~\eqref{eq:extension:gradient_inequality_1} up from
  $n = \underline{k}$
  to
  $\overline{k}$,
  the following inequality holds:
  \begin{align}
    c\sum_{n=\underline{k}}^{\overline{k}} \gamma_{n} \norm{\nabla (\cost_{n}\circ \param )(\bm{y}_{n})}^{2}
     & \leq \cost_{\underline{k}}\circ \param (\bm{y}_{\underline{k}}) - \cost_{\overline{k}+1}\circ \param (\bm{y}_{\overline{k}+1}) + L_{g}^{2}(\mu_{\underline{k}}-\mu_{\overline{k}+1})                                    \\
     & \leq \cost_{\underline{k}}\circ \param (\bm{y}_{\underline{k}}) - \inf_{\bm{y}\in\mathcal{Y}} \cost_{1}\circ \param (\bm{y}) + L_{g}^{2}(\mu_{\underline{k}}-\mu_{\overline{k}+1})                                      \\
     & \leq \cost_{1}\circ \param (\bm{y}_{1}) - \inf_{\bm{y}\in\mathcal{Y}} \cost_{1}\circ \param (\bm{y}) + L_{g}^{2}(\mu_{1}-\mu_{\overline{k}+1})                                                                          \\
     & \leq \cost_{1}\circ \param (\bm{y}_{1}) - \inf_{\bm{y}\in\mathcal{Y}} \cost_{1}\circ \param (\bm{y}) + L_{g}^{2}\mu_{1} \overset{\eqref{eq:smoothed_bounded_below}}{<} +\infty, \label{eq:extension:sum_gradient_bound}
  \end{align}
  where the second inequality follows by
  $\inf_{\bm{y}\in \mathcal{Y}}\cost_{1}\circ \param (\bm{y}) \overset{\eqref{eq:smoothed_uniformly_different_parameter}}{\leq} \inf_{\bm{y}\in \mathcal{Y}}\cost_{\overline{k}+1}\circ \param (\bm{y}) \leq \cost_{\overline{k}+1} \circ \param (\bm{y}_{\overline{k}+1})$,
  and the third inequality follows by using~\eqref{eq:extension:gradient_inequality} recursively.

  By Assumption~\ref{assumption:Lipschitz}, recall that
  (i)
  $\nabla (\cost_{n}\circ \param )$
  is Lipschitz continuous with a Lipschitz constant
  $L_{\nabla (\cost_{n}\circ \param )}=\varpi_{1} + \varpi_{2}\mu_{n}^{-1}$
  with some constants
  $\varpi_{1}\geq 0,\varpi_{2} > 0$;
  (ii)
  every
  $\gamma_{n}$
  is bounded below by
  $\beta L_{\nabla (\cost_{n}\circ \param )}^{-1}$
  with some
  $\beta \in \mathbb{R}_{++}$.
  We have
  \begin{equation}
    (n\in \mathbb{N})\quad \gamma_{n}  \geq \beta L_{\nabla (\cost_{n}\circ \param )}^{-1} = \frac{\beta\mu_{n}}{\varpi_{1}\mu_{n} + \varpi_{2}}
    =\frac{\beta\eta\mu_{n}}{\varpi_{1}\eta\mu_{n} + \eta\varpi_{2}}
    \geq \frac{\beta\eta}{\varpi_{1}+\eta\varpi_{2}}\mu_{n},
  \end{equation}
  where the last inequality follows by
  $\eta\mu_{n} \in (0,2^{-1}]$.
  Then, the LHS in~\eqref{eq:extension:sum_gradient_bound} can be further bounded below as
  \begin{align}
    c\sum_{n=\underline{k}}^{\overline{k}} \gamma_{n} \norm{\nabla (\cost_{n}\circ \param )(\bm{y}_{n})}^{2}
     & \geq \frac{c\beta\eta}{\varpi_{1} + \eta\varpi_{2}}\sum_{n=\underline{k}}^{\overline{k}} \mu_{n} \norm{\nabla (\cost_{n}\circ \param )(\bm{y}_{n})}^{2}                                                                          \\
     & \geq \frac{c\beta\eta}{\varpi_{1} + \eta\varpi_{2}}\min_{\underline{k}\leq n \leq \overline{k}}\norm{\nabla (\cost_{n}\circ \param )(\bm{y}_{n})}^{2} \sum_{n=\underline{k}}^{\overline{k}} \mu_{n} . \label{eq:tmp_convergence}
  \end{align}
  By~\eqref{eq:extension:sum_gradient_bound} and~\eqref{eq:tmp_convergence}, we get the desired inequality in~\eqref{eq:convergence_rate}.

  (b)
  Let
  $a_{n}\coloneqq \norm{\nabla (\cost_{n}\circ \param )(\bm{y}_{n})}$,
  and
  $b_{n} \coloneqq  \inf\{a_{m} \mid m \geq n\}$.
  For every
  $\underline{k} \in \mathbb{N}$,
  the inequality
  \eqref{eq:convergence_rate} with
  $\overline{k}\to\infty$
  implies
  $b_{\underline{k}} = 0$
  thanks to the nonsummability of
  $(\mu_{n})_{n=1}^{\infty}$
  (see the condition (ii) in~\eqref{eq:nonsummable}).
  This implies
  $\liminf_{n\to\infty}\norm{\nabla (\cost_{n}\circ \param )(\bm{y}_{n})} = \liminf_{n\to\infty} a_{n} = \lim_{n\to\infty} b_{n} = 0$.

  (c)
  Let
  $\bm{y}^{\star} \in \mathcal{Y}$
  be a cluster point of
  $(\bm{y}_{m(l)})_{l=1}^{\infty}$.
  By~\eqref{eq:necessary_bound} and Theorem~\ref{theorem:convergence_extension}~\ref{enum:convergence_3}, we have
  $d(\bm{0}, \Lsubdiff  (\cost\circ \param )(\bm{y}^{\star})) = 0$
  with
  $\cost = h+g\circ \mathfrak{S}$,
  implying
  $\bm{y}^{\star}$
  satisfies~\eqref{eq:necessarily_origin}.

  (d)
  By
  $\mu_{n} = \tau n^{-1/\alpha}$,
  we have
  $\mu_{n} < \epsilon$
  for all
  $n \geq \underline{k} \coloneqq \lceil \tau^{\alpha} \epsilon^{-\alpha} \rceil+1$.
  We derive
  $\overline{k} \in \mathbb{N}$
  below such that
  $\overline{k} \geq \underline{k}$
  and
  $\min_{\underline{k}\leq n \leq \overline{k}} \norm{\nabla (\cost_{n}\circ\param)(\bm{y}_{n})} < \epsilon$.
  By~\eqref{eq:convergence_rate}, we have
  \begin{equation}
    \min_{\underline{k}\leq n \leq \overline{k}} \norm{\nabla (\cost_{n}\circ\param)(\bm{y}_{n})}
    \leq \sqrt{\frac{\chi}{\tau\sum_{n=\underline{k}}^{\overline{k}}n^{-1/\alpha}}}
    \leq
    \sqrt{\frac{\chi(1-\alpha^{-1})}{\tau\left((\overline{k}+1)^{1-\alpha^{-1}} - \underline{k}^{1-\alpha^{-1}}\right)}},
  \end{equation}
  where we used
  $\sum_{n=\underline{k}}^{\overline{k}}n^{-1/\alpha} \geq \frac{(\overline{k}+1)^{1-\alpha^{-1}}-\underline{k}^{1-\alpha^{-1}}}{1-\alpha^{-1}}$
  by a discussion similar to~\eqref{eq:delta_example}.
  Hence,
  $\min_{\underline{k}\leq n \leq \overline{k}} \norm{\nabla (\cost_{n}\circ\param)(\bm{y}_{n})} < \epsilon$
  holds if
  $\overline{k}$
  satisfies
  $\sqrt{\frac{\chi(1-\alpha^{-1})}{\tau\left((\overline{k}+1)^{1-\alpha^{-1}} - \underline{k}^{1-\alpha^{-1}}\right)}} < \epsilon$,
  or equivalently
  \begin{equation}
    \overline{k}
    > \left(\chi\tau^{-1}(1-\alpha^{-1})\epsilon^{-2}  + \underline{k}^{1-\alpha^{-1}} \right)^{\frac{\alpha}{\alpha-1}} - 1.
    \label{eq:lower_iteration}
  \end{equation}
  In the following, we derive an upper bound, with
  $\epsilon$,
  of the RHS in~\eqref{eq:lower_iteration}.
  Since
  $(a+b)^{p} \leq a^{p} + b^{p}\ (p \in (0,1), a,b\in \mathbb{R}_{++})$
  holds\footnote{
  $(a+b)^{p} \leq a^{p} + b^{p}\ (p \in (0,1), a,b\in \mathbb{R}_{++})$
  is verified as follows.
  By considering
  $t = b/a > 0$,
  it suffices to show
  $(1+t)^{p} \leq 1 + t^{p}$
  for all
  $t \in \mathbb{R}_{++}$,
  or equivalently
  $\theta(t)\coloneqq 1+t^{p} - (1+t)^{p} \geq 0$
  for all
  $t \in \mathbb{R}_{++}$.
  By
  $\theta'(t) = p(t^{p-1}-(1+t)^{p-1}) > 0\ (t \in \mathbb{R}_{++})$
  from
  $p \in (0,1)$,
  $\theta$
  is monotonically increasing.
  From
  $\theta(0) = 0$,
  we get
  $\theta(t) > 0$
  for all
  $t \in \mathbb{R}_{++}$.
  },
  we get
  $\underline{k}^{1-\alpha^{-1}}
    = (\lceil \tau^{\alpha} \epsilon^{-\alpha} \rceil+1)^{1-\alpha^{-1}}
    \leq (\tau^{\alpha} \epsilon^{-\alpha} + 2)^{1-\alpha^{-1}}
    \leq (\tau^{\alpha} \epsilon^{-\alpha})^{1-\alpha^{-1}} + 2^{1-\alpha^{-1}}
    \leq \tau^{\alpha-1}\epsilon^{1-\alpha} + 2$.
  From
  $\max\{\epsilon^{-2}, \epsilon^{1-\alpha}, 1\} \leq \epsilon^{\min\{-2,1-\alpha\}}$
  by
  $\epsilon \in (0,1)$,
  we get
  $\underline{k}^{1-\alpha^{-1}} \leq (\tau^{\alpha-1} + 2)\epsilon^{1-\alpha}$,
  and
    {
      \thickmuskip=0.5\thickmuskip
      \medmuskip=0.5\medmuskip
      \thinmuskip=0.5\thinmuskip
      \arraycolsep=0.5\arraycolsep
      \begin{align}
         & \left(\chi\tau^{-1}(1-\alpha^{-1})\epsilon^{-2}  + \underline{k}^{1-\alpha^{-1}}\right)^{\frac{\alpha}{\alpha-1}}
        \leq \left(\left(\chi\tau^{-1}(1-\alpha^{-1}) + \tau^{\alpha-1}+2\right)\epsilon^{\min\{-2,1-\alpha\}}\right)^{\frac{\alpha}{\alpha-1}}                                                                  \\
         & \leq \left\lceil\left(\chi\tau^{-1}(1-\alpha^{-1}) + \tau^{\alpha-1} + 2\right)^{\frac{\alpha}{\alpha-1}} \epsilon^{\min\left\{\frac{-2\alpha}{\alpha-1},-\alpha\right\}}\right\rceil\eqqcolon n_{0}.
      \end{align}}%
  Hence,~\eqref{eq:lower_iteration} is achieved by
  $\overline{k}\coloneqq n_{0}$,
  and thus we get
  $\min_{\underline{k}\leq n \leq n_{0}} \norm{\nabla (\cost_{n}\circ\param)(\bm{y}_{n})} < \epsilon$,
  i.e.,
  there exists
  $\widehat{n} \in [\underline{k},n_{0}]\cap \mathbb{N}$
  such that
  $\norm{\nabla (\cost_{\widehat{n}}\circ\param)(\bm{y}_{\widehat{n}})} < \epsilon$.
  Since
  $\mu_{n} < \epsilon$
  holds for all
  $n \geq \underline{k}$,
  the conditions
  $\norm{\nabla (\cost_{n}\circ\param)(\bm{y}_{n})} < \epsilon$
  and
  $\mu_{n} < \epsilon$
  are jointly achieved within at most
  $n_{0}$
  iterations.
\end{proof}

To derive a first-order oracle complexity of Algorithm~\ref{alg:vsmooth}, we introduce the definitions of the first-order oracle
and an $\epsilon$-approximate stationary point, e.g.,~\cite{Garmanjani-Vicente13,Bian-Chen13,Zhang-Chen-Ma23}, for smoothing-type algorithms.

\begin{definition} \label{definition:complexity}
  Consider Problem~\ref{problem:origin} with
  $\cost^{\langle \mu \rangle} = h+\moreau{g}{\mu}\circ\mathfrak{S}$.
  \begin{enumerate}[label=(\alph*)]
    \item \label{enum:definition:complexity:oracle}
          For given
          $\bm{x},\bm{u} \in \mathcal{X}$,
          $\bm{y} \in \mathcal{Y}$,
          $\bm{z},\bm{v} \in \mathcal{Z}$,
          and
          $\mu \in (0,\eta^{-1})$,
          the {\em first-order oracle} returns
          $h(\bm{x})$,
          $\nabla h(\bm{x})$,
          $\mathfrak{S}(\bm{x})$,
          $(\mathrm{D}\mathfrak{S}(\bm{x}))^{*}[\bm{v}]$,
          $\param(\bm{y})$,
          $(\mathrm{D}\param(\bm{y}))^{*}[\bm{u}]$,
          $g(\bm{z})$,
          and
          $\prox{\mu g}(\bm{z})$
          (Note:
          one evaluation of
          $\cost^{\langle\widebar{\mu}\rangle}\circ\param(\widebar{\bm{y}})$ and $\nabla(\cost^{\langle\widebar{\mu}\rangle}\circ\param)(\widebar{\bm{y}})$
          for
          $\widebar{\bm{y}} \in \mathcal{Y}$
          and
          $\widebar{\mu} \in (0,\eta^{-1})$
          can be implemented with at most five oracle calls\footnote{
            By noting that $*$ denotes an arbitrary value,
            consider the following five oracle calls with
            \begin{enumerate}[label=(\roman*)]
              \thickmuskip=0.3\thickmuskip
              \medmuskip=0.3\medmuskip
              \thinmuskip=0.3\thinmuskip
              \arraycolsep=0.3\arraycolsep
              \item
                    $(\bm{x},\bm{u},\bm{y},\bm{z},\bm{v},\mu) = (*, *, \widebar{\bm{y}},*,*,*)$
                    to obtain
                    $\widebar{\bm{x}} \coloneqq \param(\widebar{\bm{y}}) \in\mathcal{X}$;
              \item
                    $(\bm{x},\bm{u},\bm{y},\bm{z},\bm{v},\mu) = (\widebar{\bm{x}}, *, *,*,*,*)$
                    to obtain
                    $\widebar{\xi}_{\rm cost} \coloneqq h(\widebar{\bm{x}})\in\mathbb{R}$,
                    $\widebar{\bm{\xi}}_{\rm grad} \coloneqq \nabla h(\widebar{\bm{x}}) \in \mathcal{X}$,
                    $\widebar{\bm{z}} \coloneqq \mathfrak{S}(\widebar{\bm{x}}) \in \mathcal{Z}$;
              \item
                    $(\bm{x},\bm{u},\bm{y},\bm{z},\bm{v},\mu) = (*, *, *,\widebar{\bm{z}},*,\widebar{\mu})$
                    to obtain
                    $\widebar{\bm{p}} \coloneqq  \prox{\widebar{\mu} g}(\widebar{\bm{z}}) \in \mathcal{Z}$;
              \item
                    $(\bm{x},\bm{u},\bm{y},\bm{z},\bm{v},\mu) = (\widebar{\bm{x}}, *, *,\widebar{\bm{p}},(\widebar{\bm{z}} - \widebar{\bm{p}})/\widebar{\mu},*)$
                    to obtain
                    $\widebar{\zeta}_{\rm cost} \coloneqq g(\widebar{\bm{p}}) \in \mathbb{R}$,
                    $\widebar{\bm{\zeta}}_{\rm grad}\coloneqq \left(\mathrm{D}\mathfrak{S}(\widebar{\bm{x}})\right)^{*}[(\widebar{\bm{z}} - \widebar{\bm{p}})/\widebar{\mu}] \in \mathcal{X}$;
              \item
                    $(\bm{x},\bm{u},\bm{y},\bm{z},\bm{v},\mu) = (*, \widebar{\bm{\xi}}_{\rm grad} + \widebar{\bm{\zeta}}_{\rm grad}, \widebar{\bm{y}},* ,*,*)$
                    to obtain
                    $\nabla (\cost^{\langle \widebar{\mu} \rangle}\circ\param)(\widebar{\bm{y}})=(\mathrm{D}\param(\widebar{\bm{y}}))^{*}[\widebar{\bm{\xi}}_{\rm grad} + \widebar{\bm{\zeta}}_{\rm grad}] \in \mathcal{Y}$,
            \end{enumerate}
            where we can compute
            $\cost^{\langle \widebar{\mu}\rangle} \circ \param(\widebar{\bm{y}})=\widebar{\xi}_{\rm cost} + \widebar{\zeta}_{\rm cost} + \frac{1}{2\widebar{\mu}}\norm{\widebar{\bm{p}}-\widebar{\bm{z}}}^{2} \in \mathbb{R}$
            (Compare with the expression~\eqref{eq:expression_cost_gradient}).
          }).
    \item \label{enum:definition:complexity:approximate_stationarity}
          For a given
          $\epsilon > 0$,
          a point
          $\widehat{\bm{y}} \in \mathcal{Y}$
          is said to be an {\em $\epsilon$-approximate stationary point} of Problem~\ref{problem:origin}
          if
          $\norm{\nabla (\cost^{\langle \mu \rangle}\circ\param)(\widehat{\bm{y}})} \leq \epsilon$
          holds for some
          $\mu \in (0,\epsilon]$.
  \end{enumerate}
\end{definition}

Compared with the oracle in~\cite{Xu-Jian-Liu-So25} proposed for Problem~\ref{problem:constrained} with
$g$
convex, our first-order oracle returns additionally the values
$\param(\bm{y})$
and its derivative
$(\mathrm{D}\param(\bm{y}))^{*}[\bm{u}]$
for
$\bm{y} \in \mathcal{Y}$
and
$\bm{u} \in \mathcal{X}$,
where an evaluation of
$\param$
corresponds to that of
a {\em retraction} used in~\cite{Xu-Jian-Liu-So25} for mapping a tangent vector to a point in the manifold
$C$
(see just after~\eqref{eq:retraction} for a retraction).

Our approximate stationarity notion for Problem~\ref{problem:origin} in Definition~\ref{definition:complexity}~\ref{enum:definition:complexity:approximate_stationarity} is not directly comparable to the notions used in, e.g.,~\cite{Peng-Wu-Hu-Deng23,Beck-Rosset23,Xu-Jian-Liu-So25,Yuan24,Xu-Jiang-Liu-So26},~\cite{Chen-Ma-Man-Anthony-Zhang20,Wang-Liu-Chen-Ma-Xue-Zhao22}, and~\cite{Li-Chen-Deng-Qu-Zhu-Man21}
for Problem~\ref{problem:constrained}.
Instead, our approximate stationarity reduces to another notion of approximate stationarity
used in~\cite{Bohm-Wright21} for Problem~\ref{problem:origin} with
$\param=\Id$
and
$\mathfrak{S}$
linear
(see Remark~\ref{remark:approximate_stationarity}).

\begin{remark}[Relation between approximate stationarities in Definition~\ref{definition:complexity}~\ref{enum:definition:complexity:approximate_stationarity} and in~\cite{Bohm-Wright21}] \label{remark:approximate_stationarity}
  For Problem~\ref{problem:origin},
  another type of an $\epsilon$-approximate stationary point can be defined as a point
  $\widehat{\bm{y}} \in \mathcal{Y}$
  satisfying
  \begin{equation}
    \exists \bm{z} \in \mathcal{Z}\
    \mathrm{s.t.}\
    \dist\left(-\nabla (h\circ\param)(\widehat{\bm{y}}),(\mathrm{D}(\mathfrak{S}\circ \param)(\widehat{\bm{y}}))^{*}\left[\Lsubdiff g(\bm{z})\right]\right)\leq \epsilon\ \mathrm{and}\ \norm{\bm{z} - \mathfrak{S}\circ \param(\widehat{\bm{y}})} \leq \epsilon.\label{eq:approximate_stationary}
  \end{equation}
  In a special case where
  $\param\coloneqq \Id$
  and
  $\mathfrak{S}$
  is linear,
  a variable smoothing algorithm in~\cite{Bohm-Wright21} can find an $\epsilon$-approximate stationary point in the sense of~\eqref{eq:approximate_stationary} within
  $\mathcal{O}(\epsilon^{-3})$
  oracle calls.
  In fact, our approximate stationary point in Definition~\ref{definition:complexity}~\ref{enum:definition:complexity:approximate_stationarity} implies an approximate stationary point in the sense of~\eqref{eq:approximate_stationary} as follows.
  Let
  $\widehat{\bm{y}} \in \mathcal{Y}$
  be an $\epsilon$-approximate stationary point in Definition~\ref{definition:complexity}~\ref{enum:definition:complexity:approximate_stationarity},
  i.e.,
  $\norm{\nabla (\cost^{\langle \mu \rangle}\circ\param)(\widehat{\bm{y}})} \leq \epsilon$
  with some
  $\mu \in (0,\epsilon]$.
  Then, we have
  \begin{align}
         & \dist\left(-\nabla (h\circ\param)(\widehat{\bm{y}}),(\mathrm{D}(\mathfrak{S}\circ \param)(\widehat{\bm{y}}))^{*}\left[\Lsubdiff g\left(\prox{\mu g}\left(\mathfrak{S}\circ\param(\widehat{\bm{y}})\right)\right)\right]\right)                                      \\
    \leq & \dist(- \nabla (h\circ\param)(\widehat{\bm{y}}), (\mathrm{D}(\mathfrak{S}\circ \param)(\widehat{\bm{y}}))^{*}[\nabla \moreau{g}{\mu}(\mathfrak{S}\circ \param(\widehat{\bm{y}}))]) \quad (\because \mathrm{Fact}~\ref{fact:moreau}~\ref{enum:prox_subdifferential}) \\
    =    & \norm{\nabla (h\circ\param)(\widehat{\bm{y}}) + (\mathrm{D}(\mathfrak{S}\circ \param)(\widehat{\bm{y}}))^{*}[\nabla \moreau{g}{\mu}(\mathfrak{S}\circ \param(\widehat{\bm{y}}))]}
    = \norm{\nabla (\cost^{\langle \mu \rangle}\circ\param)(\widehat{\bm{y}})} \leq \epsilon.
  \end{align}
  Moreover,
  with the Lipschitz constant
  $L_{g}$
  of
  $g$,
  $\norm{\prox{\mu g}\left(\mathfrak{S}\circ\param(\widehat{\bm{y}})\right) - \mathfrak{S}\circ \param(\widehat{\bm{y}})} \leq L_{g}\mu$
  holds
  from
  $\norm{\prox{\mu g}(\widebar{\bm{z}}) - \widebar{\bm{z}}} \leq L_{g} \mu\ (\widebar{\bm{z}} \in \mathcal{Z})$
  (see Fact~\ref{fact:moreau}~\ref{enum:gradient_moreau} and~\eqref{eq:moreau_gradient_bound}).
  Hence,
  $\widehat{\bm{y}}$
  is a $(\max\{1,L_{g}\}\epsilon)$-approximate stationary point in the sense of~\eqref{eq:approximate_stationary}.
\end{remark}

In Corollary~\ref{corollary:complexity},
we derive the first-order complexity, i.e., the number of oracle calls,
$\mathcal{O}(\epsilon^{-3})$
of Algorithm~\ref{alg:vsmooth}
for finding an
$\epsilon$-approximate stationary point of Problem~\ref{problem:origin}.
We note that
Corollary~\ref{corollary:complexity} recovers the oracle complexity
$\mathcal{O}(\epsilon^{-3})$
of the existing variable smoothing algorithm~\cite{Bohm-Wright21} for finding an $\epsilon$-approximate stationary point of Problem~\ref{problem:origin}, in the sense of~\eqref{eq:approximate_stationary},
under a special case where
$\param = \Id$
and
$\mathfrak{S}$
is linear
(see Remark~\ref{remark:approximate_stationarity} and Table~\ref{table:nonsmooth} in Section~\ref{sec:introduction}),
whereas the proposed algorithm admits the nonlinearity of
$\mathfrak{S}$.

\begin{corollary}[First-order oracle complexity of Algorithm~\ref{alg:vsmooth}] \label{corollary:complexity}
  Consider Problem~\ref{problem:origin}.
  Choose arbitrarily
  $\bm{y}_{1} \in \mathcal{Y}$,
  $c \in (0,1)$,
  and
  $\mu_{n} \coloneqq \tau n^{-1/3}\ (n\in\mathbb{N})$
  with
  $\tau \in (0,(2\eta)^{-1}]$.
  Then, under Assumption~\ref{assumption:Lipschitz},
  for any given
  $\epsilon \in (0,1)$,
  Algorithm~\ref{alg:vsmooth} can find an $\epsilon$-approximate stationary point of Problem~\ref{problem:origin}
  (i) with at most
  $\mathcal{O}(\epsilon^{-3})$
  first-order oracle calls if
  $\gamma_{n} \coloneqq 2(1-c)L_{\nabla (\cost_{n}\circ \param)}^{-1}\ (n\in\mathbb{N})$
  (see~\eqref{eq:surrogate_Lipschitz} and Example~\ref{ex:stepsize}~\ref{enum:ex:stepsize:Lipschitz}),
  and
  (ii)
  with at most
  $\mathcal{O}(\epsilon^{-3}\log(\epsilon^{-1}))$
  first-order oracle calls if
  $\gamma_{n}$
  is given by Algorithm~\ref{alg:backtracking}
  with
  $\rho \in (0,1)$
  and
  $\gamma_{\rm init} \in \mathbb{R}_{++}$
  (see Example~\ref{ex:stepsize}~\ref{enum:ex:stepsize:backtracking}).
\end{corollary}
\begin{proof}
  Recall that Algorithm~\ref{alg:vsmooth} can find an $\epsilon$-approximate stationary point within
  $\mathcal{O}(\epsilon^{-3})$
  iterations by Theorem~\ref{theorem:convergence_extension}~\ref{enum:convergence:rate} with
  $\alpha\coloneqq 3$.
  Hence, for the case (i)
  $\gamma_{n} \coloneqq 2(1-c)L_{\nabla (\cost_{n}\circ \param)}^{-1}\ (n\in\mathbb{N})$,
  Algorithm~\ref{alg:vsmooth} can find
  an $\epsilon$-approximate stationary point with at most
  $\mathcal{O}(\epsilon^{-3})$
  first-order oracle calls because
  $\cost_{n} \circ \param$
  and
  $\nabla (\cost_{n}\circ \param)$
  are evaluated only once at each iteration respectively (see also Example~\ref{ex:stepsize}~\ref{enum:ex:stepsize:Lipschitz}).

  For the case (ii),
  at $n$th iteration of Algorithm~\ref{alg:vsmooth},
  Algorithm~\ref{alg:backtracking} for estimating stepsize
  $\gamma_{n}$
  terminates within
  $\max\left\{1, \left\lceil \log_{\rho}\left(\frac{2\mu_{n}(1-c)}{\gamma_{\rm init}(\mu_{n}\varpi_{1}+\varpi_{2})}\right)\right\rceil\right\}$
  updates (see Example~\ref{ex:stepsize}~\ref{enum:ex:stepsize:backtracking}).
  Since
  $\cost_{n} \circ \param$
  is computed once at each update of
  $\gamma_{n}$
  in Algorithm~\ref{alg:backtracking},
  $\cost_{n} \circ \param$
  is evaluated for
  $M_{n}\coloneqq \max\left\{1, \left\lceil \log_{\rho}\left(\frac{2\mu_{n}(1-c)}{\gamma_{\rm init}(\mu_{n}\varpi_{1}+\varpi_{2})}\right)\right\rceil\right\}+1$
  times at most in Algorithm~\ref{alg:backtracking} for estimating
  $\gamma_{n}$
  at $n$th iteration of Algorithm~\ref{alg:vsmooth}.
  Hence, the first-order oracle complexity can be obtained by evaluating an upper bound of
  $\sum_{n=1}^{\mathcal{O}(\epsilon^{-3})} M_{n}$.
  Since
  $\log_{\rho}(\cdot)$
  with
  $\rho \in (0,1)$
  is monotonically decreasing
  and
  $\mu_{n}=\tau n^{-1/3} \leq \tau\ (n\in\mathbb{N})$,
  we have
  $\log_{\rho}(\frac{2\mu_{n}(1-c)}{\gamma_{\rm init}(\mu_{n}\varpi_{1}+\varpi_{2})}) \leq
    \log_{\rho}(\frac{2\mu_{n}(1-c)}{\gamma_{\rm init}(\tau\varpi_{1}+\varpi_{2})})
    = \log_{\rho}(\frac{2\tau(1-c)n^{-1/3}}{\gamma_{\rm init}(\tau\varpi_{1}+\varpi_{2})}) = -\frac{1}{3}\log_{\rho}(n) + \Delta = -\frac{\log(n)}{3\log(\rho)}+\Delta$
  with some
  $\Delta \in \mathbb{R}$.
  From
  $\lceil t \rceil \leq t + 1 \leq \abs{t} + 1\ (t\in\mathbb{R})$,
  we get
  $M_{n} \leq \abs{\log_{\rho}(\frac{2\mu_{n}(1-c)}{\gamma_{\rm init}(\mu_{n}\varpi_{1}+\varpi_{2})})} + 1 + 1 \leq \frac{\log(n)}{3\abs{\log(\rho)}} + \abs{\Delta}+2$.
  From
  $\sum_{n=1}^{k}\log(n) \leq \sum_{n=1}^{k}\log(k) = k\log(k)\ (k\in\mathbb{N})$,
  we get
  $\sum_{n=1}^{k} M_{n} \leq \frac{k\log(k)}{3\abs{\log(\rho)}} + k(\abs{\Delta}+2)$.
  By substituting
  $k=\mathcal{O}(\epsilon^{-3})$,
  we get the first-order oracle complexity
  $\mathcal{O}(\epsilon^{-3}\log(\epsilon^{-3})) = \mathcal{O}(\epsilon^{-3}\log(\epsilon^{-1}))$.
\end{proof}

\section{Numerical experiments} \label{sec:experiment}
We demonstrate the efficacy of Algorithm~\ref{alg:vsmooth} with
$\gamma_{n}$
given by Algorithm~\ref{alg:backtracking}\footnote{
  In our preliminary experiments, we examined that Algorithm~\ref{alg:vsmooth} with
  $\gamma_{n}$
  given by Algorithm~\ref{alg:backtracking} achieves faster convergence than Algorithm~\ref{alg:vsmooth} with
  $\gamma_{n}$
  in~\eqref{eq:step_Lipschitz}.
  Therefore, in the numerical experiments, we demonstrate the efficacy of Algorithm~\ref{alg:vsmooth} with
  $\gamma_{n}$
  given by Algorithm~\ref{alg:backtracking}.
} by numerical experiments in scenarios of two applications: (i) the sparse spectral clustering (SSC) and (ii) the sparse principal component analysis (SPCA).
In the SSC application (Section~\ref{sec:SSC}), we present a new formulation of SSC, via Problem~\ref{problem:constrained} with a weakly convex and smooth composition, and demonstrate how this new formulation improves estimation performance of the existing SSC algorithm.
In the SPCA application (Section~\ref{sec:SPCA}), we demonstrate the efficient numerical performance of the proposed Algorithm~\ref{alg:vsmooth} by comparing its empirical results, e.g., convergence speed, with those of fairly standard Riemannian nonsmooth optimization algorithms~\cite{Beck-Rosset23,Li-Chen-Deng-Qu-Zhu-Man21,Chen-Ma-Man-Anthony-Zhang20}
(Note:
the proposed Algorithm~\ref{alg:vsmooth} and the standard algorithms~\cite{Beck-Rosset23,Li-Chen-Deng-Qu-Zhu-Man21,Chen-Ma-Man-Anthony-Zhang20} are applicable to the SPCA problem).
All experiments were performed by MATLAB on MacBook Pro (Apple M3, 16GB).

Both applications are reduced to nonsmooth optimization problems over an embedded submanifold, say the Stiefel manifold
$\St(p,N) \coloneqq  \{\bm{U} \in \mathbb{R}^{N\times p} \mid \bm{U}^{\TT}\bm{U} = \bm{I}_{p}\}\ (p \leq N)$,
of
$\mathbb{R}^{N\times p}$~\cite[Section 3.3.2]{Absil-Mahony-Sepulchre08}.
In order to apply Algorithm~\ref{alg:vsmooth} to such problems, we employ the generalized inverse Cayley transform, denoted by
$\Phi_{\bm{S}}^{-1}$
in~\cite{Kume-Yamada22}, as a parametrization
$\param $
of
$\St(p,N)$:
\begin{equation}
  \param_{\bm{S}}^{\rm Cay}:(\mathcal{Y}\coloneqq )Q_{N,p}\to \St(p,N)\setminus E_{N,p}(\bm{S}):\bm{V}\mapsto \bm{S}(\bm{I}-\bm{V})(\bm{I}+\bm{V})^{-1}\bm{I}_{N\times p} \label{eq:Cayley}
\end{equation}
with
$\bm{S} \in {\rm O}(N)\coloneqq \St(N,N)$,
where
\begin{equation} \label{eq:skew}
  Q_{N,p}\coloneqq \left\{
  \left.\begin{bmatrix} \bm{A} & -\bm{B}^{\TT} \\ \bm{B} & \bm{0} \end{bmatrix}
  \in \mathbb{R}^{N\times N}\;\right|\;\substack{ -\bm{A}^{\TT} = \bm{A} \in \mathbb{R}^{p\times p},\\ \bm{B} \in \mathbb{R}^{(N-p)\times p}}\right\}
  \equiv \mathbb{R}^{Np-p(p+1)/2}
\end{equation}
is the set of skew-symmetric matrices,
$E_{N,p}(\bm{S})\coloneqq  \{\bm{U} \in \St(p,N) \mid \det(\bm{I}_{p} + \bm{I}_{N\times p}^{\TT}\bm{S}^{\TT}\bm{U}) = 0\}$
is called the {\em singular-point set}, and
$\bm{I}_{N\times p} \in \mathbb{R}^{N\times p}$
is the first $p$ column vectors of the
$N$-by-$N$
identity matrix
$\bm{I}$.
Although the image
$\param_{\bm{S}}^{\rm Cay}(Q_{N,p})=\St(p,N)\setminus E_{N,p}(\bm{S})$
is a proper subset of
$\St(p,N)$,
$\St(p,N)\setminus E_{N,p}(\bm{S})$
is open\footnote{
  There exists an open subset
  $O \subset \mathbb{R}^{N\times p}$
  in the Euclidean space
  $\mathbb{R}^{N\times p}$
  such that
  $\St(p,N) \setminus E_{N,p}(\bm{S}) = O\cap \St(p,N)$.
} and dense\footnote{
  For any
  $\bm{U} \in \St(p,N)$,
  there exists
  $(\bm{U}_{n})_{n=1}^{\infty} \subset \St(p,N) \setminus E_{N,p}(\bm{S})$
  such that
  $\lim_{n\to \infty} \bm{U}_{n} = \bm{U}$.
} in
$\St(p,N)$~\cite[Thm. 2.3 (b)]{Kume-Yamada22},
implying thus almost all points in
$\St(p,N)$
can be parameterized in terms of the Euclidean space
$Q_{N,p}$
with
$\param_{\bm{S}}^{\rm Cay}$.

Although
$\param_{\bm{S}}^{\rm Cay}$
in~\eqref{eq:Cayley}
contains the inverse of an
$N$-by-$N$
matrix
$(\bm{I}+\bm{V})^{-1}$,
this inverse need not be formed explicitly if we use
a structured block diagonal orthogonal matrix
$\bm{S} = \diag(\bm{Q}, \bm{I}_{N-p}) \in {\rm O}(N)$
with some
$\bm{Q} \in {\rm O}(p)$.
Indeed,
the evaluation
$\param_{\bm{S}}^{\rm Cay}(\bm{V})\ (\bm{V}\in Q_{N,p})$
with such an
$\bm{S}$
requires
$\mathcal{O}(Np^{2}+p^{3})$
flops (FLoating-point OPerationS), rather than
$\mathcal{O}(N^{3})$
flops~\cite[Sect. 2.2]{Kume-Yamada22}.
The complexity
$\mathcal{O}(Np^{2}+p^{3})$
matches that of standard retractions, e.g., polar decomposition-based, QR decomposition-based, and Cayley transform-based retractions (see~\cite[Sect. 7.3]{Boumal23}), used in Riemannian optimization.
For a given
$\bm{U}\in \St(p,N)$,
we can construct such an orthogonal matrix
$\bm{S} = \diag(\bm{Q}, \bm{I}_{N-p}) \in {\rm O}(N)$
satisfying
$\bm{U} = \param_{\bm{S}}^{\rm Cay}(\bm{V})$
with some
$\bm{V} \in Q_{N,p}$
as
\begin{equation}
  \bm{S} \coloneqq \diag(\bm{Q}_{1}\bm{Q}_{2}^{\TT}, \bm{I}_{N-p})\ \mathrm{with\ a\ singular\mathchar`-value\ decomposition\ of}\
  \bm{U}_{\rm up}=\bm{Q}_{1}\bm{\Sigma}\bm{Q}_{2}^{\TT} \in \mathbb{R}^{p\times p}, \label{eq:Cay_S_initial}
\end{equation}
where
$\bm{U}_{\rm up} \in \mathbb{R}^{p\times p}$
is the upper block matrix of
$\bm{U}$,
$\bm{Q}_{1},\bm{Q}_{2} \in {\rm O}(p)$
and a nonnegative-valued diagonal matrix
$\bm{\Sigma} \in \mathbb{R}^{p\times p}$~\cite[Thm. 2.7]{Kume-Yamada22}.
We note
$\bm{V}\in Q_{N,p}$
achieving
$\bm{U}=\param_{\bm{S}}^{\rm Cay}(\bm{V})$
can be computed efficiently by the inverse mapping of
$\param_{\bm{S}}^{\rm Cay}$~\cite[Sect. 2.2]{Kume-Yamada22}.

Corollary~\ref{corollary:Cayley} below implies that
$\param_{\bm{S}}^{\rm Cay}$
possesses well-suited properties for its application
to Algorithm~\ref{alg:vsmooth} as a parametrization
$\param $.
More precisely, with
$C\coloneqq \St(p,N)$
and
$\param \coloneqq \param_{\bm{S}}^{\rm Cay}$,
(i) Corollary~\ref{corollary:Cayley}~\ref{enum:Cayley_necessary} implies the equivalence of the first-order optimality conditions for Problem~\ref{problem:constrained} and Problem~\ref{problem:origin};
(ii)
Corollary~\ref{corollary:Cayley}~\ref{enum:Cayley_assumption} and Theorem~\ref{theorem:convergence_extension} imply that Algorithm~\ref{alg:vsmooth} is guaranteed to approximate a stationary point of Problem~\ref{problem:origin}.

\begin{corollary}[Properties of $\param_{\bm{S}}^{\rm Cay}$ in~\eqref{eq:Cayley} as parametrization of $\St(p,N)$] \label{corollary:Cayley}
  Consider Problem~\ref{problem:constrained} with
  $C\coloneqq  \St(p,N)$.
  Let
  $\param \coloneqq \param_{\bm{S}}^{\rm Cay}:(\mathcal{Y}\coloneqq )Q_{N,p}\to \St(p,N)\setminus E_{N,p}(\bm{S})$
  in~\eqref{eq:Cayley}
  with arbitrarily chosen
  $\bm{S} \in {\rm O}(N)$.
  Then, the following hold:
  \begin{enumerate}[label=(\alph*)]
    \item \label{enum:Cayley_necessary}
          For
          $\bm{V}^{\star} \in Q_{N,p}$
          and
          $\bm{U}^{\star}\coloneqq  \param (\bm{V}^{\star}) = \param_{\bm{S}}^{\rm Cay}(\bm{V}^{\star}) \in \St(p,N) \setminus E_{N,p}(\bm{S})$,
          \begin{equation}
            \Lsubdiff  (h+g\circ \mathfrak{S} + \iota_{C})(\bm{U}^{\star}) \ni \bm{0}
            \Leftrightarrow \Lsubdiff  ((h+g\circ \mathfrak{S})\circ \param )(\bm{V}^{\star}) \ni \bm{0}.
          \end{equation}
    \item \label{enum:Cayley_assumption}
          If
          $\mathfrak{S}$
          is twice continuously differentiable, then Assumption~\ref{assumption:Lipschitz}~\ref{enum:gradient_Lipschitz} with
          $C\coloneqq \St(p,N)$
          and
          $\param \coloneqq \param_{\bm{S}}^{\rm Cay}$
          is satisfied.
  \end{enumerate}
\end{corollary}
\begin{proof}
  See Appendix~\ref{appendix:Cayley}.
\end{proof}

By virtue of Corollary~\ref{corollary:Cayley}, we reformulate Problem~\ref{problem:constrained} with
$C\coloneqq \St(p,N)$
into Problem~\ref{problem:origin} with
$\param \coloneqq \param_{\bm{S}}^{\rm Cay}$
in~\eqref{eq:Cayley}:
\begin{equation}
  \mathrm{find}\ \bm{V}^{\star}\in \argmin_{\bm{V} \in Q_{N,p}} (\underbrace{h+g\circ \mathfrak{S}}_{\eqqcolon \cost})\circ \param_{\bm{S}}^{\rm Cay}(\bm{V}). \label{eq:problem_reformulated_St}
\end{equation}
In the subsequent numerical experiments, for a given initial guess
$\bm{U}_{1} \in \St(p,N)$,
we applied Algorithm~\ref{alg:vsmooth}
with
$\gamma_{n}$
given by Algorithm~\ref{alg:backtracking} to the problem~\eqref{eq:problem_reformulated_St}
under the setting
\begin{equation}
  (\mu_{n})_{n=1}^{\infty} = ((2\eta)^{-1}n^{-1/3})_{n=1}^{\infty}, c = 2^{-13}, \rho = 0.5, \gamma_{\rm init} = \min\left\{1, \norm{\nabla (\cost_{1}\circ \param_{\bm{S}}^{\rm Cay})(\bm{V}_{1})}_{F}^{-1}\right\}, \label{eq:setting}
\end{equation}
where a choice of $(\mu_{n})_{n=1}^{\infty}$
follows the setting in the oracle complexity in Corollary~\ref{corollary:complexity},
$\eta>0$
is the level of weak convexity of
$g$,
$\norm{\cdot}_{F}$
is the Frobenius norm,
$\cost_{1}$
is defined as in Algorithm~\ref{alg:vsmooth}, and
$(\bm{S},\bm{V}_{1}) \in {\rm O}(N)\times Q_{N,p}$
was chosen to satisfy
$\param_{\bm{S}}^{\rm Cay}(\bm{V}_{1}) = \bm{U}_{1}$
for an initial guess
$\bm{U}_{1} \in \St(p,N)$.
More precisely,
$\bm{S} \in {\rm O(N)}$
was obtained by the construction in~\eqref{eq:Cay_S_initial} with
$\bm{U}\coloneqq \bm{U}_{1}$
(see Remark~\ref{remark:singular} for the choice of
$\bm{S}$).
Here,
$\gamma_{\rm init}$
in~\eqref{eq:setting}
was used as a common initial trial stepsize in Algorithm~\ref{alg:backtracking} at every iteration, and 
the accepted stepsize
$\gamma_{n}$
was determined by Algorithm~\ref{alg:backtracking} with the initial
$\gamma_{\rm init}$.
We fixed
$\gamma_{\rm init}$
during each run in our experiments
in order to ensure the boundedness condition
$\gamma_{n} \geq \beta L_{\nabla (\cost_{n}\circ\param)}^{-1} (\forall n\in\mathbb{N})$
in~\eqref{eq:stepsize} with some
$\beta \in \mathbb{R}_{++}$
(see also Example~\ref{ex:stepsize}~\ref{enum:ex:stepsize:backtracking}).

\begin{remark}[Singular-point issue and the choice of $\bm{S}\in{\rm O}(N)$] \label{remark:singular}
  The choice of 
  $\bm{S} \in {\rm O}(N)$
  in
  $\param_{\bm{S}}^{\rm Cay}$
  in~\eqref{eq:Cayley}
  can affect the performance of Algorithm~\ref{alg:vsmooth}.
  Indeed, if a target stationary point
  $\bm{U}^{\star} \in \St(p,N)$
  is close to the singular-point set
  $E_{N,p}(\bm{S})$
  of
  $\param_{\bm{S}}^{\rm Cay}$,
  then its parameter
  $\bm{V}^{\star} \in Q_{N,p}$
  achieving
  $\bm{U}^{\star} = \param_{\bm{S}}^{\rm Cay}(\bm{V}^{\star})$
  may be far from the origin in
  $Q_{N,p}$,
  and
  $\param_{\bm{S}}^{\rm Cay}$
  can become less sensitive to changes in the parameter far from
  the origin in
  $Q_{N,p}$~\cite[Sect. 3.3]{Kume-Yamada22}.
  This may cause slow convergence of an algorithm applied to problems parameterized by
  $\param_{\bm{S}}^{\rm Cay}$,
  which is known as the singular-point issue.
  In our experiments, we employed
  $\bm{S}$
  obtained by following~\eqref{eq:Cay_S_initial} with the initial guess
  $\bm{U}\coloneqq \bm{U}_{1}$.
  By using this
  $\bm{S}$,
  the initial guess
  $\bm{U}_{1}$
  is placed away from the singular-point set
  $E_{N,p}(\bm{S})$~\cite[Thm. 2.7]{Kume-Yamada22}.
  We kept this
  $\bm{S}$
  fixed during each run in our experiments.
  Adaptive updates of
  $\bm{S}$
  may further improve performance (see, e.g.,~\cite{Kume-Yamada23,Kume-Yamada24A} for smooth optimization), but are left for future work.
\end{remark}

\subsection{Sparse Spectral Clustering (SSC)}\label{sec:SSC}
The spectral clustering (SC)~\cite{Ng-Michael-Weiss01} is one of powerful modern clustering algorithms, where the goal of clustering is to split given data
$(\bm{\xi}_{i})_{i=1}^{N} \subset \mathbb{R}^{d}$
into
$K \leq N$
groups without requiring labels of data.
The SC has been used not only for clustering on graphs~\cite{Dong-Frossard-Vandergheynst-Nefedov12} in graph signal processing, but also for single-cell RNA sequencing~\cite{Wang-Liu-Chen-Ma-Xue-Zhao22}, remote sensing image analysis~\cite{Tasdemir-Yalcin-Yildirim15}, and detecting clusters in networks~\cite{Wang-Lin-Wang17}.

\begin{algorithm}[t]
  \caption{(Sparse) Spectral clustering}
  \label{alg:SC}
  \begin{algorithmic}[1]
    \Require
    Data
    $(\bm{\xi}_{i})_{i=1}^{N} \subset \mathbb{R}^{d}$.
    \State
    Construct an affinity matrix
    $\bm{W} \in \mathbb{R}^{N\times N}$
    whose entries
    $[\bm{W}]_{i,j} \geq 0$
    stand for the similarity between
    $\bm{\xi}_{i}$
    and
    $\bm{\xi}_{j}$.
    \State
    Compute the normalized Laplacian
    $\bm{L} \coloneqq \bm{I} - \bm{D}^{-1/2}\bm{W}\bm{D}^{-1/2} \in \mathbb{R}^{N\times N}$
    by regarding
    $\bm{W}$
    as the adjacency matrix of a certain graph
    $\mathcal{G}$,
    where
    $\bm{D} \in \mathbb{R}^{N\times N}$
    is the degree (diagonal) matrix of
    $\bm{W}$,
    i.e.,
    $[\bm{D}]_{i,i} = \sum_{j=1}^{N}[\bm{W}]_{i,j}$.
    \State
    Compute
    $\bm{U}^{\star} \in \St(K,N)$
    by solving~\eqref{eq:SC} (\eqref{eq:SSC} in case of Sparse Spectral Clustering).\label{lst:line:SC}
    \State
    Form
    $\widehat{\bm{U}}^{\star} \in \mathbb{R}^{N\times K}$
    by normalizing each row
    $\widehat{\bm{u}}^{\star}_{i} \in \mathbb{R}^{K}$
    of
    $\bm{U}^{\star}$
    to length
    $1$.
    \State \label{lst:line:clustering}
    Treat each row
    $\widehat{\bm{u}}^{\star}_{i} \in \mathbb{R}^{K}$
    of
    $\widehat{\bm{U}}^{\star}$
    as a feature vector of
    $\bm{\xi}_{i}$,
    and cluster
    $(\widehat{\bm{u}}_{i}^{\star})_{i=1}^{N}$
    into
    $K$
    clusters by the k-means algorithm.
  \end{algorithmic}
\end{algorithm}

To split given data
$(\bm{\xi}_{i})_{i=1}^{N}$
into
$K$
groups, the SC is executed as in Algorithm~\ref{alg:SC} based on the graph theory (see, e.g.,~\cite{Ulrike07}).
The SC consists of two parts:
(i) constructing an affinity matrix
$\bm{W} \in \mathbb{R}^{N\times N}$
and an affinity graph
$\mathcal{G}$,
e.g., k-nearest neighborhood graph~\cite{Alshmmari-Stavrakakis-Takatsuka21}, of
$(\bm{\xi}_{i})_{i=1}^{N}$;
(ii) decomposing
$\mathcal{G}$
into
$K$
connected subgraphs based on a certain spectral behavior of the normalized Laplacian
$\bm{L} \coloneqq  \bm{I} - \bm{D}^{-1/2}\bm{W}\bm{D}^{-1/2}\in \mathbb{R}^{N\times N}$
of
$\mathcal{G}$,
where
$\bm{D} \in \mathbb{R}^{N\times N}$
is the degree matrix of
$\mathcal{G}$.
More precisely for (ii), the SC exploits the fact that
if
$\bm{L}$
has an eigenvalue
$0$
with the multiplicity
$K$,
then
$\mathcal{G}$
can be decomposed into
$K$
connected subgraphs by clustering the eigenvectors associated with the eigenvalue
$0$ (see, e.g.,~\cite{Ulrike07}).
By leveraging this fact, the SC computes the
$K$
smallest eigenvectors of
$L$,
in the third step of Algorithm~\ref{alg:SC}, which can be characterized as a minimizer of the following optimization problem:
\begin{equation}
  \mathrm{find}\ \bm{U}^{\star} \in \argmin_{\bm{U} \in \St(K,N)} \trace(\bm{U}^{\TT}\bm{L}\bm{U}). \label{eq:SC}
\end{equation}
In the fourth step of Algorithm~\ref{alg:SC}, we normalize each row vector, say
$\widehat{\bm{u}}^{\star}_{i} \in \mathbb{R}^{K}$,
of
$\bm{U}^{\star}$
as
$\norm{\widehat{\bm{u}}^{\star}_{i}}_{2} = 1$.
Finally, we cluster
$(\widehat{\bm{u}}^{\star}_{i})_{i=1}^{N}$
into
$K$
groups by using the k-means algorithm.

In order to improve the numerical performance of the SC, the sparse SC (SSC)~\cite{Lu-Yan-Lin16,Wang-Liu-Chen-Ma-Xue-Zhao22} exploits a prior knowledge that the affinity matrix
$\bm{W}$
and Laplacian
$\bm{L}$
can be block diagonal, in the ideal case where there is no edge between nodes in different groups, by using a certain permutation
$p:\{1,2,\ldots,N\}\to \{1,2,\ldots,N\}$
of indices of
$(\bm{\xi}_{i})_{i=1}^{N}$
as
\begin{equation}
  (\underbrace{\bm{\xi}_{p(1)},\bm{\xi}_{p(2)},\ldots, \bm{\xi}_{p(\mathcal{I}_{1})}}_{1\mathrm{st\ group}}, \underbrace{\bm{\xi}_{p(\mathcal{I}_{1}+1)},\ldots,\bm{\xi}_{p(\mathcal{I}_{1} + \mathcal{I}_{2})}}_{2\mathrm{nd\ group}},\ldots,\underbrace{\bm{\xi}_{p(\sum_{k=1}^{K-1}\mathcal{I}_{k}+1)},\ldots,\bm{\xi}_{p(N)}}_{K\mathrm{th\ group}}),
\end{equation}
where
$\mathcal{I}_{k} \in \mathbb{N}$
denotes the size of $k$th group.
The block diagonality of
$\bm{L}$
indicates the block diagonality of
$\bm{U}^{\star}\bm{U}^{\star\TT} \in \mathbb{R}^{N\times N}$
with
$\bm{U}^{\star} \in \St(K,N)$
in~\eqref{eq:SC}, thereby inducing the sparsity of
$\bm{U}^{\star}\bm{U}^{\star\TT} \in \mathbb{R}^{N\times N}$.
Although the above permutation is not available in general,
$\bm{U}^{\star}\bm{U}^{\star\TT}$
is expected to be sparse even if any permutation is employed.
Therefore, in the SSC,
the sparsity of
$\bm{U}^{\star}\bm{U}^{\star\TT}$
is promoted via solving the following problem in place of~\eqref{eq:SC}:
\begin{equation}
  \mathrm{find}\ \bm{U}^{\star} \in \argmin_{\bm{U} \in \St(K,N)} \trace(\bm{U}^{\TT}\bm{L}\bm{U}) + \lambda \psi(\bm{U}\bm{U}^{\TT}), \label{eq:SSC}
\end{equation}
where
$\lambda > 0$
and
$\psi:\mathbb{R}^{N\times N}\to \mathbb{R}$
is a sparsity promoting function, e.g.,
$\ell_{1}$-norm, MCP, and SCAD (see, e.g.,~\cite{Bauschke-Combettes17,Chierchia-Chouzenoux-Combettes-Pesquet}).

By Corollary~\ref{corollary:Cayley}, the problem~\eqref{eq:SSC} can be reformulated as
the problem~\eqref{eq:problem_reformulated_St} with
$h(\bm{U})\coloneqq  \trace(\bm{U}^{\TT}\bm{L}\bm{U})$,
$\mathfrak{S}(\bm{U})\coloneqq \bm{U}\bm{U}^{\TT}$
and
$g\coloneqq  \lambda \psi$.
By employing a Lipschitz continuous and weakly convex function as
$\psi$,
we can approximate iteratively a stationary point of~\eqref{eq:SSC} by Algorithm~\ref{alg:vsmooth} (see Theorem~\ref{theorem:convergence_extension} with Corollary~\ref{corollary:Cayley}), where we employed parameters in the paragraph containing~\eqref{eq:setting}.
We note that (i) standard Riemannian nonsmooth optimization algorithms, e.g.,~\cite{Chen-Ma-Man-Anthony-Zhang20,Beck-Rosset23,Peng-Wu-Hu-Deng23}, cannot be applied to the problem~\eqref{eq:SSC} because
$\mathfrak{S}$
is nonlinear;
(ii) \cite{Wang-Liu-Chen-Ma-Xue-Zhao22} reported a prox-linear method for the problem~\eqref{eq:SSC} in a case where
$\psi$
is the convex $\ell_{1}$-norm.
In order to examine how much weakly convex regularizers can enhance the performance, we compared (i) the proposed SSC (SSC+MCP) by employing MCP~\cite{Zhang10}
\begin{align}
  (\bm{X}\in \mathbb{R}^{N\times N})\quad & \psi^{\rm MCP}_{\theta}(\bm{X})\coloneqq \sum_{1\leq i \leq N}\sum_{1\leq j \leq N} r^{\rm MCP}_{\theta}([\bm{X}]_{i,j}) \\
                                          & \mathrm{with}\
  r^{\rm MCP}_{\theta}([\bm{X}]_{i,j})\coloneqq
  \begin{cases}
    \abs{[\bm{X}]_{i,j}} - \frac{[\bm{X}]_{i,j}^{2}}{2\theta} & (\mathrm{if}\ \abs{[\bm{X}]_{i,j}}\leq \theta) \\
    \frac{\theta}{2}                                          & (\mathrm{otherwise})
  \end{cases} \label{eq:MCP}
\end{align}
with a parameter
$\theta > 0$
with (ii) the proposed SSC (SSC+$\ell_{1}$) by employing
$\ell_{1}$-norm
$\norm{\bm{X}}_{1}\coloneqq \sum_{1\leq i\leq N}\sum_{1\leq j\leq N}\abs{[\bm{X}]_{i,j}}\ (\bm{X} \in \mathbb{R}^{N\times N})$
as
$\psi$
(for the closed-form expressions of proximity operators of
$\norm{\cdot}_{1}$
and
$\psi^{\rm MCP}_{\theta}$,
see~\cite[Exm. 24.22]{Bauschke-Combettes17} and~\cite{Bohm-Wright21} respectively).
Since
$ \psi^{\rm MCP}_{\theta}$
is $\theta^{-1}$-weakly convex,
$\lambda\psi^{\rm MCP}_{\theta}$
is $\lambda\theta^{-1}$-weakly convex.
In order to use a common
$(\mu_{n})_{n=1}^{\infty}$
for all SSC+MCP instances with a given
$\theta$,
we set
$\eta \coloneqq  \theta^{-1}$
in~\eqref{eq:setting},
because
$\lambda \in (0,1]$
was employed in this experiment as shown below; hence
$\lambda\psi^{\rm MCP}_{\theta}$
is
$\theta^{-1}$-weakly convex.
For SSC+$\ell_{1}$,
although the minimal level of weak convexity of
$\lambda \norm{\cdot}_{1}$
is zero, every convex function is also
$\eta$-weakly convex for every
$\eta > 0$.
We therefore set the positive upper bound
$\eta \coloneqq 1$,
for simplicity,
in the choice of
$(\mu_{n})_{n=1}^{\infty}$
in~\eqref{eq:setting}.
Since MCP can alleviate underestimation of desired sparse target compared with $\ell_{1}$-norm (see, e.g.,~\cite{Selesnick17,Abe-Yamagishi-Yamada20,Yata-Yamagishi-Yamada22}), SSC+MCP is expected to achieve better performance than SSC+$\ell_{1}$.

Moreover, we compared (i) SSC+MCP and (ii) SSC+$\ell_{1}$ with (iii) the SC~\cite{Ng-Michael-Weiss01}, which solves the problem~\eqref{eq:SC}; (iv) the SSC~\cite{Lu-Yan-Lin16} denoted by SSC[relax]\footnote{
  We used the code available at
  https://github.com/canyilu/LibADMM-toolbox.
}, which solves a certain convex relaxation of~\eqref{eq:SSC} as
\begin{equation}
  \mathrm{find}\ \bm{P}^{\star} \in \argmin_{\bm{P} \in \mathbb{R}^{N\times N}, \bm{0} \preceq \bm{P} \preceq \bm{I}, \trace(\bm{P}) = K} \trace(\bm{P}^{\TT}\bm{L}) + \lambda \norm{\bm{P}}_{1}, \label{eq:SSC_convex}
\end{equation}
where
$\{\bm{U}\bm{U}^{\TT} \in \mathbb{R}^{N\times N} \mid \bm{U}\in \St(K,N)\}$
is relaxed to the convex set
$\{\bm{P} \in \mathbb{R}^{N\times N} \mid \bm{0} \preceq \bm{P} \preceq \bm{I}, \trace(\bm{P}) = K\}$.
The problem~\eqref{eq:SSC_convex} can be solved by classical convex optimization algorithms, e.g., ADMM (see, e.g.,~\cite{Condat-Kitahara-Contreras-Hirabayashi23}).
In the SSC[relax]~\cite{Lu-Yan-Lin16},
$\bm{U}^{\star}$
in the third step of Algorithm~\ref{alg:SC} is obtained as first
$K$ eigenvectors corresponding to the largest
$K$
eigenvalues of
$\bm{P}^{\star}$
in~\eqref{eq:SSC_convex}.

To evaluate clustering performance, we computed two standard criteria: the Normalized Mutual Information (NMI)~\cite{Strehl-Ghosh03} and the Adjusted Rand Index (ARI)~\cite{Hubert-Arabie85}.
These scores closer to one indicate better clustering performance.

We evaluated performances of algorithms by using 7 real-world datasets in UCI Machine Learning Repository~\cite{Dua-Graff17}:
(i) ``iris''; (ii) ``shuttle'' (chosen $1500$ samples randomly); (iii) ``segmentation'';
(iv) ``breast cancer'';
(v) ``glass'';
(vi) ``wine'';
and (vii) ``seeds'',
where the number
$K$
of clusters for each dataset was assumed to be known.
For each dataset, we made an affinity matrix
$\bm{W} \in \mathbb{R}^{N\times N}$
of data
$(\bm{\xi}_{i})_{i=1}^{N}$
by following~\cite{Alshmmari-Stavrakakis-Takatsuka21}\footnote{
  We used the code available at
  https://github.com/mashaan14/Spectral-Clustering.
} as the first step of Algorithm~\ref{alg:SC}.
For SSC[relax], SSC+$\ell_{1}$ and SSC+MCP, parameters
$\lambda$
for regularizers
$\ell_{1}$-norm
and MCP, and
$\theta$
in MCP in~\eqref{eq:MCP} were chosen from
$\{10^{-i} \mid i = 0,1,2,3,4,5,6\}$
to achieve the highest value
$(\mathrm{NMI}+\mathrm{ARI})/2$.
All algorithms except for SC were terminated when running CPU time exceeded
$120$ seconds or the iteration number exceeded
$10000$.

Table~\ref{table:SSC_nmi} shows the averaged NMI and ARI by running k-means algorithm in the fifth step of Algorithm~\ref{alg:SC} for $100$ times.
The values in bold indicate the best results among all algorithms.
We see that three SSCs almost achieve higher NMIs and ARIs than those of SC, implying thus three SSCs outperform SC.
We also observe that SSC+$\ell_1$ is competitive to SSC[relax].
Moreover, among all four algorithms, SSC+MCP achieves the highest NMIs (except for ``wine'') and ARIs as we expected.
In particular, for ``breast cancer'', SSC+MCP overwhelms the others.
From these results, the proposed SSC with a weakly convex regularizer (MCP) has great potential to improve the numerical performance of SSC.

\begin{table}[!ht]
  \centering
  \footnotesize
  \caption{Averaged NMI and ARI of each algorithm.}
  \begin{tabular}{ccccccccc}
    \toprule
                         & ~            & iris      & shuttle   & segmentation & breast cancer & glass     & wine      & seeds     \\ \midrule
    \multirow{4}{*}{NMI} &
    SC                   & 0.778        & 0.435     & 0.501     & 0.417        & 0.321         & \bf 0.433 & 0.662                 \\
                         & SSC[relax]   & 0.785     & 0.485     & 0.503        & 0.433         & 0.322     & \bf 0.433 & 0.671     \\
                         & SSC+$\ell_1$ & 0.785     & 0.486     & 0.503        & 0.433         & 0.323     & \bf 0.433 & 0.667     \\
                         & SSC+MCP      & \bf 0.794 & \bf 0.496 & \bf 0.507    & \bf 0.514     & \bf 0.331 & 0.432     & \bf 0.698 \\ \midrule
    \multirow{4}{*}{ARI} &
    SC                   & 0.745        & 0.231     & 0.341     & 0.419        & 0.174         & 0.363     & 0.659                 \\
                         & SSC[relax]   & 0.786     & 0.315     & 0.343        & 0.462         & 0.174     & 0.363     & 0.675     \\
                         & SSC+$\ell_1$ & 0.786     & 0.317     & 0.343        & 0.462         & 0.175     & 0.363     & 0.668     \\
                         & SSC+MCP      & \bf 0.794 & \bf 0.397 & \bf 0.352    & \bf 0.595     & \bf 0.181 & \bf 0.388 & \bf 0.709 \\ \bottomrule
  \end{tabular}
  \label{table:SSC_nmi}
\end{table}

\subsection{Sparse Principal Component Analysis (SPCA)} \label{sec:SPCA}
We consider the SPCA (see, e.g.,~\cite{Chen-Ma-Man-Anthony-Zhang20,Beck-Rosset23,Peng-Wu-Hu-Deng23}) formulated as
\begin{equation}
  \mathrm{find}\ \bm{U}^{\star} \in \argmin_{\bm{U}\in \St(p,N)} \underbrace{-\trace(\bm{U}^{\TT}\bm{\Xi}^{\TT}\bm{\Xi}\bm{U})}_{\eqqcolon h(\bm{U})} + \underbrace{\lambda \norm{\bm{U}}_{1}}_{\eqqcolon g(\bm{U})} \label{eq:SPCA}
\end{equation}
where
$\bm{\Xi} \in \mathbb{R}^{\mathcal{I} \times N}$
is a data matrix,
$\mathcal{I} \in \mathbb{N}$
is the number of data,
$N$
is the dimension of each data, and
$\lambda > 0$
is a predetermined weight.
The SPCA in~\eqref{eq:SPCA} is a special instance of Problem~\ref{problem:constrained} with
$C\coloneqq \St(p,N)$
and
$\mathfrak{S}=\mathrm{Id}$.
By Corollary~\ref{corollary:Cayley}, the SPCA in~\eqref{eq:SPCA} can be reformulated
as the problem~\eqref{eq:problem_reformulated_St}.
In this experiment, we applied Algorithm~\ref{alg:vsmooth} (VSmooth) to the reformulated problem, where we employed parameters in the paragraph containing~\eqref{eq:setting}.
For the parameter
$\eta$
in~\eqref{eq:setting},
we employed
$\eta \coloneqq  1$,
for the same reason as in the SSC+$\ell_{1}$ experiment, i.e., the convex function
$\lambda\norm{\cdot}_{1}$
is also $1$-weakly convex.

We compared the numerical performance of VSmooth with that of three Riemannian nonsmooth optimization algorithms  applied to the problem in~\eqref{eq:SPCA}: (i) the Riemannian smoothing gradient algorithm~\cite[Section 4.1]{Beck-Rosset23} (RSmooth),
(ii)
the Riemannian subgradient method~\cite[(2.6)]{Li-Chen-Deng-Qu-Zhu-Man21} (RSub), and
(iii)
the Riemannian proximal gradient method~\cite[Algorithm 2]{Chen-Ma-Man-Anthony-Zhang20} (ManPGAda\footnote{
  We used the code available at
  https://github.com/chenshixiang/ManPG.
}).
For each Riemannian optimization algorithm,
the update from
$\bm{U}_{n} \in \St(p,N)$
to
$\bm{U}_{n+1} \in \St(p,N)$
is expressed as
\begin{equation}
  \bm{U}_{n+1} \coloneqq  R_{\bm{U}_{n}}(\gamma_{n}\bm{\mathcal{D}}_{n}) \label{eq:retraction}
\end{equation}
with a stepsize
$\gamma_{n} >0$
and a search direction
$\bm{\mathcal{D}}_{n} \in T_{\St(p,N)}(\bm{U}_{n})\coloneqq \{\bm{\mathcal{D}}\in \mathbb{R}^{N\times p} \mid \bm{U}_{n}^{\TT}\bm{\mathcal{D}} + \bm{\mathcal{D}}^{\TT}\bm{U}_{n} = \bm{0}\}$
in the tangent space
$T_{\St(p,N)}(\bm{U}_{n})$
(see Example~\ref{ex:normal}),
where
$R_{\bm{U}_{n}}:T_{\St(p,N)}(\bm{U}_{n}) \to \St(p,N)$
is a {\em retraction} satisfying
$R_{\bm{U}_{n}}(\bm{0}) = \bm{U}_{n}$
and
$\mathrm{D}R_{\bm{U}_{n}}(\bm{0})[\bm{\mathcal{D}}] = \bm{\mathcal{D}}\ (\bm{\mathcal{D}} \in T_{\St(p,N)}(\bm{U}_{n}))$
(see, e.g.,~\cite{Absil-Mahony-Sepulchre08,Sato21,Boumal23}).
For these algorithms, we employed the polar decomposition-based retraction~\cite{Absil-Mahony-Sepulchre08}
$R_{\bm{U}}^{\rm polar}(\bm{\mathcal{D}}) \coloneqq  (\bm{U}+\bm{\mathcal{D}})(\bm{I}_{p}+\bm{\mathcal{D}}^{\TT}\bm{\mathcal{D}})^{-\frac{1}{2}}\ (\bm{U}\in \St(p,N), \bm{\mathcal{D}} \in T_{\St(p,N)}(\bm{U}))$
as the retraction
$R_{\bm{U}}$
because
$R_{\bm{U}}^{\rm polar}$
is the retraction that was used in~\cite{Chen-Ma-Man-Anthony-Zhang20,Li-Chen-Deng-Qu-Zhu-Man21,Beck-Rosset23}.
Every search direction
$\bm{\mathcal{D}}_{n}$ in~\eqref{eq:retraction}
is given respectively by:
\begin{align}
  \bm{\mathcal{D}}_{n}^{\rm RSmooth}  & \coloneqq  - P_{T_{\St(p,N)}(\bm{U}_{n})}(\nabla (h+\moreau{g}{\mu_{n}})(\bm{U}_{n}))\ \mathrm{with}\ \mu_{n} \searrow 0; \label{eq:RSmooth}                                                                  \\
  \bm{\mathcal{D}}_{n}^{\rm RSub}     & \coloneqq  -P_{T_{\St(p,N)}(\bm{U}_{n})}(\nabla h(\bm{U}_{n}) + \bm{\mathcal{G}})\  \mathrm{with}\ \bm{\mathcal{G}} \in \Lsubdiff  g(\bm{U}_{n}); \label{eq:RSub}                                             \\
  \bm{\mathcal{D}}_{n}^{\rm ManPGAda} & \coloneqq  \argmin_{\bm{\mathcal{D}} \in T_{\St(p,N)}(\bm{U}_{n})} \inprod{\nabla h(\bm{U}_{n})}{\bm{\mathcal{D}}} + \frac{1}{2t_{n}}\norm{\bm{\mathcal{D}}}_{F}^{2} + g(\bm{U}_{n} + \bm{\mathcal{D}}) \\
                                      & \hspace{3em}\mathrm{with}\ t_{n} > 0, \label{eq:step_manpgada}
\end{align}
where
$P_{T_{\St(p,N)}(\bm{U}_{n})}:\mathbb{R}^{N\times p} \to T_{\St(p,N)}(\bm{U}_{n})$
is the orthogonal projection mapping onto
$T_{\St(p,N)}(\bm{U}_{n})$ (see, e.g.,~\cite{Absil-Mahony-Sepulchre08}),
and
$t_{n} > 0$
for ManPGAda~\cite{Chen-Ma-Man-Anthony-Zhang20} is adaptively designed to accelerate convergence speed.
Clearly, ManPGAda requires some iterative solver for the subproblem~\eqref{eq:step_manpgada} to find a search direction
while the others do not require, i.e., the others are single-loop algorithms.
For RSmooth, we employed
$(\mu_{n})_{n=1}^{\infty} = ((2\eta)^{-1}n^{-1/3})_{n=1}^{\infty}$
with
$\eta\coloneqq 1$,
because RSmooth with this
$(\mu_{n})_{n=1}^{\infty}$
achieves the best convergence rate~\cite[Remark 4.1]{Beck-Rosset23}.
In this experiment, to find a stepsize
$\gamma_{n}$,
we used backtracking algorithms for ManPGAda and RSmooth suggested in~\cite{Chen-Ma-Man-Anthony-Zhang20} and~\cite{Beck-Rosset23} respectively, and we employed
$\gamma_{n} \coloneqq  0.99^{n}$
for RSub.

For the problem~\eqref{eq:SPCA} with each
$N \in \{200, 500, 1000\}$,
$p \in \{1, N/100, N/10\}$
and
$\lambda = 0.1$,
we tested performance of all algorithms with
$10$ random
$\bm{\Xi} \in \mathbb{R}^{\mathcal{I}\times N}$
with
$\mathcal{I}= 5000$
generated by the following procedure:
(i) generate
$\bm{\Xi} \in \mathbb{R}^{\mathcal{I} \times N}$
from the standard normal distribution;
(ii) shift the columns of
$\bm{\Xi}$
to have zero means;
(iii) normalize
$\bm{\Xi}$
such that
$\norm{\bm{\Xi}}_{F} = 1$.
For each trial, all algorithms started with the same initial guess
$\bm{U}_{1} \in \St(p,N)$
generated by a MATLAB code ``orth(randn(N,p))''.
Each algorithm was terminated at the $n$th iteration
when the elapsed CPU time at the end of $n$th iteration exceeded
$t(>0)$ seconds (see Table~\ref{table:CPU} for the value of $t$).
\begin{table}[t]
  \caption{Time limits $t$ (s) used for termination}\label{table:CPU}
  \centering
  \begin{tabular}[c]{cccc}
    \toprule
    \multicolumn{1}{c}{}        &
    \multicolumn{1}{c}{p=1}     &
    \multicolumn{1}{c}{p=N/100} &
    \multicolumn{1}{c}{p=N/10}                    \\
    \midrule
    N=200                       & 0.5 & 0.5 & 1.5 \\
    N=500                       & 1.5 & 1.5 & 4.5 \\
    N=1000                      & 3.0 & 3.0 & 9.0 \\
    \bottomrule
  \end{tabular}
\end{table}

Figure~\ref{fig:SPCA_01}
demonstrates the convergence histories for the problem~\eqref{eq:SPCA}, where the plots show CPU time on the horizontal axis versus the averaged cost function value on the vertical axis.
The detailed results are demonstrated in Table~\ref{table:SPCA_01},
where for each final estimate
$\bm{U}^{\diamond} \in \St(p,N)$,
'fval' means the averaged value
$f(\bm{U}^{\diamond})$,
'feasi' means the averaged feasibility error
$\norm{\bm{I}_{p}-\bm{U}^{\diamond\TT}\bm{U}^{\diamond}}_{F}$
of
$\bm{U}^{\diamond}$
from the constraint set
$\St(p,N)$,
'itr' means the averaged number of iterations, 'time' means the CPU time (s), and
'sparsity' means the averaged ratio of entries in
$\bm{U}^{\diamond}$
such that
$\abs{[\bm{U}^{\diamond}]_{i,j}} < 10^{-4}$.

From Figure~\ref{fig:SPCA_01},
we see that the proposed VSmooth converges faster than RSmooth and RSub for all problem sizes.
For
$p \in \{1,N/100\}$,
VSmooth seems to be competitive or inferior to ManPGAda.
However,
for
$N\in \{500,1000\}$
and large
$p = N/10$,
the performance of ManPGAda severely deteriorates while no such deterioration is observed for VSmooth.
Moreover, from Table~\ref{table:SPCA_01},
we also observe that ManPGAda updates for only a few iterations;
and
CPU time of ManPGAda significantly exceeds the time limitation shown in Table~\ref{table:CPU}.
These observations imply that ManPGAda took a long time to find an approximate solution of a subproblem in~\eqref{eq:step_manpgada} at every iteration,
and thus single-loop algorithms, namely, VSmooth, RSmooth and RSub, are more computationally reliable than ManPGAda for large problem sizes.
From this experimental result, VSmooth shows numerical performance comparable to the fairly standard Riemannian nonsmooth optimization algorithms~\cite{Chen-Ma-Man-Anthony-Zhang20,Li-Chen-Deng-Qu-Zhu-Man21,Beck-Rosset23} even for Problem~\ref{problem:constrained} with a simple case, i.e.,
$g$
is convex and
$\mathfrak{S} = \mathrm{Id}$,
where the proposed VSmooth can be applied to Problem~\ref{problem:constrained}, via Problem~\ref{problem:origin}, with more general setting of the nonlinearity of
$\mathfrak{S}$
as seen in Section~\ref{sec:SSC}.

\begin{figure}[t]
  \centering
  \subfloat[][$(N,p) = (200,1)$]{
    \includegraphics[clip, width=0.30\textwidth]{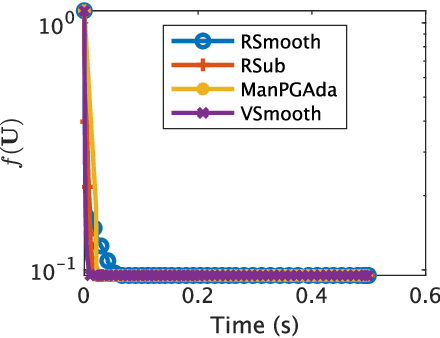}
  }
  \subfloat[][$(N,p) = (200,2)$]{
    \includegraphics[clip, width=0.30\textwidth]{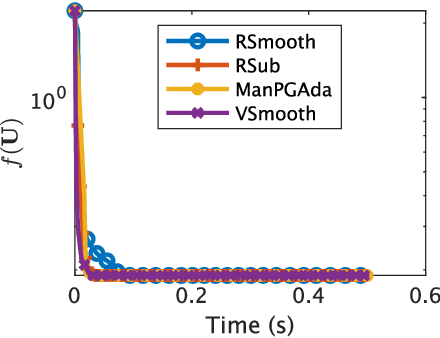}
  }
  \subfloat[][$(N,p) = (200,20)$]{
    \includegraphics[clip, width=0.30\textwidth]{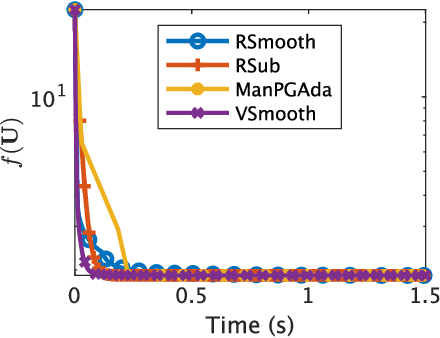}
  } \\
  \subfloat[][$(N,p) = (500,1)$]{
    \includegraphics[clip, width=0.30\textwidth]{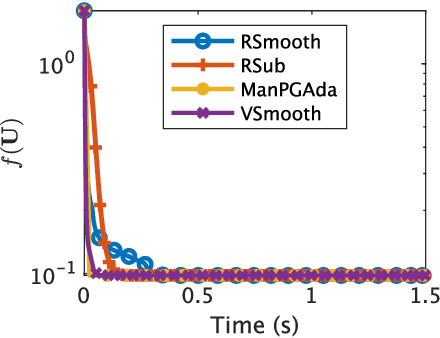}
  }
  \subfloat[][$(N,p) = (500,5)$]{
    \includegraphics[clip, width=0.30\textwidth]{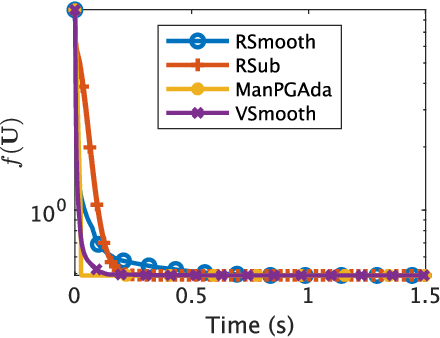}
  }
  \subfloat[][$(N,p) = (500,50)$]{
    \includegraphics[clip, width=0.30\textwidth]{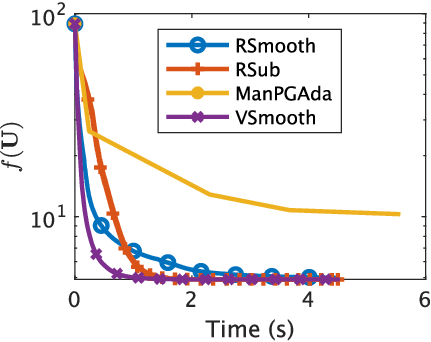}
  } \\
  \subfloat[][$(N,p) = (1000,1)$]{
    \includegraphics[clip, width=0.30\textwidth]{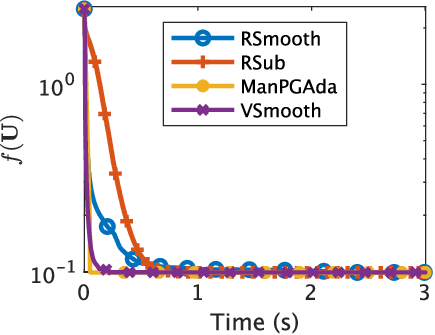}
  }
  \subfloat[][$(N,p) = (1000,10)$]{
    \includegraphics[clip, width=0.30\textwidth]{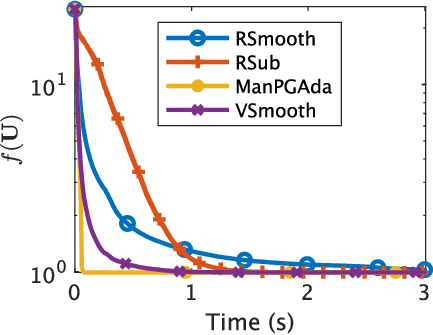}
  }
  \subfloat[][$(N,p) = (1000,100)$]{
    \includegraphics[clip, width=0.30\textwidth]{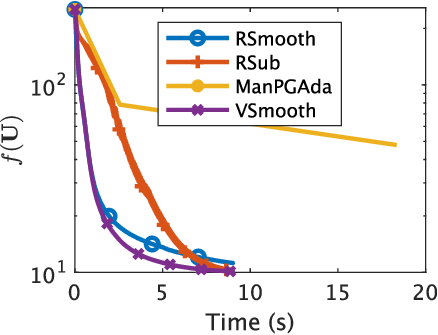}
  }
  \caption{Convergence histories of all algorithms for the problem~\eqref{eq:SPCA} regarding the value $f(\bm{U})$ at CPU time for each problem size.
    Markers are put at every 100 iterations.}
  \label{fig:SPCA_01}
\end{figure}

\begin{table}[!ht]
  \centering
  \footnotesize
  \caption{Averaged performance of all algorithms for the problem~\eqref{eq:SPCA}.}
  \begin{tabular}{cllllll}
    \toprule
    Problem size & Algorithm & fval        & feasi      & itr     & time  & sparsity  \\ \midrule
    \multirow{4}{*}{$(N,p) = (200,1)$}
                 & RSmooth   & 9.50229e-02 & 7.216e-16  & 3673.8  & 0.50  & 9.950e-01 \\
                 & RSub      & 9.50082e-02 & 0.000e+00  & 17357   & 0.50  & 9.950e-01 \\
                 & ManPGAda  & 9.50082e-02 & 2.220e-17  & 3927.5  & 0.50  & 9.950e-01 \\
                 & VSmooth   & 9.50728e-02 & 2.220e-16  & 5493    & 0.50  & 9.950e-01 \\ \midrule
    \multirow{4}{*}{ $(N,p) = (200,2)$}
                 & RSmooth   & 1.90027e-01 & 9.578e-16  & 2292.4  & 0.50  & 9.950e-01 \\
                 & RSub      & 1.89986e-01 & 9.687e-106 & 11810.9 & 0.50  & 9.950e-01 \\
                 & ManPGAda  & 1.89973e-01 & 2.116e-16  & 1956.2  & 0.50  & 9.950e-01 \\
                 & VSmooth   & 1.90124e-01 & 7.477e-16  & 2994.9  & 0.50  & 9.950e-01 \\ \midrule
    \multirow{4}{*}{$(N,p) = (200,20)$}
                 & RSmooth   & 1.90032e+00 & 5.081e-15  & 1646.5  & 1.50  & 9.950e-01 \\
                 & RSub      & 1.89997e+00 & 3.513e-15  & 8918.2  & 1.50  & 9.950e-01 \\
                 & ManPGAda  & 1.89987e+00 & 6.250e-15  & 1377.4  & 1.50  & 9.950e-01 \\
                 & VSmooth   & 1.90149e+00 & 3.050e-15  & 2552.4  & 1.50  & 9.950e-01 \\ \midrule
    \multirow{4}{*}{$(N,p) = (500,1)$}
                 & RSmooth   & 9.80240e-02 & 7.661e-16  & 1943    & 1.50  & 9.980e-01 \\
                 & RSub      & 9.80058e-02 & 0.000e+00  & 7027.6  & 1.50  & 9.980e-01 \\
                 & ManPGAda  & 9.80058e-02 & 2.220e-17  & 1805.2  & 1.50  & 9.980e-01 \\
                 & VSmooth   & 9.80899e-02 & 1.554e-16  & 2513.7  & 1.50  & 9.980e-01 \\ \midrule
    \multirow{4}{*}{$(N,p) = (500,5)$}
                 & RSmooth   & 4.90112e-01 & 3.119e-15  & 1133.5  & 1.50  & 9.980e-01 \\
                 & RSub      & 4.89997e-01 & 9.807e-16  & 4875.8  & 1.50  & 9.980e-01 \\
                 & ManPGAda  & 4.90014e-01 & 2.370e-15  & 784.8   & 1.50  & 9.980e-01 \\
                 & VSmooth   & 4.90525e-01 & 2.025e-15  & 1299.5  & 1.50  & 9.980e-01 \\ \midrule
    \multirow{4}{*}{$(N,p) = (500,50)$}
                 & RSmooth   & 4.98856e+00 & 1.732e-14  & 774.8   & 4.51  & 9.977e-01 \\
                 & RSub      & 4.90002e+00 & 9.091e-15  & 2145.3  & 4.50  & 9.980e-01 \\
                 & ManPGAda  & 1.02775e+01 & 1.992e-14  & 5       & 5.57  & 8.736e-01 \\
                 & VSmooth   & 4.90486e+00 & 6.629e-15  & 1150.1  & 4.50  & 9.980e-01 \\ \midrule
    \multirow{4}{*}{$(N,p) = (1000,1)$}
                 & RSmooth   & 9.90208e-02 & 8.882e-16  & 1076.8  & 3.00  & 9.990e-01 \\
                 & RSub      & 9.89987e-02 & 0.000e+00  & 3238    & 3.00  & 9.990e-01 \\
                 & ManPGAda  & 9.89987e-02 & 4.441e-17  & 967.9   & 3.00  & 9.990e-01 \\
                 & VSmooth   & 9.91019e-02 & 3.775e-16  & 1338.7  & 3.00  & 9.990e-01 \\ \midrule
    \multirow{4}{*}{$(N,p) = (1000,10)$}
                 & RSmooth   & 1.02418e+00 & 6.645e-15  & 570.5   & 3.01  & 9.988e-01 \\
                 & RSub      & 9.90010e-01 & 2.970e-14  & 1700.3  & 3.00  & 9.990e-01 \\
                 & ManPGAda  & 9.90017e-01 & 4.901e-15  & 328.8   & 3.00  & 9.990e-01 \\
                 & VSmooth   & 9.91335e-01 & 3.006e-15  & 612.5   & 3.00  & 9.990e-01 \\ \midrule
    \multirow{4}{*}{$(N,p) = (1000,100)$}
                 & RSmooth   & 1.12218e+01 & 3.053e-14  & 376.2   & 9.02  & 9.973e-01 \\
                 & RSub      & 1.02637e+01 & 3.397e-14  & 723.6   & 9.01  & 9.051e-01 \\
                 & ManPGAda  & 4.77660e+01 & 3.306e-14  & 3       & 18.35 & 3.618e-01 \\
                 & VSmooth   & 1.01238e+01 & 9.032e-15  & 502.4   & 9.01  & 9.989e-01 \\ \bottomrule
  \end{tabular}
  \label{table:SPCA_01}
\end{table}

\section{Conclusions}
We addressed a manifold constrained nonsmooth optimization problem, where the cost function consists of a smooth function and the composite function of a weakly convex function and a smooth (nonlinear) mapping.
For the target problem, we proposed a variable smoothing algorithm with a parametrization of a nonconvex constraint set.
We also presented a subsequential convergence guarantee to a stationary point of the parametrized problem and a first-order oracle complexity
$\mathcal{O}(\epsilon^{-3})$
(resp. 
$\mathcal{O}(\epsilon^{-3}\log(\epsilon^{-1}))$)
for an $\epsilon$-approximate stationary point of the parameterized problem with diminishing stepsizes
(resp. backtracking-based stepsizes) 
in the proposed algorithm.
Under an additional condition on the parametrization, e.g., the submersion condition, the corresponding stationary point of the parameterized problem is mapped, by the parametrization, to a stationary point of the original manifold constrained problem.
The numerical experiments demonstrate that (i)
the proposed algorithm improves the numerical performance of sparsity-aware application by using a weakly convex regularizer;
(ii)
the proposed algorithm has numerical performance comparable to that of the existing Riemannian nonsmooth optimization algorithms based on a retraction.

\bmhead{Acknowledgements}
This work was supported partially by JSPS Grants-in-Aid (19H04134, 22KJ1270, 24K23885) and by JST SICORP (JPMJSC20C6).
We thank the anonymous reviewers for their thoughtful comments, which greatly improved the clarity and quality of the paper.
{
  \bibliography{main}%
}
\appendix

\section{Related Work} \label{appendix:related_work}

\subsection{Smoothing-type algorithms}
\label{appendix:related_work:smoothing}
This appendix provides a descriptive summary of smoothing-type algorithms for finding a  stationary point of a nonsmooth function
$J:\mathcal{X}\to\mathbb{R}$
because the proposed algorithm (Algorithm~\ref{alg:vsmooth}) can be viewed as a variant of smoothing-type algorithms.
Smoothing methods have been extensively developed~\cite{Zhang-Chen09,Chen12,Huang-Liu16,Wang-Chen21,Lai-Yoshise22,Zhang-Chen-Ma23,Nishioka-Kanno23,Liu-Liu-Chen22} as a general framework for nonsmooth (nonconvex) optimization with a smoothed surrogate function, say
$J^{\langle\mu\rangle}:\mathcal{X}\to\mathbb{R}\ (\mu\in \mathbb{R}_{++})$,
having the {\em (sub-)gradient consistency property} (see Remark~\ref{remark:gradient_consistency}) with
$\lim_{\mathbb{R}_{++}\ni \mu\to 0}  J^{\langle \mu \rangle}(\bm{x}) = J(\bm{x})\ (\bm{x} \in \mathcal{X})$.
The gradient consistency property has been utilized as a key tool  to express the first-order optimality condition in terms of a gradient sequence of
$\nabla J^{\langle \mu \rangle}$
with
$\mu\searrow 0$
(see, e.g.,~\cite{Chen12}).
The basic update of a smoothing gradient method~\cite{Chen12} is given by
$\bm{x}_{n+1} \coloneqq \bm{x}_{n} - \gamma_{n} \nabla J^{\langle \mu_{n} \rangle}(\bm{x}_{n})\ (n\in\mathbb{N})$
with a properly chosen stepsize
$\gamma_{n} \in \mathbb{R}_{++}$,
where smoothing parameters
$\mu_{n}$
are conditionally updated with
$\sigma \in (0,1)$
and
$\delta \in \mathbb{R}_{++}$
by setting, e.g.,
$\mu_{n+1} \coloneqq \sigma\mu_{n}$
if
$\norm{\nabla J^{\langle\mu_{n}\rangle}(\bm{x}_{n})} < \delta\mu_{n}$;
$\mu_{n+1} \coloneqq \mu_{n}$
otherwise.
Thanks to the (sub-)gradient consistency property of
$J^{\langle\mu\rangle}$,
a cluster point of
$(\bm{x}_{n})_{n=1}^{\infty}$
is guaranteed to be a (Clarke) stationary point of
$J$~\cite{Chen12}.
Based on the conditional update of smoothing parameters, many numerical algorithms, e.g., accelerated smoothing gradient methods~\cite{Wang-Chen21,Nishioka-Kanno23}, smoothing projected gradient methods~\cite{Zhang-Chen09,Huang-Liu16}, smoothing proximal gradient methods~\cite{Liu-Liu-Chen22} for linearly-constrained nonsmooth optimization appearing in training deep neural networks, and manifold constrained smoothing methods~\cite{Lai-Yoshise22,Zhang-Chen-Ma23}, for updating
$(\bm{x}_{n})_{n=1}^{\infty}$
have been established with similar subsequential convergence guarantees to a stationary point.

Another class of smoothing-type methods has been proposed with iteratively decreasing smoothing parameters
$(\mu_{n})_{n=1}^{\infty}$~\cite{Bot-Hendrich15,Tran-Dinh17,Bot-Bohm20,Bohm-Wright21,Jalilzadeh-Shanbhag-Blanchet-Glynn22,Wu-Bian23,Beck-Rosset23,Peng-Wu-Hu-Deng23,Nishioka-Kanno24}, where these methods are called variable/adaptive/dynamic smoothing methods.
For convex optimization problems, the oracle complexities for finding the optimal value of the cost function have been shown for~\cite{Bot-Hendrich15,Tran-Dinh17,Bot-Bohm20,Jalilzadeh-Shanbhag-Blanchet-Glynn22,Wu-Bian23,Nishioka-Kanno24}, where an exceptional paper~\cite{Wu-Bian23} guarantees that a generated sequence converges to a minimizer of the cost function.
Recently, for nonconvex nonsmooth optimization problems, a variable smoothing algorithm~\cite{Bot-Hendrich15} has been extended by~\cite{Bohm-Wright21} to a special version of Problem~\ref{problem:constrained} where
$\mathfrak{S}$
is a linear operator,
$g$
is a weakly convex,
and
$C\coloneqq \mathcal{X}$.
The variable smoothing algorithm in~\cite{Bohm-Wright21} uses
a smoothed surrogate function
$h+\moreau{g}{\mu}\circ\mathfrak{S}$
with
the {\em Moreau envelope}
$\moreau{g}{\mu}$
of
$g$
(see~\eqref{eq:Moreau}),
and its first-order oracle complexity
$\mathcal{O}(\epsilon^{-3})$
for finding a certain
$\epsilon$-approximate stationary point in the sense of~\eqref{eq:approximate_stationary} with
$\param=\Id$
has been shown under the linearity of
$\mathfrak{S}$.
This variable smoothing algorithm~\cite{Bohm-Wright21} with an oracle complexity guarantee has also been extended by~\cite{Beck-Rosset23,Peng-Wu-Hu-Deng23} to Problem~\ref{problem:constrained} where
$C$
is a compact embedded submanifold of
$\mathcal{X}$,
but
$\mathfrak{S}$
is still assumed to be linear.
Moreover, in~\cite{Beck-Rosset23},
a subsequential convergence guarantee to a stationary point has been explicitly reported
(see also Table~\ref{table:nonsmooth} in Section~\ref{sec:introduction}).

The proposed algorithm is also a variable smoothing algorithm, but its analysis is carried out for the parameterized problem (Problem~\ref{problem:origin}).
This is different from~\cite{Bohm-Wright21,Beck-Rosset23,Peng-Wu-Hu-Deng23} where a linear operator
$\mathfrak{S}$
is considered.
In the present setting, the mapping
$\mathfrak{S}\circ\param$
in the composite term
$g\circ\mathfrak{S}\circ\param$
is nonlinear, and therefore the analysis in~\cite{Bohm-Wright21,Beck-Rosset23,Peng-Wu-Hu-Deng23} cannot be applied directly.
The gradient (sub-)consistency result in Theorem~\ref{theorem:stationary} of our smoothed surrogate function is the key ingredient that allows us to establish the subsequential stationarity result for Problem~\ref{problem:origin} in Theorem~\ref{theorem:convergence_extension}, and then the corresponding stationarity statement for the original Problem~\ref{problem:constrained} follows under an additional condition on
$\param$,
such as the submersion
condition, discussed in Theorem~\ref{theorem:necessary} and Corollary~\ref{corollary:optimality_C_F}.

\subsection{Penalty-based strategy}
\label{appendix:related_work:Lagrangian}
In a case where
$C$
in Problem~\ref{problem:constrained} is defined by equality constraints, i.e.,
$C=H^{-1}(\bm{0})$
with some continuously differentiable mapping
$H:\mathcal{X} \to \mathbb{R}^{m}$,
penalty-based strategies have been proposed, e.g.,~\cite[Ch. 12]{Nocedal-Wright06},~\cite{Xiao-Liu-Toh24}, as an alternative strategy to the parametrization strategy in Section~\ref{sec:necessary} for translating Problem~\ref{problem:constrained} into an unconstrained problem.
Such penalty-based strategies aim to find a point satisfying the {\em Karush-Kuhn-Tucker (KKT) condition}
by assuming the availability of
$H$
with suitable regularities, e.g., {\em Linearly Independent Constraint Qualification (LICQ)}, i.e.,
$\mathrm{rank}(\mathrm{D}H(\bm{x}))=m$
for all
$\bm{x} \in C$,
where
$H$
with LICQ is called a {\em (global) defining map}~\cite[Cor. 8.76]{Boumal23}.
Indeed,
LICQ of
$H$
makes the KKT condition a necessary condition for the local optimality of Problem~\ref{problem:constrained} with a smooth
$\cost$~\cite[Thm. 12.1]{Nocedal-Wright06}, while the KKT condition without LICQ of
$H$
is not guaranteed to be a necessary condition for the local optimality.
However,
the existence of a global defining map
$H$
with LICQ
is not guaranteed for general embedded submanifolds~\cite[p.50]{Boumal23}, e.g., {\em non-orientable manifolds}~\cite[Prop. 15.23, pp.106-107]{Lee12}
(see~\cite{Vogt-Strekalovskiy-Cremers-Lellmann20} for applications of optimization over non-orientable manifolds).
Moreover, penalty-based strategies~\cite{Nocedal-Wright06,Xiao-Liu-Toh24} may generate an infeasible sequence
$(\bm{x}_{n})_{n=1}^{\infty} \subset \mathcal{X}$,
i.e.,
$\bm{x}_{n}$
may not belong to
$C$,
where the feasibility is guaranteed at cluster points of
$(\bm{x}_{n})_{n=1}^{\infty}$
in general.

In contrast, for an embedded submanifold
$C$
in Problem~\ref{problem:constrained},
a parametrization
$\param$
of
$C$
exists (see Example~\ref{example:parametrization}~\ref{enum:ex:parameterization:genearal}),
and
a stationary point
$\bm{y}^{\star} \in \mathcal{Y}$
of Problem~\ref{problem:origin} satisfies
a necessary condition for the local optimality of Problem~\ref{problem:constrained} at
$\bm{x}^{\star}\coloneqq \param(\bm{y}^{\star}) \in C$
(see Remark~\ref{remark:submersion}).
Under an additional condition, e.g.,
submersion condition on
$\param$,
$\bm{x}^{\star}$
is a stationary point of Problem~\ref{problem:constrained} (see Theorem~\ref{theorem:necessary} and Corollary~\ref{corollary:optimality_C_F}).
Furthermore, the parametrization strategy described in Section~\ref{sec:necessary} generates a sequence
$(\bm{y}_{n})_{n=1}^{\infty}\subset \mathcal{Y}$
for finding a stationary point
$\bm{y}^{\star}$;
hence the feasibility of
$(\param(\bm{y}_{n}))_{n=1}^{\infty} \subset C$
is guaranteed via
$\param$
(see Algorithm~\ref{alg:vsmooth} in Section~\ref{sec:proposed}; see also, e.g.,~\cite{Yamada-Ezaki03,Fraikin-Huper-Dooren07,Hori-Tanaka10,Helfrich-Willmott-Ye18,Lezcano19,Kume-Yamada21,Kume-Yamada22,Levin-Kileel-Boumal24}, for parametrization strategy with smooth cost function).

Finally, we remark that
the submersion condition of a parametrization
$\param$
and LICQ of
$H$
are closely related
if both
$\param$
with the submersion condition
and
a global defining map
$H$
with LICQ exist for an embedded submanifold
$C$.
Indeed, in this case, we have
$T_{C}(\bm{x}) = \mathrm{Ker}(\mathrm{D}H(\bm{x})) = \mathrm{Ran}(\mathrm{D}\param(\bm{y})) \ (\bm{y} \in \mathcal{Y}, \bm{x}\coloneqq \param(\bm{y}) \in C)$.
However,
the availabilities of
$\param$
with the submersion condition and
a global defining map
$H$
with LICQ are not comparable in general because we have an example\footnote{ \label{foot:Klein}
  The {\em Klein bottle} has no global defining map satisfying LICQ due to its non-orientability (see, e.g.,~\cite[Sect. 3.3.1]{Vogt-Strekalovskiy-Cremers-Lellmann20} and~\cite[Prop. 15.23, pp.106-107]{Lee12}).
  However,
  since
  $\mathbb{R}^{2}$
  is the {\em universal covering manifold}~\cite[p.90]{Lee12} of
  the Klein bottle (see, e.g.,~\cite{Anosov90}),
  there exists a {\em smooth covering mapping}~\cite[p.90]{Lee12} from
  $\mathbb{R}^{2}$
  onto the Klein bottle.
  Such a smooth covering mapping is
  a parametrization
  $\param$,
  defined over
  $\mathbb{R}^{2}$,
  of the Klein bottle
  such that the submersion condition is satisfied at every point in
  $\mathbb{R}^{2}$ by~\cite[Prop. 4.33 (a)]{Lee12}.
}of
$C$
such that any global defining map
$H$
of
$C$
with LICQ
does not exist, whereas a parametrization
$\param:\mathcal{Y}\to\mathcal{X}$
of
$C$
with the submersion condition exists.

\section{Proof of Lemma~\ref{lemma:subdifferentially}}\label{appendix:lemma:subdifferentially}
Although Lemma~\ref{lemma:subdifferentially} can be checked by the subdifferential calculus rules~\cite[Thm. 10.6, Cor. 10.9]{Rockafellar-Wets98}, we present a short proof of Lemma~\ref{lemma:subdifferentially} with some necessary facts in order to increase the readability for researchers in the area of applications.
The subdifferential calculus rules heavily rely on the {\em horizontal subdifferential}, which gives constraint qualifications, and the {\em subdifferential regularity} of a function (see Definition~\ref{definition:regular}~\ref{enum:definition:regular:function} and Example~\ref{ex:regular}).

\begin{definition}[Horizontal subdifferential{~\cite[Def. 8.3 and Thm. 8.9]{Rockafellar-Wets98}}] \label{definition:subdifferentially_regular}
  Let
  $J:\mathcal{X} \to \exR$
  be a proper lower semicontinuous function.
  Then,
  $\subdiff^{\infty}J(\widebar{\bm{x}})\coloneqq \{\bm{v}\in \mathcal{X} \mid (\bm{v},0) \in N_{\epi{J}}(\widebar{\bm{x}},J(\widebar{\bm{x}}))\}$
  is called the {\em horizontal subdifferential} of
  $J$
  at
  $\widebar{\bm{x}} \in \dom{J}$.
\end{definition}

\begin{example}[Subdifferentially regular functions] \label{ex:regular}
  Subdifferentially regular functions include, e.g.,
  continuously differentiable functions~\cite[Exm. 7.28]{Rockafellar-Wets98},
  proper lower semicontinuous (weakly) convex functions~\cite[Exm. 7.27]{Rockafellar-Wets98}~\cite[Prop. 4.4.15]{Ying-Jong-Shi}, and indicator functions of Clarke regular subsets~\cite[Exm. 7.28]{Rockafellar-Wets98}.
\end{example}

Fact~\ref{fact:strictly_continuous} presents useful properties, regarding
$\subdiff^{\infty}J$,
that will be used for our analyses.
\begin{fact}[Horizontal subdifferential being zero~{\cite[Thm. 9.13]{Rockafellar-Wets98}}] \label{fact:strictly_continuous}
  Let
  $J:\mathcal{X} \to \exR$
  be a proper lower semicontinuous function.
  Then, for
  $\widebar{\bm{x}} \in \dom{J}$,
  the following conditions are equivalent:
  \begin{enumerate}[label=(\alph*)]
    \item \label{enum:horizontal}
          $\subdiff^{\infty}J(\widebar{\bm{x}}) = \{\bm{0}\}$.
    \item \label{enum:strictly_continuous}
          $J$
          is locally Lipschitz continuous at
          $\widebar{\bm{x}}$.
    \item
          $\Lsubdiff  J$
          is {\em locally bounded} at
          $\widebar{\bm{x}}$,
          i.e., there exists an open neighborhood
          $\mathcal{N}_{\widebar{\bm{x}}} \subset \mathcal{X}$
          of
          $\widebar{\bm{x}}$
          such that
          $\Lsubdiff  J(\mathcal{N}_{\widebar{\bm{x}}}) = \bigcup_{\bm{x}\in\mathcal{N}_{\widebar{\bm{x}}}}\Lsubdiff  J(\bm{x})$
          is bounded.
  \end{enumerate}
  Moreover, if these conditions hold at
  $\widebar{\bm{x}} \in \dom{J}$,
  then
  $\Lsubdiff  J(\widebar{\bm{x}}) \neq \emptyset$.
\end{fact}

\begin{fact}[Calculus rules for subdifferential] \label{fact:chain_rule}
  Let
  $\mathcal{X}$
  and
  $\mathcal{Y}$
  be Euclidean spaces.
  Let
  $J_{1},J_{2}:\mathcal{X} \to \exR$
  be proper lower semicontinuous and subdifferentially regular functions, and
  $\mathcal{F}:\mathcal{Y} \to \mathcal{X}$
  be a continuously differentiable mapping.
  Then, the following hold:
  \begin{enumerate}[label=(\alph*)]
    \item (Chain rule~\cite[Thm. 10.6]{Rockafellar-Wets98}) \label{enum:basic_chain_rule}
          If the constraint qualification
          $\subdiff^{\infty} J_{1}(\mathcal{F}(\widebar{\bm{y}})) \cap \mathrm{Ker}((\mathrm{D}\mathcal{F}(\widebar{\bm{y}}))^{*}) = \{\bm{0}\}$
          holds at
          $\widebar{\bm{y}} \in \dom{J_{1}\circ\mathcal{F}}$,
          then
          $J_{1}\circ \mathcal{F}$
          is subdifferentially regular at
          $\widebar{\bm{y}}$
          and
          \begin{align}
            \Lsubdiff  (J_{1}\circ \mathcal{F})(\widebar{\bm{y}})        & = (\mathrm{D}\mathcal{F}(\widebar{\bm{y}}))^{*}[\Lsubdiff  J_{1}(\mathcal{F}(\widebar{\bm{y}}))];     \label{eq:basic_chain}         \\
            \subdiff^{\infty} (J_{1}\circ \mathcal{F})(\widebar{\bm{y}}) & = (\mathrm{D}\mathcal{F}(\widebar{\bm{y}}))^{*}[\subdiff^{\infty} J_{1}(\mathcal{F}(\widebar{\bm{y}}))]. \label{eq:chain_horizontal}
          \end{align}
    \item (Sum rule~\cite[Cor. 10.9, Exe. 8.8]{Rockafellar-Wets98}) \label{enum:sum_rule_general}
          If the constraint qualification
          $\subdiff^{\infty}J_{1}(\widebar{\bm{x}}) \cap (-\subdiff^{\infty}J_{2}(\widebar{\bm{x}})) = \{\bm{0}\}$
          holds at
          $\widebar{\bm{x}} \in \dom{J_{1}+J_{2}}$,
          then
          $J_{1} + J_{2}$
          is subdifferentially regular at
          $\widebar{\bm{x}}$,
          and
          \begin{align}
            \Lsubdiff  (J_{1}+J_{2})(\widebar{\bm{x}}) & = \Lsubdiff  J_{1}(\widebar{\bm{x}}) + \Lsubdiff  J_{2}(\widebar{\bm{x}}). \label{eq:sum_general}
          \end{align}
          Moreover, if
          $J_{2}$
          is continuously differentiable, then
          $J_{1}+J_{2}$
          is subdifferentially regular at
          $\widebar{\bm{x}} \in \dom{J_{1}}$,
          and
          \begin{align}
            \Lsubdiff  (J_{1}+J_{2})(\widebar{\bm{x}})        & = \Lsubdiff  J_{1}(\widebar{\bm{x}}) + \nabla J_{2}(\widebar{\bm{x}}); \label{eq:sum_smooth_subgradient} \\
            \subdiff^{\infty} (J_{1}+J_{2})(\widebar{\bm{x}}) & = \subdiff^{\infty} J_{1}(\widebar{\bm{x}}). \label{eq:sum_smooth}
          \end{align}
  \end{enumerate}
\end{fact}

\begin{proof}[Proof of Lemma~\ref{lemma:subdifferentially}]
  (Step 1)
  We show the subdifferential regularity of
  $\cost=h+g\circ \mathfrak{S}$
  and
  $\subdiff^{\infty}\cost(\widebar{\bm{x}}) = \{\bm{0}\}\ (\widebar{\bm{x}} \in \mathcal{X})$.
  By the latter part of Fact~\ref{fact:chain_rule}~\ref{enum:sum_rule_general} and~\eqref{eq:sum_smooth} with the continuous differentiability of
  $h$,
  it suffices to show (i) the subdifferential regularity of
  $g\circ\mathfrak{S}$
  and (ii)
  $\subdiff^{\infty} (g\circ\mathfrak{S})(\widebar{\bm{x}}) = \{\bm{0}\}\ (\widebar{\bm{x}} \in \mathcal{X})$.
  Recall that
  $g$
  is subdifferentially regular from its weak convexity (see Example~\ref{ex:regular}).
  Since
  $g$
  is (locally) Lipschitz continuous,
  Fact~\ref{fact:strictly_continuous} ensures
  $\subdiff^{\infty} g(\mathfrak{S}(\widebar{\bm{x}})) = \{\bm{0}\}$,
  implying thus the constraint qualification in Fact~\ref{fact:chain_rule}~\ref{enum:basic_chain_rule} holds with
  $J_{1} \coloneqq g$
  and
  $\mathcal{F}\coloneqq \mathfrak{S}$.
  Then, Fact~\ref{fact:chain_rule}~\ref{enum:basic_chain_rule} implies
  (i) the subdifferential regularity of
  $g\circ \mathfrak{S}$,
  and
  (ii)
  $\subdiff^{\infty}(g\circ\mathfrak{S})(\widebar{\bm{x}}) \overset{\eqref{eq:chain_horizontal}}{=}(\mathrm{D}\mathfrak{S}(\widebar{\bm{x}}))^{*}[\subdiff^{\infty} g(\mathfrak{S}(\widebar{\bm{x}}))] = \{\bm{0}\}\ (\widebar{\bm{x}} \in \mathcal{X})$.

  (Step 2)
  Recall that
  $\iota_{C}$
  is subdifferentially regular (see Example~\ref{ex:regular} and Example~\ref{ex:normal}), and lower semicontinuous due to the closedness of
  $C$.
  Then,
  by letting
  $J_{1} \coloneqq \cost$,
  $J_{2} \coloneqq \iota_{C}$,
  and
  $\mathcal{F} \coloneqq \param $,
  Fact~\ref{fact:chain_rule} yields that
  (i) the subdifferential regularities of
  $\cost \circ \param$
  and
  $\cost + \iota_{C}$;
  (ii)
  the equalities~\eqref{eq:subdifferentially_f_C} and~\eqref{eq:chain2} from~\eqref{eq:sum_general} and~\eqref{eq:basic_chain},
  where
  $J_{1}$
  is subdifferentially regular and
  $\subdiff^{\infty}J_{1}(\widebar{\bm{x}}) = \{\bm{0}\}\ (\widebar{\bm{x}} \in \mathcal{X})$
  from (Step 1).
\end{proof}

\section{Proof of Proposition~\ref{proposition:extension:Lipschitz_f}}\label{appendix:extension:Lipschitz_f}

To prove Proposition~\ref{proposition:extension:Lipschitz_f}, we use the following lemma for two times.
\begin{lemma}[Lipschitz continuity of gradient of composite function] \label{lemma:Lipschitz_composite}
  Let
  $\mathcal{H},\mathcal{K}$
  be Euclidean spaces.
  Let
  $\mathcal{F}:\mathcal{K}\to \mathcal{H}$
  and
  $J:\mathcal{H} \to \mathbb{R}$
  be continuously differentiable,
  and let
  $(\emptyset \neq)E \subset \mathcal{K}$
  be a subset of
  $\mathcal{K}$.
  Assume that
  \begin{enumerate}[label=(\roman*)]
    \item \label{enum:c1}
          The operator norm of
          $(\mathrm{D}\mathcal{F}(\cdot))^{*}$
          is bounded above by
          $\kappa_{\mathcal{F}} > 0$
          over
          $\mathrm{Conv}(E)$,
          and thus
          $\mathcal{F}$
          is Lipschitz continuous with a Lipschitz constant
          $\kappa_{\mathcal{F}}$
          over
          $\mathrm{Conv}(E)$
          by~\cite[Thm. 9.2 and Thm. 9.7]{Rockafellar-Wets98}.
    \item \label{enum:c2}
          $(\mathrm{D}\mathcal{F}(\cdot))^{*}$
          is Lipschitz continuous with a Lipschitz constant
          $L_{\mathrm{D}\mathcal{F}} > 0$
          over
          $E$.
    \item \label{enum:c3}
          The norm of
          $\nabla J$
          is bounded above by
          $\kappa_{J} > 0$
          over
          $\mathcal{F}(E)$.
    \item \label{enum:c4}
          $\nabla J$
          is Lipschitz continuous with a Lipschitz constant
          $L_{\nabla J} > 0$
          over
          $\mathcal{F}(E)$.
  \end{enumerate}
  Then,
  $\nabla (J\circ \mathcal{F})$
  is Lipschitz continuous with a Lipschitz constant
  $\kappa_{\mathcal{F}}^{2}L_{\nabla J} + \kappa_{J}L_{\mathrm{D}\mathcal{F}}$
  over
  $E$.
\end{lemma}
\begin{proof}
  For
  $\bm{x}_{1},\bm{x}_{2} \in E$,
  we have  {
      \small
      \begin{align}
         & \norm{\nabla (J\circ \mathcal{F})(\bm{x}_{1}) - \nabla (J\circ \mathcal{F})(\bm{x}_{2})}
        = \norm{(\mathrm{D}\mathcal{F}(\bm{x}_{1}))^{*}[\nabla J(\mathcal{F}(\bm{x}_{1}))] - (\mathrm{D}\mathcal{F}(\bm{x}_{2}))^{*}[\nabla J(\mathcal{F}(\bm{x}_{2}))]}          \\
         & \leq \norm{(\mathrm{D}\mathcal{F}(\bm{x}_{1}))^{*}[\nabla J(\mathcal{F}(\bm{x}_{1})) - \nabla J(\mathcal{F}(\bm{x}_{2}))]}
        + \norm{(\mathrm{D}\mathcal{F}(\bm{x}_{1})-\mathrm{D}\mathcal{F}(\bm{x}_{2}))^{*}[\nabla J(\mathcal{F}(\bm{x}_{2}))]}                                                     \\
         & \overset{\ref{enum:c2}}{\leq} \norm{(\mathrm{D}\mathcal{F}(\bm{x}_{1}))^{*}}_{\mathrm{op}}\norm{\nabla J(\mathcal{F}(\bm{x}_{1})) - \nabla J(\mathcal{F}(\bm{x}_{2}))}
        + L_{\mathrm{D}\mathcal{F}}\norm{\bm{x}_{1} - \bm{x}_{2}}\norm{\nabla J(\mathcal{F}(\bm{x}_{2}))}                                                                         \\
         & \overset{\ref{enum:c1},\ref{enum:c3}}{\leq} \kappa_{\mathcal{F}}\norm{\nabla J(\mathcal{F}(\bm{x}_{1})) -\nabla J(\mathcal{F}(\bm{x}_{2}))}
        + \kappa_{J}L_{\mathrm{D}\mathcal{F}}\norm{\bm{x}_{1} - \bm{x}_{2}}                                                                                                       \\
         & \overset{\ref{enum:c4},\ref{enum:c1}}{\leq} \kappa_{\mathcal{F}}^{2}L_{\nabla J}\norm{\bm{x}_{1}-\bm{x}_{2}}
        + \kappa_{J}L_{\mathrm{D}\mathcal{F}}\norm{\bm{x}_{1} - \bm{x}_{2}}
        = (\kappa_{\mathcal{F}}^{2}L_{\nabla J} +\kappa_{J}L_{\mathrm{D}\mathcal{F}})\norm{\bm{x}_{1}-\bm{x}_{2}}.
      \end{align}
    }
\end{proof}

\begin{proof}[Proof of Proposition~\ref{proposition:extension:Lipschitz_f}]
  From Fact~\ref{fact:moreau}~\ref{enum:gradient_moreau},
  recall that
  $\nabla \moreau{g}{\mu}$
  is Lipschitz continuous with the Lipschitz constant
  $L_{\nabla \moreau{g}{\mu}}\coloneqq \max\left\{\mu^{-1},\frac{\eta}{1-\eta\mu}\right\} = \mu^{-1}$
  by
  $\mu \in (0, 2^{-1}\eta^{-1}]$.

  (a)
  Since
  $\nabla \moreau{g}{\mu}$
  is Lipschitz continuous with the Lipschitz constant
  $L_{\nabla\moreau{g}{\mu}}$
  and
  $\norm{\nabla \moreau{g}{\mu}(\bm{z})} \leq L_{g}\ (\bm{z} \in \mathcal{Z})$
  by~\eqref{eq:moreau_gradient_bound}, Lemma~\ref{lemma:Lipschitz_composite} with setting
  $(\mathcal{H},\mathcal{K},\mathcal{F},J,E) \coloneqq  (\mathcal{Z},\mathcal{X},\mathfrak{S},\moreau{g}{\mu},C)$
  ensures the Lipschitz continuity of
  $\nabla (\moreau{g}{\mu}\circ \mathfrak{S})$
  with a Lipschitz constant
  $L_{\nabla (\moreau{g}{\mu}\circ \mathfrak{S})}\coloneqq  \kappa_{\mathfrak{S}}^{2}L_{\nabla \moreau{g}{\mu}} + L_{g}L_{\mathrm{D}\mathfrak{S}}=\kappa_{\mathfrak{S}}^{2}\mu^{-1} + L_{g}L_{\mathrm{D}\mathfrak{S}}$.
  Thus,
  $\nabla (h+\moreau{g}{\mu}\circ \mathfrak{S})$
  is Lipschitz continuous over
  $C$
  with a Lipschitz constant
  $L_{\nabla (h+\moreau{g}{\mu}\circ \mathfrak{S})} = L_{\nabla h} + L_{g}L_{\mathrm{D}\mathfrak{S}} + \kappa_{\mathfrak{S}}^{2}\mu^{-1}$,
  where
  $L_{\nabla h}$
  is the Lipschitz constant of
  $\nabla h$
  by the setting of Problem~\ref{problem:origin}.

  (b)
  To apply Lemma~\ref{lemma:Lipschitz_composite}, we derive an upper bound
  $\kappa_{h+\moreau{g}{\mu}\circ \mathfrak{S}}\coloneqq \kappa_{h} + \kappa_{\mathfrak{S}}L_{g} > 0$
  of
  $\sup_{\bm{x}\in C}\norm{\nabla (h+\moreau{g}{\mu}\circ \mathfrak{S})(\bm{x})}$
  as:
  \begin{equation}
    \sup_{\bm{x}\in C}\norm{\nabla (h+\moreau{g}{\mu}\circ \mathfrak{S})(\bm{x})}
    \leq \sup_{\bm{x}\in C}\norm{\nabla h(\bm{x})} + \sup_{\bm{x}\in C}\norm{\nabla (\moreau{g}{\mu}\circ \mathfrak{S})(\bm{x})}
    \leq    \kappa_{h} + \kappa_{\mathfrak{S}}L_{g},
    \label{eq:nabla_f_bound}
  \end{equation}
  where we used the inequality
  $\norm{\nabla (\moreau{g}{\mu}\circ \mathfrak{S})(\bm{x})} = \norm{(\mathrm{D}\mathfrak{S}(\bm{x}))^{*}[\nabla \moreau{g}{\mu}(\mathfrak{S}(\bm{x}))]} \leq \norm{(\mathrm{D}\mathfrak{S}(\bm{x}))^{*}}_{\mathrm{op}}\norm{\nabla \moreau{g}{\mu}(\mathfrak{S}(\bm{x}))} \overset{\eqref{eq:moreau_gradient_bound}}{\leq} \kappa_{\mathfrak{S}}L_{g}\ (\bm{x} \in C)$.
  Then, by applying Lemma~\ref{lemma:Lipschitz_composite} with setting
  $(\mathcal{H},\mathcal{K},\mathcal{F},J,E) \coloneqq  (\mathcal{X},\mathcal{Y},\param ,h+\moreau{g}{\mu}\circ \mathfrak{S},\mathcal{Y})$
  together with Proposition~\ref{proposition:extension:Lipschitz_f}~\ref{enum:Lipschitz_g_G},
  $\nabla ((h+\moreau{g}{\mu}\circ \mathfrak{S})\circ \param )$
  is Lipschitz continuous with a Lipschitz constant
  $L_{\nabla ((h+\moreau{g}{\mu}\circ \mathfrak{S})\circ \param )}=\kappa_{\param }^{2}L_{\nabla (h+\moreau{g}{\mu}\circ \mathfrak{S})} + \kappa_{h+\moreau{g}{\mu}\circ \mathfrak{S}}L_{\mathrm{D}\param }$
  over
  $\mathcal{Y}$.
  By substituting
  $L_{\nabla (h+\moreau{g}{\mu}\circ \mathfrak{S})}=L_{\nabla h}  + L_{g}L_{\mathrm{D}\mathfrak{S}} +  \kappa_{\mathfrak{S}}^{2}\mu^{-1}$
  and
  $\kappa_{h+\moreau{g}{\mu}\circ \mathfrak{S}}=\kappa_{h} + \kappa_{\mathfrak{S}}L_{g}$,
  we obtain the desired Lipschitz constant.
\end{proof}

\section{Proof of Theorem~\ref{theorem:stationary}}\label{appendix:theorem_stationary}
To prove Theorem~\ref{theorem:stationary}, we need the following lemma.

\begin{lemma}[Subdifferential limit]\label{lemma:outer_continuous_image}
  Let
  $\mathcal{Y}$
  and
  $\mathcal{Z}$
  be Euclidean spaces.
  Let
  $J:\mathcal{Z} \to \mathbb{R}$
  be locally Lipschitz continuous at
  $\widebar{\bm{z}} \in \mathcal{Z}$,
  and
  $A:\mathcal{Z} \to \mathcal{Y}$
  a linear operator.
  Suppose that
  (i)
  an arbitrarily given
  $(\widebar{\bm{z}}_{n})_{n=1}^{\infty} \subset \mathcal{Z}$
  satisfies
  $\lim_{n\to\infty}\widebar{\bm{z}}_{n} = \widebar{\bm{z}}$;
  (ii)
  an arbitrarily given
  $(A_{n})_{n=1}^{\infty}$
  with linear operators
  $A_{n}:\mathcal{Z} \to \mathcal{Y}\ (n\in \mathbb{N})$
  satisfies
  $\lim_{n\to\infty} A_{n} = A$.
  Then,
  \begin{equation}
    \Limsup_{n\to\infty} A_{n}(\Lsubdiff  J(\widebar{\bm{z}}_{n})) \subset A\left(\Limsup_{n\to\infty}\Lsubdiff  J(\widebar{\bm{z}}_{n})\right)  \subset A(\Lsubdiff  J(\widebar{\bm{z}})). \label{eq:subdifferential_inclusion}
  \end{equation}
\end{lemma}
\begin{proof}
  The last inclusion in~\eqref{eq:subdifferential_inclusion} can be checked by
  \begin{equation}
    \Limsup_{n\to\infty} \Lsubdiff  J(\widebar{\bm{z}}_{n})
    \subset \bigcup\limits_{\mathcal{Z} \ni \bm{z}_{n}\to\widebar{\bm{z}}} \Limsup_{n\to\infty} \Lsubdiff  J(\bm{z}_{n})
    \overset{\eqref{eq:outer_limit_another}}{=} \Limsup_{\mathcal{Z} \ni \bm{z}\to\widebar{\bm{z}}}\Lsubdiff  J(\bm{z})
    = \Lsubdiff  J(\widebar{\bm{z}}),
  \end{equation}
  where the last equality follows by the
    {\em outer semicontinuity} of the limiting subdifferential~\cite[Prop. 8.7]{Rockafellar-Wets98}
  with the continuity of
  $J$
  at
  $\widebar{\bm{z}}$.

  To show the first inclusion in~\eqref{eq:subdifferential_inclusion}, let
  $\bm{v} \in \Limsup_{n\to\infty} A_{n}(\Lsubdiff  J(\widebar{\bm{z}}_{n}))$,
  i.e., there exists a sequence
  $(\bm{v}_{n})_{n=1}^{\infty} \subset \mathcal{Y}$
  with
  $\bm{v}_{n} \in A_{n}(\Lsubdiff  J(\widebar{\bm{z}}_{n}))$
  such that its subsequence
  $(\bm{v}_{m(l)})_{l=1}^{\infty}$
  converges to
  $\bm{v}$,
  where
  $m:\mathbb{N}\to\mathbb{N}$
  is monotonically increasing.
  Then, there exists a sequence
  $(\bm{u}_{n})_{n=1}^{\infty} \subset \mathcal{Z}$
  with
  $\bm{u}_{n}\in \Lsubdiff  J(\widebar{\bm{z}}_{n})$
  satisfying
  $\bm{v}_{n} = A_{n}(\bm{u}_{n})$.
  From Fact~\ref{fact:strictly_continuous}, the local Lipschitz continuity of
  $J$
  at
  $\widebar{\bm{z}}$
  ensures the boundedness of
  $\bigcup_{n=n_{0}}^{\infty} \Lsubdiff  J(\widebar{\bm{z}}_{n})$
  with a sufficiently large
  $n_{0} \in \mathbb{N}$.
  This implies the boundedness of
  $(\bm{u}_{n})_{n=1}^{\infty}$,
  thereby its subsequence
  $(\bm{u}_{m(l)})_{l=1}^{\infty}$
  converges to some
  $\bm{u} \in \mathcal{Z}$
  (by passing to further subsequence if necessary),
  i.e.,
  $\bm{u} \in \Limsup_{n\to\infty} \Lsubdiff  J(\widebar{\bm{z}}_{n})$.
  We get
  $\bm{v} = \lim_{l\to\infty}A\bm{u}_{m(l)}$
  by
  \begin{align}
    \norm{\bm{v} - A\bm{u}_{m(l)}}
     & \leq \norm{\bm{v} - \bm{v}_{m(l)}} + \norm{\bm{v}_{m(l)} - A_{m(l)}\bm{u}_{m(l)}} + \norm{A_{m(l)}\bm{u}_{m(l)} - A\bm{u}_{m(l)}}                             \\
     & = \norm{\bm{v} - \bm{v}_{m(l)}}  + \norm{A_{m(l)}\bm{u}_{m(l)} - A\bm{u}_{m(l)}}\quad (\because \bm{v}_{m(l)} = A_{m(l)}\bm{u}_{m(l)})                        \\
     & \leq \norm{\bm{v} - \bm{v}_{m(l)}} +  \norm{A_{m(l)}-A}_{\rm op}\norm{\bm{u}_{m(l)}}
    \overset{l\to \infty}{\to} 0 \label{eq:limit_Au}                                                                                                                 \\
     & \hspace{10em}(\because \bm{v}_{m(l)} \overset{l\to\infty}{\to} \bm{v},\  A_{m(l)}\overset{l\to\infty}{\to} A,\ \bm{u}_{m(l)}\overset{l\to\infty}{\to}\bm{u}).
  \end{align}
  Thus,
  $\bm{v} \overset{\eqref{eq:limit_Au}}{=} \lim_{l\to\infty}A\bm{u}_{m(l)}= A\lim_{l\to\infty}\bm{u}_{m(l)} = A\bm{u} \in A\left(\Limsup_{n\to\infty} \Lsubdiff  J(\widebar{\bm{z}}_{n})\right)$.

\end{proof}

\begin{proof}[Proof of Theorem~\ref{theorem:stationary}]
  (a)
  Let
  $\bm{v} \in \Limsup\limits_{\mathbb{R}_{++} \ni \mu\searrow0,\ \bm{y}\to\widebar{\bm{y}}} \nabla (\moreau{J}{\mu}\circ \mathcal{F})(\bm{y})$.
  Then, there exist
  (i)
  $(\bm{y}_{n})_{n=1}^{\infty} \subset\mathcal{Y}$
  with
  $\lim_{n\to\infty}\bm{y}_{n} = \widebar{\bm{y}}$;
  (ii)
  $(\mu_{n})_{n=1}^{\infty}\subset (0,\eta^{-1})$
  with
  $\mu_{n}\searrow 0$;
  and
  (iii)
  $\bm{v}_{n} \coloneqq  \nabla (\moreau{J}{\mu_{n}}\circ \mathcal{F})(\bm{y}_{n}) =  (\mathrm{D}\mathcal{F}(\bm{y}_{n}))^{*}[\nabla \moreau{J}{\mu_{n}}(\mathcal{F}(\bm{y}_{n}))]$
  with
  $\lim_{n\to\infty}\bm{v}_{n} = \bm{v}$.
  From Fact~\ref{fact:moreau}~\ref{enum:prox_subdifferential}, we have
  $\nabla \moreau{J}{\mu_{n}}(\mathcal{F}(\bm{y}_{n})) \in \Lsubdiff  J(\prox{\mu_{n}J}(\mathcal{F}(\bm{y}_{n})))\ (n\in \mathbb{N})$.
  This implies
  $\bm{v}_{n} \in (\mathrm{D}\mathcal{F}(\bm{y}_{n}))^{*}[\Lsubdiff  J(\prox{\mu_{n}J}(\mathcal{F}(\bm{y}_{n})))]\ (n \in \mathbb{N})$,
  and thus
  \begin{equation}
    \bm{v} \in \Limsup_{n\to\infty}(\mathrm{D}\mathcal{F}(\bm{y}_{n}))^{*}[\Lsubdiff  J(\prox{\mu_{n}J}(\mathcal{F}(\bm{y}_{n})))]. \label{eq:v_outer_limit}
  \end{equation}

  For
  $n\in \mathbb{N}$,
  let
  $\widebar{\bm{z}}_{n} \coloneqq \prox{\mu_{n}J}(\mathcal{F}(\bm{y}_{n}))$
  and
  $A_{n}\coloneqq (\mathrm{D}\mathcal{F}(\bm{y}_{n}))^{*}$.
  Then, we have
  $\lim_{n\to \infty} A_{n} = (\mathrm{D}\mathcal{F}(\widebar{\bm{y}}))^{*}\eqqcolon A$
  from the continuous differentiability of
  $\mathcal{F}$.
  Recall that
  $J$
  is locally Lipschitz continuous\footnote{
    Since every finite convex function is locally Lipschitz continuous~\cite[Exm. 9.14]{Rockafellar-Wets98},
    $J+\frac{\eta}{2}\norm{\cdot}^{2}$
    and
    $\frac{\eta}{2}\norm{\cdot}^{2}$
    are locally Lipschitz continuous due to the weak convexity of
    $J$.
    Then,
    $J = (J+\frac{\eta}{2}\norm{\cdot}^{2}) + (-\frac{\eta}{2}\norm{\cdot}^{2})$
    is locally Lipschitz continuous because the sum of locally Lipschitz continuous functions is locally Lipschitz continuous~\cite[Exe. 9.8]{Rockafellar-Wets98}.
  }, and thus
  $\Lsubdiff  J(\bm{z}) \neq \emptyset$
  for all
  $\bm{z} \in \mathcal{Z}$,
  by Fact~\ref{fact:strictly_continuous},
  i.e.,
  $\dom{\Lsubdiff  J} = \mathcal{Z}$.
  From
  $\dom{\Lsubdiff  J} = \mathcal{Z}$,
  $\lim_{n\to\infty} \mathcal{F}(\bm{y}_{n}) = \mathcal{F}(\widebar{\bm{y}})$,
  and
  $\mu_{n} \searrow 0$,
  Fact~\ref{fact:moreau}~\ref{enum:prox_convergence} implies
  $\lim_{n\to\infty}\widebar{\bm{z}}_{n} = \lim_{n\to\infty}\prox{\mu_{n}J}(\mathcal{F}(\bm{y}_{n})) = P_{\overline{\dom{\Lsubdiff J}}}(\mathcal{F}(\widebar{\bm{y}})) = \mathcal{F}(\widebar{\bm{y}}) \eqqcolon \widebar{\bm{z}}$.
  From the local Lipschitz continuity of
  $J$,
  Lemma~\ref{lemma:outer_continuous_image}
  with
  $\lim_{n\to\infty}\widebar{\bm{z}}_{n} = \widebar{\bm{z}}$
  and
  $\lim_{n\to\infty}A_{n} = A$
  yields
  \begin{align}
     & \Limsup_{n\to\infty}(\mathrm{D}\mathcal{F}(\bm{y}_{n}))^{*}[\Lsubdiff  J(\prox{\mu_{n}J}(\mathcal{F}(\bm{y}_{n})))]
    = \Limsup_{n\to\infty} A_{n}(\Lsubdiff  J(\widebar{\bm{z}}_{n}))
    \overset{\eqref{eq:subdifferential_inclusion}}{\subset} A(\Lsubdiff  J(\widebar{\bm{z}}))                              \\
     & = (\mathrm{D}\mathcal{F}(\widebar{\bm{y}}))^{*}[\Lsubdiff  J(\mathcal{F}(\widebar{\bm{y}}))]
    \overset{\eqref{eq:basic_chain}}{=} \Lsubdiff  (J\circ \mathcal{F})(\widebar{\bm{y}}),
  \end{align}
  from which
  $\bm{v} \in \Lsubdiff  (J\circ \mathcal{F})(\widebar{\bm{y}})$
  holds by~\eqref{eq:v_outer_limit}.

  (b)
  Theorem~\ref{theorem:stationary}~\ref{enum:gradient_consistency} with
  $J\coloneqq g$
  and
  $\mathcal{F}\coloneqq \mathfrak{S}\circ \param$
  yields
  $\Lsubdiff  (g\circ \mathfrak{S}\circ \param )(\widebar{\bm{y}}) \supset \Limsup\limits_{\mathbb{R}_{++}\ni \mu\searrow 0,\ \mathcal{Y}\ni \bm{y}\to\widebar{\bm{y}}} \nabla (\moreau{g}{\mu}\circ \mathfrak{S}\circ \param )(\bm{y})$.
  The inclusion in~\eqref{eq:gradient_consistency_problem} is obtained by
    {
      \thickmuskip=0.5\thickmuskip
      \medmuskip=0.5\medmuskip
      \thinmuskip=0.5\thinmuskip
      \arraycolsep=0.5\arraycolsep
      \begin{align}
         & \Lsubdiff  ((h+g\circ\mathfrak{S})\circ \param )(\widebar{\bm{y}})
        = \Lsubdiff  (h\circ \param  + g\circ \mathfrak{S}\circ \param )(\widebar{\bm{y}})
        \overset{\eqref{eq:sum_smooth_subgradient}}{=}\nabla (h\circ \param )(\widebar{\bm{y}})+ \Lsubdiff  (g\circ\mathfrak{S}\circ \param )(\widebar{\bm{y}})                                                      \\
         & \supset \nabla (h\circ \param )(\widebar{\bm{y}}) + \Limsup_{\mathbb{R}_{++}\ni \mu\searrow 0,\ \mathcal{Y}\ni \bm{y}\to\widebar{\bm{y}}} \nabla (\moreau{g}{\mu}\circ \mathfrak{S}\circ \param )(\bm{y})
        = \Limsup_{\mathbb{R}_{++} \ni \mu\searrow 0,\ \mathcal{Y}\ni \bm{y}\to\widebar{\bm{y}}} \nabla ((h+\moreau{g}{\mu}\circ\mathfrak{S})\circ \param )(\bm{y}), \hspace{-6em}
      \end{align}}%
  where the last equality follows by the continuity\footnote{
    Indeed, the continuity of
    $\nabla (h\circ \param )$
    yields the following relation:
    \begin{align}
       & \bm{v} \in \nabla (h\circ \param )(\widebar{\bm{y}}) +
      \Limsup_{\mathbb{R}_{++}\ni \mu\searrow 0,\ \mathcal{Y}\ni \bm{y}\to\widebar{\bm{y}}} \nabla (\moreau{g}{\mu}\circ \mathfrak{S}\circ \param )(\bm{y})                                                    \\
       & \Leftrightarrow
      \exists (\bm{y}_{n})_{n=1}^{\infty} \subset \mathcal{Y},\ \exists (\mu_{n})_{n=1}^{\infty} \subset (0,\eta^{-1})\ \mathrm{such\ that}\ \widebar{\bm{y}}=\lim_{n\to\infty}\bm{y}_{n},\ \mu_{n}\searrow 0, \\
       & \hspace{2em} \mathrm{and}\
      \bm{v}=
      \nabla (h\circ \param )(\widebar{\bm{y}}) + \lim_{n\to\infty}\nabla (\moreau{g}{\mu_{n}}\circ \mathfrak{S}\circ \param )(\bm{y}_{n})
      =\lim_{n\to\infty}\left(\nabla (h\circ \param )(\bm{y}_{n}) + \nabla (\moreau{g}{\mu_{n}}\circ \mathfrak{S}\circ \param )(\bm{y}_{n})\right)                                                             \\
       & \Leftrightarrow
      \bm{v} \in \Limsup_{\mathbb{R}_{++}\ni \mu\searrow 0,\ \mathcal{Y}\ni \bm{y}\to\widebar{\bm{y}}} (\nabla (h\circ \param )(\bm{y}) + \nabla (\moreau{g}{\mu}\circ \mathfrak{S}\circ \param )(\bm{y}))
      = \Limsup_{\mathbb{R}_{++}\ni \mu\searrow 0,\ \mathcal{Y}\ni \bm{y}\to\widebar{\bm{y}}} \nabla ((h+\moreau{g}{\mu}\circ \mathfrak{S})\circ \param )(\bm{y}).
    \end{align}
  }
  of
  $\nabla (h\circ \param )$.
  Moreover, we have
  \begin{equation}
    \thickmuskip=0.1\thickmuskip
    \medmuskip=0.1\medmuskip
    \thinmuskip=0.1\thinmuskip
    \arraycolsep=0.1\arraycolsep
    \Lsubdiff  ((h+g\circ \mathfrak{S})\circ \param )(\widebar{\bm{y}})
    \overset{\eqref{eq:gradient_consistency_problem}}{\supset} \Limsup_{\mathbb{R}_{++} \ni \mu\searrow 0,\ \mathcal{Y}\ni \bm{y}\to\widebar{\bm{y}}} \nabla ((h+\moreau{g}{\mu}\circ \mathfrak{S})\circ \param )(\bm{y})
    \overset{\eqref{eq:outer_limit_another}}{\supset} \Limsup_{n\to\infty} \nabla ((h+\moreau{g}{\mu_{n}}\circ \mathfrak{S})\circ \param )(\bm{y}_{n}). \label{eq:inclusion_subdifferential}
    \hspace{-6em}
  \end{equation}
  Then, the inequality in~\eqref{eq:necessary_bound} is obtained by
  \begin{align}
     & d(\bm{0},\Lsubdiff  ((h+g\circ \mathfrak{S})\circ \param )(\widebar{\bm{y}}))
    \overset{\eqref{eq:inclusion_subdifferential}}{\leq} d\left(\bm{0},\Limsup_{n\to\infty} \nabla ((h+\moreau{g}{\mu_{n}}\circ \mathfrak{S})\circ \param )(\bm{y}_{n})\right) \\
     & = \liminf_{n\to\infty} d(\bm{0}, \nabla ((h+\moreau{g}{\mu_{n}}\circ \mathfrak{S})\circ \param )(\bm{y}_{n}))
    = \liminf_{n\to\infty}\norm{\nabla ((h+\moreau{g}{\mu_{n}}\circ \mathfrak{S})\circ \param )(\bm{y}_{n})},
  \end{align}
  where the first equality follows by the fact
  $\liminf_{n\to\infty} d(\bm{v},E_{n}) = d\left(\bm{v},\Limsup_{n\to\infty}E_{n}\right)$
  for
  $\bm{v} \in \mathcal{Y}$
  and a given sequence
  $(E_{n})_{n=1}^{\infty}$
  of subsets
  $E_{n} \subset \mathcal{Y}$ (see~\cite[Exe. 4.8]{Rockafellar-Wets98}).
\end{proof}

\section{Proof of Corollary~\ref{corollary:Cayley}} \label{appendix:Cayley}
To prove Corollary~\ref{corollary:Cayley}, we need the following fact and lemma.
\begin{fact}[{\cite[Lemma A.4 and proof in Proposition 2.9]{Kume-Yamada22}}]\label{fact:matrix}
  \
  \begin{enumerate}[label=(\alph*)]
    \item \label{enum:norm_upper}
          For
          $\bm{A} \in \mathbb{R}^{l\times m}$
          and
          $\bm{B} \in \mathbb{R}^{m\times n}$,
          $\|\bm{A}\bm{B}\|_{F} \leq \|\bm{A}\|_{2} \|\bm{B}\|_{F}$
          and
          $\|\bm{A}\bm{B}\|_{F} \leq \|\bm{A}\|_{F} \|\bm{B}\|_{2}$
          hold, where
          $\norm{\cdot}_{F}$
          and
          $\norm{\cdot}_{2}$
          denote the Frobenius norm and the spectral norm.
    \item \label{enum:IV_inv_norm}
          For
          $\bm{V} \in Q_{N,p}$,
          we have
          $\|(\bm{I}+\bm{V})^{-1}\|_{2} \leq 1$.
    \item \label{enum:norm_Lipschitz}
          For
          $\bm{V}_{1},\bm{V}_{2} \in Q_{N,p}$,
          $\left\| (\bm{I}+\bm{V}_1)^{-1} - (\bm{I}+\bm{V}_2)^{-1} \right\|_{F} \leq \|\bm{V}_1 - \bm{V}_2 \|_{F}$.
    \item  \label{enum:differential_Cayley}
          For
          $\bm{V},\bm{D}\in Q_{N,p}$,
          $\mathrm{D}\param_{\bm{S}}^{\rm Cay}(\bm{V})[\bm{D}] = -2\bm{S}(\bm{I}+\bm{V})^{-1}\bm{D}(\bm{I}+\bm{V})^{-1}\bm{I}_{N\times p}$
  \end{enumerate}
\end{fact}

\begin{lemma}[Boundedness and Lipschitz continuity of $\mathrm{D}\param_{\bm{S}}^{\rm Cay}(\bm{V})$]
  \label{lemma:Cayley_Lipschitz}
  For
  $\param_{\bm{S}}^{\rm Cay}$
  in~\eqref{eq:Cayley} with
  $\bm{S} \in {\rm O}(N)$,
  we have
  \begin{align}
    (\bm{V} \in Q_{N,p}) \quad &
    \norm{\mathrm{D}\param_{\bm{S}}^{\rm Cay}(\bm{V})}_{\rm op} \leq 2; \label{eq:Cayley_1}                                                        \\
    (\bm{V}_{1},\bm{V}_{2} \in Q_{N,p}) \quad
                               & \norm{\mathrm{D}\param_{\bm{S}}^{\rm Cay}(\bm{V}_{1}) - \mathrm{D}\param_{\bm{S}}^{\rm Cay}(\bm{V}_{2})}_{\rm op}
    \leq 4\norm{\bm{V}_{1} - \bm{V}_{2}}_{F} \label{eq:Cayley_2}.
  \end{align}
\end{lemma}
\begin{proof}
  (Proof of~\eqref{eq:Cayley_1})
  Let
  $\bm{V}, \bm{D} \in Q_{N,p}$
  such that
  $\norm{\bm{D}}_{F} \leq 1$.
  By Fact~\ref{fact:matrix}~\ref{enum:differential_Cayley}, we have
  \begin{align}
     & \norm{\mathrm{D}\param_{\bm{S}}^{\rm Cay}(\bm{V})[\bm{D}]}_{F}
    = 2\norm{\bm{S}(\bm{I}+\bm{V})^{-1}\bm{D}(\bm{I}+\bm{V})^{-1}\bm{I}_{N\times p}}_{F}                                                                                     \\
     & \overset{\rm Fact~\ref{fact:matrix}~\ref{enum:norm_upper}}{\leq} 2 \norm{\bm{S}}_{2}\norm{(\bm{I}+\bm{V})^{-1}}_{2}^{2}\norm{\bm{D}}_{F}\norm{\bm{I}_{N\times p}}_{2}
    \overset{\rm Fact~\ref{fact:matrix}~\ref{enum:IV_inv_norm}}{\leq} 2\norm{\bm{D}}_{F} \leq 2,
  \end{align}
  implying thus~\eqref{eq:Cayley_1}, where we used
  $\norm{\bm{S}}_{2} = \norm{\bm{I}_{N\times p}}_{2} = 1$.

  (Proof of~\eqref{eq:Cayley_2})
  Let
  $\bm{V}_{1}, \bm{V}_{2}, \bm{D} \in Q_{N,p}$
  such that
  $\norm{\bm{D}}_{F} \leq 1$.
  By Fact~\ref{fact:matrix}~\ref{enum:differential_Cayley}, we have
  \begin{align}
     & \norm{\mathrm{D}\param_{\bm{S}}^{\rm Cay}(\bm{V}_{1})[\bm{D}] - \mathrm{D}\param_{\bm{S}}^{\rm Cay}(\bm{V}_{2})[\bm{D}]}_{F}                                                                                                               \\
     & = 2\norm{\bm{S}\left((\bm{I}+\bm{V}_{1})^{-1}\bm{D}(\bm{I}+\bm{V}_{1})^{-1}- (\bm{I}+\bm{V}_{2})^{-1}\bm{D}(\bm{I}+\bm{V}_{2})^{-1}\right)\bm{I}_{N\times p}}_{F}                                                                          \\
     & \overset{\rm Fact~\ref{fact:matrix}~\ref{enum:norm_upper}}{\leq}  2\norm{\bm{S}}_{2}\norm{(\bm{I}+\bm{V}_{1})^{-1}\bm{D}(\bm{I}+\bm{V}_{1})^{-1}- (\bm{I}+\bm{V}_{2})^{-1}\bm{D}(\bm{I}+\bm{V}_{2})^{-1}}_{F}\norm{\bm{I}_{N\times p}}_{2} \\
     & \leq 2 \norm{(\bm{I}+\bm{V}_{1})^{-1}\bm{D}\left((\bm{I}+\bm{V}_{1})^{-1}-(\bm{I}+\bm{V}_{2})^{-1}\right)}_{F}                                                                                                                             \\
     & \quad + 2\norm{\left((\bm{I}+\bm{V}_{1})^{-1}-(\bm{I}+\bm{V}_{2})^{-1}\right)\bm{D}(\bm{I}+\bm{V}_{2})^{-1}}_{F}. \label{eq:Cayley_Lipschitz}
    \quad (\because \norm{\bm{S}}_{2}=\norm{\bm{I}_{N\times p}}_{2} = 1)
  \end{align}
  The first term in~\eqref{eq:Cayley_Lipschitz} can be evaluated further as
  \begin{align}
     & \norm{(\bm{I}+\bm{V}_{1})^{-1}\bm{D}\left((\bm{I}+\bm{V}_{1})^{-1}-(\bm{I}+\bm{V}_{2})^{-1}\right)}_{F}                                                                           \\
     & \overset{\rm Fact~\ref{fact:matrix}~\ref{enum:norm_upper}}{\leq} \norm{(\bm{I}+\bm{V}_{1})^{-1}}_{2}\norm{\bm{D}}_{F}\norm{(\bm{I}+\bm{V}_{1})^{-1}-(\bm{I}+\bm{V}_{2})^{-1}}_{F} \\
     & \overset{\rm Fact~\ref{fact:matrix}~\ref{enum:IV_inv_norm}}{\leq} \norm{\bm{D}}_{F}\norm{(\bm{I}+\bm{V}_{1})^{-1}-(\bm{I}+\bm{V}_{2})^{-1}}_{F}                                   \\
     & \overset{\rm Fact~\ref{fact:matrix}~\ref{enum:norm_Lipschitz}}{\leq} \norm{\bm{D}}_{F}\norm{\bm{V}_{1}-\bm{V}_{2}}_{F}
    \leq \norm{\bm{V}_{1}-\bm{V}_{2}}_{F}. \label{eq:triangle_1}
  \end{align}
  Since the second term in~\eqref{eq:Cayley_Lipschitz} can be evaluated similarly to~\eqref{eq:triangle_1}, we have
  $\norm{\mathrm{D}\param_{\bm{S}}^{\rm Cay}(\bm{V}_{1})[\bm{D}] - \mathrm{D}\param_{\bm{S}}^{\rm Cay}(\bm{V}_{2})[\bm{D}]}_{F} \leq 4\norm{\bm{V}_{1}-\bm{V}_{2}}_{F}$,
  implying thus~\eqref{eq:Cayley_2}.
\end{proof}

\begin{proof}[Proof of Corollary~\ref{corollary:Cayley}]
  (a)
  By noting that
  $\param_{\bm{S}}^{\rm Cay}:Q_{N,p}\to \St(p,N)\setminus E_{N,p}(\bm{S})$
  is a diffeomorphism~\cite[Prop. 2.2]{Kume-Yamada22},
  Corollary~\ref{corollary:optimality_C_F} with
  $C\coloneqq \St(p,N)$
  and
  $\param \coloneqq \param_{\bm{S}}^{\rm Cay}$ completes the proof\footnote{
    Although
    $\param_{\bm{S}}^{\rm Cay}$
    is not surjective onto
    $\St(p,N)$,
    the same statement in Corollary~\ref{corollary:optimality_C_F} follows (see Remark~\ref{remark:not_surjective}).
  } (see also the footnote\footref{foot:submersion} in Corollary~\ref{corollary:optimality_C_F}).

  (b)
  To invoke Proposition~\ref{proposition:extension:Lipschitz_f}, we check below the conditions~\ref{enum:bounded_DG}-\ref{enum:Lipschitz_DF} in Proposition~\ref{proposition:extension:Lipschitz_f}
  with
  $C\coloneqq \St(p,N)$
  and
  $\param \coloneqq \param_{\bm{S}}^{\rm Cay}$.
  Since
  $\St(p,N)$
  is compact and
  $\mathfrak{S}$
  is twice continuously differentiable,~\ref{enum:bounded_DG}-\ref{enum:bounded_gradh} are automatically satisfied (see Remark~\ref{remark:assumption}~\ref{enum:S}).
  For~\ref{enum:bounded_DF}-\ref{enum:Lipschitz_DF}, see Lemma~\ref{lemma:Cayley_Lipschitz}.
  Therefore, by Proposition~\ref{proposition:extension:Lipschitz_f}~\ref{enum:Lipschitz}, Assumption~\ref{assumption:Lipschitz}~\ref{enum:gradient_Lipschitz} is satisfied with
  $C\coloneqq \St(p,N)$
  and
  $\param \coloneqq \param_{\bm{S}}^{\rm Cay}$.
\end{proof}

\end{document}

%% file: preamble.tex
\usepackage{natbib}

\usepackage{graphicx}%
\usepackage{multirow}%
\usepackage{amsmath,amssymb,amsfonts}%
\usepackage{amsthm}%
\usepackage{mathrsfs}%
\usepackage[title]{appendix}%
\usepackage{xcolor}%
\usepackage{textcomp}%
\usepackage{manyfoot}%
\usepackage{booktabs}%
\usepackage{algorithm}%
\usepackage{algorithmicx}%
\usepackage{algpseudocode}%
\usepackage{listings}%

\usepackage{epstopdf}%
\usepackage[caption=false]{subfig}%

\usepackage{here}
\usepackage{bm}
\usepackage{algorithm}
\usepackage{algpseudocode}
\usepackage{cases}
\usepackage{mathtools}
\usepackage{enumitem}
\usepackage{mathabx}
\usepackage{hyperref}
\hypersetup{
    colorlinks=true,
    citecolor=blue,
    linkcolor=red,
    urlcolor=orange,
}
\usepackage{stmaryrd}
\usepackage{csvsimple}

\usepackage{etoolbox}
\apptocmd{\liminf}{\limits}{}{}
\apptocmd{\limsup}{\limits}{}{}

\usepackage{tabularx}
\usepackage{booktabs}

\newcommand{\argmin}{\mathop{\mathrm{argmin}}\limits}

\newcommand{\inprod}[2]{{\left\langle #1,#2 \right\rangle}}
\newcommand{\trace}{{\rm Tr}}

\newcommand{\St}{{\rm St}}

\newcommand{\TT}{\mathsf{T}}
\newcommand{\diag}{\mathrm{diag}}

\newcommand{\prox}[1]{\mathrm{prox}_{#1}}
\newcommand{\moreau}[2]{{}^{#2}#1}

\newcommand{\exR}{\mathbb{R}\cup \{+\infty\}}
\newcommand{\dom}[1]{\mathrm{dom}(#1)}
\newcommand{\epi}[1]{\mathrm{epi}(#1)}
\newcommand{\Limsup}{\mathop{\mathrm{Lim~sup}}}

\newcommand{\norm}[1]{\left\lVert #1 \right\rVert}
\newcommand{\Id}{\mathrm{Id}}

\newcommand{\Fsubdiff}{\partial_{\rm F}}
\newcommand{\Lsubdiff}{\partial_{\rm L}}

\newcommand{\subdiff}{\partial}

\newcommand{\dist}{d}

\makeatletter
\csvset{
  autotabularcenter/.style={
    file=#1,
    after head=\csv@pretable\begin{tabular}{|*{\csv@columncount}{c|}}\csv@tablehead,
    table head=\hline\csvlinetotablerow\\\hline,
    late after line=\\,
    table foot=\\\hline,
    late after last line=\csv@tablefoot\end{tabular}\csv@posttable,
    command=\csvlinetotablerow},
  autobooktabularcenter/.style={
    file=#1,
    after head=\csv@pretable\begin{tabular}{*{\csv@columncount}{c}}\csv@tablehead,
    table head=\toprule\csvlinetotablerow\\\midrule,
    late after line=\\,
    table foot=\\\bottomrule,
    late after last line=\csv@tablefoot\end{tabular}\csv@posttable,
    command=\csvlinetotablerow},
}
\makeatother

\makeatletter
\newcommand{\doublewidetilde}[1]{{%
  \mathpalette\double@widetilde{#1}%
}}
\newcommand{\double@widetilde}[2]{%
  \sbox\z@{$\m@th#1\widetilde{#2}$}%
  \ht\z@=.9\ht\z@
  \widetilde{\box\z@}%
}
\makeatother

\setlist[trivlist]{topsep=0.5em plus 2pt minus 2pt} %

\theoremstyle{plain}%
\newtheorem{theorem}{Theorem}[section]
\newtheorem{lemma}[theorem]{Lemma}
\newtheorem{corollary}[theorem]{Corollary}
\newtheorem{proposition}[theorem]{Proposition}

\theoremstyle{definition}
\newtheorem{definition}[theorem]{Definition}
\newtheorem{fact}[theorem]{Fact}
\newtheorem{example}[theorem]{Example}
\newtheorem{assumption}[theorem]{Assumption}
\newtheorem{problem}[theorem]{Problem}
\newtheorem{remark}[theorem]{Remark}

\theoremstyle{remark}

\DeclarePairedDelimiter{\abs}{\lvert}{\rvert}